\documentclass[12pt]{article}

\usepackage{verbatim}
\usepackage{csquotes}
\usepackage[a4paper, total={6in, 9in}]{geometry}

\usepackage{amsfonts}
\usepackage{amsmath}
\usepackage{amssymb}    
\usepackage{amsthm}     
\usepackage{calligra}	
\usepackage{cancel}     
\usepackage{dsfont}     
\usepackage{mathtools}  
\usepackage{mathrsfs}   
\usepackage{stmaryrd}   
\usepackage{thmtools}   
\usepackage{bm}   
\usepackage{dynkin-diagrams}
\usepackage{rank-2-roots}
\usepackage{longtable}
\usepackage{float}
\usepackage{enumerate}
\usepackage{graphicx}         
\graphicspath{{figures/}}
\usepackage{tikz}             
\usetikzlibrary{calc}
\usetikzlibrary{intersections}
\usetikzlibrary{decorations.markings}
\usetikzlibrary{arrows}
\usetikzlibrary{positioning}
\usepackage{tikz-cd}		  
\usepackage[all]{xy}

\usepackage{xspace}         
\usepackage{textcomp}       
\usepackage{multirow}       
\usepackage{tablefootnote}  
\usepackage[intoc, refpage]{nomencl}
\usepackage{multicol}
\makenomenclature

\renewcommand{\nompreamble}{\begin{multicols}{3}}
\renewcommand{\nompostamble}{\end{multicols}}

\usepackage{varioref}
\usepackage{hyperref}

\declaretheorem[style = plain, numberwithin = section]{theorem}
\declaretheorem[style = plain,      sibling = theorem]{corollary}
\declaretheorem[style = plain,      sibling = theorem]{lemma}
\declaretheorem[style = plain,      sibling = theorem]{proposition}

\declaretheorem[style = plain,      sibling = theorem]{conjecture}
\declaretheorem[style = definition, sibling = theorem]{definition}
\declaretheorem[style = definition, sibling = theorem]{example}

\declaretheorem[style = remark,     sibling = theorem]{remark}

\declaretheorem[style = plain,     numberwithin = theorem]{claim}

\DeclareMathOperator{\Spec}{Spec}				
\DeclareMathOperator{\rk}{rk}					
\DeclareMathOperator{\Hom}{Hom}					
\DeclareMathOperator{\IC}{IC}					
\DeclareMathOperator{\sheafHom}					
{
    \mathscr{H}\text{\kern -5.2pt {\calligra\large om}}\,
}
\DeclareMathOperator{\sgn}{sgn}				    
\DeclareMathOperator{\FT}{FT}				    

\DeclarePairedDelimiter{\set}{\lbrace}{\rbrace}        

\newcommand{\WF}{\mathrm{WF}}

\newcommand{\pr}{\mathrm{pr}}

\newcommand{\OO}{\mathbb{O}}

\newcommand \bfG{{\mathbf G}}
\newcommand \bfT{{\mathbf T}}
\newcommand \bfB{{\mathbf B}}
\newcommand \bfI{{\mathbf I}}
\newcommand \bfP{{\mathbf P}}
\newcommand \bfU{{\mathbf U}}
\newcommand \bfL{{\mathbf L}}
\newcommand \bfl{{\mathbf l}}
\newcommand \bfu{{\mathbf u}}
\newcommand \bfp{{\mathbf p}}
\newcommand \Gal{{\mathrm {Gal}}}
\newcommand \sfk{{\mathsf k}}
\newcommand \bark{{\bar{\mathsf k}}}
\newcommand \barF{{\overline{\mathbb F}}}

\newcommand \CX{{\mathcal X}}

\newcommand{\Z}{\mathbb{Z}}    						
\newcommand{\R}{\mathbb{R}}    						
\newcommand{\C}{\mathbb{C}}    						

\newcommand{\GL}{\mathbf{GL}}						
\newcommand{\SL}{\mathbf{SL}}						

\newcommand{\lalg}[1]{{\normalfont\mathfrak{#1}}}	
\newcommand{\lsl}{\lalg{sl}}						

\newcommand{\mf}[1]{\mathfrak{#1}}							
\newcommand{\ms}[1]{\mathsf{#1}}							
\newcommand{\mb}[1]{\mathbf{#1}}							
\DeclareMathOperator{\id}{id}								

\renewcommand{\tilde}{\widetilde}
\newcommand{\supp}{\text{supp}}						        
\newcommand{\Res}{\text{Res}}						        
\newcommand{\Ind}{\text{Ind}}						        
\newcommand{\Ad}{\text{Ad}}						            
\newcommand{\tor}{\text{tor}}						        

\title{The wavefront set over a maximal unramified field extension}
\author{Emile T. Okada}

\begin{document}

\setcounter{secnumdepth}{4}
\setcounter{tocdepth}{3}
\counterwithin{paragraph}{subsection}
\numberwithin{equation}{subsection}

\maketitle
\begin{abstract}
    Let $(\pi,X)$ be a depth-$0$ admissible smooth complex representation of a $p$-adic reductive group that splits over an unramified extension.
    In this paper we develop the theory necessary to study the wavefront set of $X$ over a maximal unramified field extension of the base $p$-adic field.
    In the final section we then apply these methods to compute the geometric wavefront set of spherical Arthur representations of split $p$-adic reductive groups.
    In this case we see how the wavefront set over a maximal unramified extension can be computed using perverse sheaves on the Langlands dual group.
\end{abstract} 

\tableofcontents

\section*{Introduction}
Let $k$ be a finite extension of $\mathbb Q_p$ with finite residue field $\mathbb F_q$ of sufficiently large characteristic, algebraic closure $\bar k$ and let $\mathbf G(k)$ be the $k$-points of a connected reductive group $\mathbf G$ defined over $k$ and split over an unramified extension of $k$.
Let $\mf g$ be the Lie algebra of $\bfG$, and $\mf g(k)$ be the $k$-points of $\mf g$.
For an admissible smooth irreducible representation $(\pi, X)$ of $\bfG(k)$, the wavefront set of $X$, denoted $\WF(X)$, is a harmonic analytic invariant of $X$ of fundamental importance.
Roughly speaking it measures the direction of the singularities of the character distribution $\Theta_X$ of $X$ near the identity.
More precisely: the Harish-Chandra--Howe local character expansion dictates that there exists an open neighbourhood $\mathcal V$ of the identity and coefficients $c_{\OO}(X)\in \C$, one for each $\OO$ in the collection, $\mathcal N_o(k)$, of nilpotent orbits of $\mf g(k)$, such that 
\begin{equation}
    \Theta_{\pi}(f) = \sum_{\OO \in \mathcal N_o(k)}c_{\OO}(X) \hat \mu_{\OO}(f\circ \exp), \quad \forall f\in C_c^\infty(\mathcal V)
\end{equation}
where $\hat \mu_\OO$ denotes the Fourier transform of the nilpotent orbital integral associated to $\OO\in \mathcal N_o(k)$.
The \textit{($p$-adic) wavefront set} is the set (not necessarily a singleton)
$$\WF(X) := \max \{\OO:c_{\OO}(X)\ne 0\} \subseteq \mathcal N_o(k)$$
where the maximum is taken with respect to the closure ordering on the $k$-rational nilpotent orbits.

Much can be said about $X$ from the wavefront set.
When $\WF(X)$ consists of regular nilpotent elements, a famed result of Rodier \cite{rodier} states that $X$ admits a Whittaker model.
If an orbit in $\WF(X)$ meets a Levi subgroup of $\bfG(k)$, then the work of Moeglin and Waldspurger \cite{waldmoeg} shows that $X$ cannot be supercuspidal.
From an analytic point of view the wavefront set controls the asymptotic growth of the space of vectors fixed by the Moy--Prasad filtration subgroups as you go further along the filtration \cite[Section 5.1]{barbaschmoy}.
The wavefront set is also expected to play an important role in the definition of Arthur packets for $p$-adic groups, making it relevant to the study of automorphic forms and the unitary dual.

This distinguishes the wavefront set as a particularly powerful invariant to study.
However, in practice it is notoriously difficult to compute.
Moeglin and Waldspurger \cite{waldmoeg} have calculated the wavefront set for irreducible smooth representations of $\GL_n$ and for irreducible subquotients of the regular principal series for split classical groups, but little is known in general.
A slightly coarser invariant, which one might hope to have more control over, is the \textit{geometric wavefront set} $^{\bar k}WF(X)$.
This is defined to be 
$$^{\bar k}\WF(X) := \max\{\mathcal N_o(\bar k/k)(\OO):c_{\OO}(X)\ne 0\} \subseteq \mathcal N_o(\bar k)$$
where $\mathcal N_o(\bar k)$ denotes that set of $\bar k$-rational nilpotent orbits and $\mathcal N_o(\bar k/k)(\OO)$ denotes the (unique) $\bar k$-rational nilpotent orbit of $\mf g(\bar k)$ that $\OO$ lies in (see page 9 for the general definition of $\mathcal N_o$).
Much more is known about the geometric wavefront set. 
For classical groups, Moeglin \cite{moeglin} showed that it must always be a special orbit, and in \cite{wso3} and \cite{tempunip}, Waldspurger computed $^{\bar k}\WF(X)$ for anti-tempered and tempered unipotent representations of the pure inner twists of the split form of $SO(2n+1)$.
Moreover, in analogy with real reductive groups and finite groups of Lie type, it is expected that there is a \textit{single} nilpotent orbit in the geometric wavefront set and that all the nilpotent orbits in $\WF(X)$ lie in it.
In this sense, $^{\bar k}\WF(X)$ is a good first approximation for $\WF(X)$.

In the first part of this paper we shall concern ourselves with studying the wavefront set over an intermediate field. 
This will result in two new invariants which provide richer information than the geometric wavefront set, but are still more tractable than the $p$-adic wavefront set.
Let $K$ be the maximal unramified extension of $k$ in $\bar k$.
The residue field of $K$ is naturally an algebraic closure of $\mathbb F_q$ so we will write $\barF_q$ for the residue field of $K$.
Define the \textit{unramified wavefront set} to be
$$^K\tilde\WF(X) := \max\{\mathcal N_o(K/k)(\OO):c_{\OO}(X)\ne 0\} \subseteq \mathcal N_o(K)$$
where $\mathcal N_o(K)$ denotes the set of $K$-rational nilpotent orbits (which we henceforth refer to as \textit{unramified nilpotent orbits}) and $\mathcal N_o(K/k)(\OO)$ denotes the $K$-rational nilpotent orbit of $\mf g(K)$ that $\OO$ lies in.
The motivation for this modification comes from the work of Barbasch and Moy \cite{barbaschmoy} where they relate the coefficients $c_{\OO}(X)$ to representations of the reductive quotients of the parahoric subgroups of $\bfG(k)$.
The introduction of the field $K$ is inevitable if one wishes to lift geometric results from the reductive quotients up to the $p$-adic group, and much trouble is taken in \cite{barbaschmoy} to interpret these results over the original field $k$ again. 
What we show in this paper is that one should accept the field $K$ as a fact of life - and once one does this many statements take a more natural form and a lot of new structure becomes apparent.

The first main result of this paper - which we now state - illustrates this well.
Let $\mathcal B(\bfG,k)$ denote the Bruhat--Tits building of $\bfG(k)$.
For each face $c$ of $\mathcal B(\bfG,k)$ recall that we have the following short exact sequence
$$1\to \bfU_c(\mf o) \to \bfP_c(\mf o) \to \bfL_c(\mathbb F_q) \to 1.$$
The group $\bfL_c(\mathbb F_q)$ is a finite group of Lie type, and the space of fixed vectors $X^{\bfU_c(\mf o)}$ is a finite dimensional representation of $\bfL_c(\mathbb F_q)$.
As we alluded to earlier, the wavefront set is an invariant which also makes sense for representations of finite groups of Lie type and we write 
$$^{\barF_q}\WF(X^{\bfU_c(\mf o)})$$
for the (geometric) wavefront set of the representation $X^{\bfU_c(\mf o)}$ of $\bfL_c(\mathbb F_q)$.
This is a collection of nilpotent orbits of $\bfL_c$ over $\barF_q$ (see Section \ref{sec:kawanaka} for precise details).
Motivated by the work of \cite{barbaschmoy}, \cite{debacker} and \cite{waldsnil} we introduce a lifting map $\mathcal L_c$ from the partially ordered set of $\barF_q$-rational nilpotent orbits of $\bfL_c$ to the partially ordered set of $K$-rational nilpotent orbits of $\bfG(k)$.
The wavefront sets of the representations of the reductive quotients are then related to the wavefront set of $X$ by the following theorem.
\begin{theorem}
    \label{thm:localwf}
    [Theorem \ref{lem:liftwf}]
    Let $(\pi,X)$ be a depth-$0$ representation of $\bfG(k)$.
    Then 
    \begin{equation}
        ^K\widetilde\WF(X) = \max_{c\subseteq \mathcal B(\bfG,k)}\mathcal L_c(\hphantom{ }^{\barF_q}\WF(X^{\bfU_c(\mf o)})).
    \end{equation}
    In fact one can restrict $c$ to range over the faces (or vertices) of any fixed chamber of $\mathcal B(\bfG,k)$.
\end{theorem}
This easily yields the following corollary.
\begin{corollary}
    \label{cor:geomwf}
    [Proposition \ref{prop:wfs}, Theorem \ref{lem:liftwf}]
    Let $(\pi,X)$ be a depth-$0$ representation of $\bfG(k)$.
    Then 
    \begin{equation}
        ^{\bar k}\WF(X) = \max_{c\subseteq c_0} \mathcal N_o(\bar k/K) (\mathcal L_c(\hphantom{ }^{\barF_q}\WF(X^{\bfU_c(\mf o)})))
    \end{equation}
    where $c_0$ is any chamber of $\mathcal B(\bfG,k)$ and $\mathcal N_o(\bar k/k)$ is defined on page 9.
\end{corollary}
This corollary cleanly repackages the main idea of \cite[Section 5.1]{barbaschmoy}  (indeed it makes precise \cite[Proposition 5.2]{barbaschmoy}) and in a sense this alone would be a satisfactory conclusion to our foray into unramified territory.
However from a philosophical perspective Theorem \ref{thm:localwf} suggests that the unramified wavefront set is natural in its own right and warrants further investigation. 

The next section of this paper investigates the unramified wavefront set closer.
The main practical obstruction to doing this is the partially ordered set $\mathcal N_o(K)$ for which very little is known.
The main result of this section is a natural parameterisation of this set which depends equivariantly on the choice of hyperspecial point.
To make this precise, let us fix some notation.
Let $\bfG_K$ denote the base change of $\bfG$ along $\Spec(K)\to \Spec(k)$ (recall by assumption that $\bfG_K$ is split).
Let $\mathcal B(\bfG,K)$ be the Bruhat-Tits building for $\bfG_K(K)=\bfG(K)$ and let $\bfT_K$ be a maximal $K$-split torus of $\bfG_K$.
A hyperspecial face of $\mathcal B(\bfG,K)$ is a face that contains a hyperspecial point.
Let $\mathscr H$ denote the set of hyperspecial faces of $\mathcal B(\bfG,K)$ and $\bfG(K)\backslash\mathscr H$ denote the set of $\bfG(K)$-orbits of hyperspecial faces.
From the root data attached to $\bfT_K$ we may construct complex reductive groups $G$ (resp. $G^\vee$) with the same (resp. dual) root data.
Let $\mathcal N_{o,c}$ denote the set of all pairs $(\OO,C)$ where $\OO$ is a complex nilpotent orbit of $G$ and $C$ is a conjugacy class of $A(\OO)$ - the $G$-equivariant fundamental group of $\OO$.
Let $Z_G$ denote the center of $G$ and $A(Z_G)$ be the group of components of $Z_G$.
In Section \ref{par:equiv} we define an action of $A(Z_G)$ on $\mathcal N_{o,c}$ and a simply transitive action of $A(Z_G)$ on $\bfG(K)\backslash \mathscr H$.
\begin{theorem}
    \label{thm:unramparam}
    [Theorem \ref{thm:unramifiedparam},Theorem \ref{thm:naturality},Proposition \ref{prop:equivariance}]
    For each orbit $\mathscr O\in \bfG(K)\backslash \mathscr H$ there is a bijection 
    \begin{equation}
        \theta_{\mathscr O,\bfT_K}:\mathcal N_o(K) \xrightarrow{\sim} \mathcal N_{o,c}.
    \end{equation}
    This map is natural in $\bfT_K$ and $A(Z_G)$-equivariant in $\mathscr O$.
\end{theorem}
We first remark that the set $\mathcal N_{o,c}$ might look familiar to readers for its resemblance to the parameters arising in the Springer correspondence. 
In that setting one considers the set of pairs $(\OO,\rho)$ where $\OO$ is a complex nilpotent orbit of $G$ and $\rho$ is an \emph{irreducible representation} of $A(\OO)$.
The similarity is of course not precise, but certainly suggests some tantalising connections.
More than simply being suggestive however, the set $\mathcal N_{o,c}$ is well studied in its own right.
It naturally arises (non-canonically) as the parameterising set for nilpotent orbits of finite groups of Lie type and in Sommers' work generalising the Bala--Carter theorem for nilpotent orbits \cite{Sommers2001}.
It is also the domain for a powerful extension of the incredibly important Barbasch--Lusztig--Spaltenstein--Vogan duality map $d:\mathcal N_o\to \mathcal N_o^\vee$ going from complex nilpotent orbits of $G$ to complex nilpotent orbits of $G^\vee$.
This extension $d_S:\mathcal N_{o,c}\to \mathcal N_o^\vee$, discovered by Sommers, but also apparent in earlier work by Lusztig, extends the map $d$ in the sense that
$$d_S(\OO,1) = d(\OO)$$
where $1$ denotes the trivial conjugacy class, and is notable because in contrast to $d$, the map $d_S$ is surjective.
Fixing an orbit $\mathscr O\in \bfG(k)\backslash \mathscr H$ and using the bijection $\theta_{\mathscr O,\bfT_K}$ from Theorem \ref{thm:unramparam} one can of course interpret the map $d_S$ as a map 
$$d_{S,\mathscr O,\bfT_K}:\mathcal N_o(K) \to \mathcal N_o^\vee.$$ 
Although this map ostensibly depends on $\mathscr O$, we show in Proposition \ref{prop:indpd} that it is in fact \emph{independent} of the choice of $\mathscr O$.
Thus let us discard the $\mathscr O$ from the notation and simply write 
$$d_{S,\bfT_K}:\mathcal N_o(K)\to \mathcal N_o^\vee.$$
Having discarded with the dependence on $\mathscr O$, and given the naturality of $d_{S,\bfT_K}$ in $\bfT_K$, the map $d_{S,\bfT_K}$ certainly seems natural (in the colloquial sense of the word. Indeed one might wonder if it admits a more intrinsic construction than the one we have given).
This raises the obvious - and pertinent - question: is $d_S$ an order reversing map?
From a purely philosophical perspective the answer ought to be yes - duality maps \emph{should} be order reversing.
Indeed we provide some evidence to support this belief in Lemma \ref{lem:orderrev}.
However the partial order on $\mathcal N_o(K)$ is difficult to study and the bijection $\theta_{\mathscr O,\bfT_K}$ is not so well suited to give easy answers to this question.
The best we can do in this paper is leave this as a conjecture.
\begin{conjecture}
    \label{conj:orderrev}
    The map $d_{S,\bfT_K}:\mathcal N_o(K)\to \mathcal N_o^\vee$ is order reversing.
\end{conjecture}
But let us set aside this issue for the moment.
The work of Achar in \cite{achar} strongly suggests that there is merit in \emph{declaring} $d_{S,\bfT_K}$ to be order-reversing.
What we mean by this is we introduce a new order on $\mathcal N_o(K)$ where $d_S$ is order reversing by design.
For $\OO_1,\OO_2\in \mathcal N_o(K)$ define 
$$\OO_1\le_A\OO_2\quad  \text{ if } \quad \mathcal N_o(\bar k/K)(\OO_1)\le \mathcal N_o(\bar k/K)(\OO_2) \text{ and } d_{S,\bfT_K}(\OO_1)\ge d_{S,\bfT_K}(\OO_2).$$
This is the \emph{finest} pre-order on $\mathcal N_o(K)$ that makes $\mathcal N_o(\bar k/K)$ order preserving and $d_{S,\bfT_K}$ order reversing.
The naturality of $d_{S,\bfT_K}$ in $\bfT_K$ means that this pre-order is independent of the choice of maximal $K$-split torus $\bfT_K$.
It is not a partial order because there are orbits $\OO_1,\OO_2\in \mathcal N_o(K)$ such that $\OO_1\le_A \OO_2$ and $\OO_2\le_A \OO_1$ but $\OO_1 \ne \OO_2$.
However, if we define $\OO_1\sim\OO_2$ when $\OO_1\le_A\OO_2$ and $\OO_2\le_A\OO_1$ then $\le_A$ descends to a partial order on $\mathcal N_o(K)/\sim_A$ and we can also obtain a explicit parameterisation of this set.
Let $\mathcal N_{o,\bar c}$ denote the set of all pairs $(\OO,\bar C)$ of complex nilpotent orbits of $G$ and conjugacy classes $\bar C$ of $\bar A(\OO)$ - Lusztig's canonical quotient of $A(\OO)$ (in the sense of \cite[Section 5]{Sommers2001}).
Let $\mf Q:\mathcal N_{o,c}\to \mathcal N_{o,\bar c}$ be the map induced by the quotient map $A(\OO)\to \bar A(\OO)$ (see the discussion preceding Proposition \ref{prop:indpd} for details).
\begin{theorem}
    [Theorem \ref{thm:thetabar}]
    Let $\mathscr O\in \bfG(K)\backslash \mathscr H$.
    The composition $\mf Q\circ\theta_{\mathscr O,\bfT_K}$ descends to a natural (in $\bfT_K$) bijection 
    $$\bar\theta_{\bfT_K}:\mathcal N_o(K)/\sim_A \xrightarrow{\sim} \mathcal N_{o,\bar c}$$
    which does not depend on $\mathscr O$. 
\end{theorem}
Crucially for us, the partial order $\le_A$ on $\mathcal N_{o}(K)/\sim_A$ is considerably easier to compute in practice than the closure ordering on $\mathcal N_o(K)$.
This motivates our next definition.
We define the \emph{canonical unramified wavefront set} to be 
%
$$^K\WF(X):=\max_{c\subseteq \mathcal B(\bfG,k)}[\mathcal L_c(\hphantom{ }^{\barF_q}\WF(X^{\bfU_c(\mf o)}))] \quad (\subseteq \mathcal N_o(K)/\sim_A)$$
where $[\bullet]:\mathcal N_o(K)\to \mathcal N_o(K)/\sim_A$ is the natural quotient map.
The compatibility between $\le_A$ and $\mathcal N_o(\bar k/K)$ ensures that an analogue of Corollary \ref{cor:geomwf} (Theorem \ref{thm:algwf}) holds for $^K\WF(X)$, but $^K\WF(X)$ also has the added benefit that it is frequently (conjecturally always) a singleton.
When $^K\WF(X)$ is a singleton, if we view $^K\WF(X)$ as an element of $\mathcal N_{o,\bar c}$ (via $\bar \theta_{\bfT_K}$) and $^{\bar k}\WF(X)$ as an element of $\mathcal N_o$ (under the natural isomorphism between $\mathcal N_o(\bar k)$ and $\mathcal N_o$), then $^K\WF(X)$ takes the form
$$^K\WF(X) = (\hphantom{ }^{\bar k}\WF(X),\bar C)$$
for some $\bar C\in \bar A(\hphantom{ }^{\bar k}\WF(X))$.
For those familiar with representations of real reductive groups, this is reminiscent of the associated cycle of a representation, but of course differs crucially in that we are dealing with conjugacy classes rather than irreducible representations, and with $\bar A(\hphantom{ }^{\bar k}\WF(X))$ instead of $A(\hphantom{ }^{\bar k}\WF(X))$.

Beyond being the culmination of a series of `natural' considerations, it is not a priori clear what utility this new invariant has.
Although we will not address this matter in this paper, in joint work with Dan Ciubotaru and Lucas Mason-Brown \cite{cmo}, we use the canonical unramified wavefront set to construct anti-tempered Arthur packets for $p$-adic groups.
This approach crucially relies on the information encoded in the canonical unramified wavefront set and fails if one attempts to use the geometric wavefront set instead.

Let us now briefly digress to explain the terminology used.
In the language of partial orders, Conjecture \ref{conj:orderrev} is equivalent to $[\bullet]:\mathcal N_o(K)\to \mathcal N_o(K)/\sim_A$ being a homomorphism.
Under the assumption that this is true then
$$^K\WF(X) = \max\{[\mathcal N_o(K/k)(\OO)]:c_{\OO}(X)\ne 0\}$$
for depth-$0$ representations and so indeed $^K\WF(X)$ is a `wavefront set'.
If we further assume that all the orbits of $^K\tilde\WF(X)$ lie in a single geometric orbit, then $^K\WF(X)$ simply picks out the $\sim_A$ classes of elements in $^K\tilde\WF(X)$ that minimise $d_S$.
Conjecturally there is a unique such class - a `canonical' such class with respect to this property if you will.
Perhaps `distinguished' would have been a better modifier, but that adjective already has an important meaning in the context of nilpotent orbits.

The third and final section of this paper is dedicated to developing the tools needed to compute $^K\WF(X)$ for irreducible representations in the principal block of $\mathrm{Rep}(\bfG(k))$, the category of smooth complex $\bfG(k)$-representations, when $\bfG$ is split over $k$. 
Recall that the principal block of $\mathrm{Rep}(\bfG(k))$, which we denote $\mathrm{Rep}_{\bfI}(\bfG(k))$, consists of those representations that are generated by their Iwahori fixed vectors and is equivalent to the category of modules of the Iwahori--Hecke algebra $\mathcal H_{\bfI}$ of $\bfG(k)$.
Moreover, when $X$ is admissible, by the theory of unrefined minimal $K$-types, the representations $X^{\bfU_c(\mf o)}$ are sums of principal series unipotent representations.
The wavefront sets of such representations have a particularly simple expression connected to the Hecke algebra of $\bfL_c(\mathbb F_q)$.
We use the compatibility of these Hecke algebras with the Iwahori--Hecke algebra and a simple deformation argument to obtain an explicit algorithm in Theorem \ref{thm:locwf} for computing $^K\WF(X)$.

Finally, we use the tools developed in this section to compute $^K\WF(X)$ and $^{\bar k}\WF(X)$ for the spherical Arthur representations of a split adjoint group over $k$.
For $\mb G$ a reductive group defined and split over a number field $F$, these are expected to be the spherical representations arising as local factors of irreducible subrepresentations of $L^2_{disc}(\mb G(F)\backslash \mb G(\mathbb A_F))$ (see \cite{as1}, \cite{as2}, \cite{as3}, \cite{as4}, \cite{as5}, \cite{as6}, for proofs in various special cases. See \cite{as7} for a uniform proof that all the spherical Arthur representations arise in this way). 
Note that knowledge of the geometric wavefront set of spherical Arthur representations provides valuable structural insight for automorphic representations. 
In particular the geometric wavefront set of the local factors bound the Fourier coefficients of the automorphic form (see \cite{gomez2020whittaker}).

We now state our results for the non-archimedean spherical Arthur representations in terms of their Arthur parameters.
Let $\mb G$ be defined and split over the $p$-adic field $k$.
Let $G^\vee$ denote the compelx Langalnds dual group of $\bfG$, $W_k$ be the Weil group of $k$, and 
$$WD_k = W_k\times \SL(2,\C)$$
the Weil--Deligne group of $k$.
Let $(\pi,X)$ be the spherical Arthur representation of $\bfG(k)$ lying in the Arthur packet $\psi:WD_k\times\SL(2,\C)\to G^\vee$ that is trivial on $WD_k$.
Within this packet, $X$ is the representation corresponding to the trivial representation of $A_\psi$ - the component group of the centraliser of the image of $\psi$.
Let $\psi_0 = \psi\mid_{1\times\SL(2,\C)}$ and
\begin{equation}
    \ms n = d(\psi_0)\left(\begin{pmatrix}
        0 & 1 \\ 0 & 0
    \end{pmatrix}\right).
\end{equation}
The nilpotent orbit $\OO^\vee:= G^\vee. \ms n$ completely determines the representation $X$ among the spherical Arthur representations and so we refer to $X$ as the spherical Arthur representation with parameter $\OO^\vee$.
The final ingredient that we need in order to state the result is a refinement $d_A$ due to Achar \cite{achar} of the duality map $d_S$. 
Let $\mathcal N^\vee_{o,c},\mathcal N^\vee_{o,\bar c}$ be the sets $\mathcal N_{o,c},\mathcal N_{o,\bar c}$ defined realative to $G^\vee$ instead of $G$. The duality $d_A$ is a map
\[d_A: \mathcal N^\vee_{o,\bar c}\rightarrow \mathcal N_{o,\bar c}
\]
satisfying certain properties. In particular,
\[d_A(\OO^\vee,1)=(d(\OO^\vee),\bar C'),
\]
for some class $C'$ which is the trivial class when $\OO^\vee$ is special in the sense of Lusztig. 
Using the bijection $\bar\theta_{\bfT_K}$ we view $d_A$ as a map $\mathcal N_{o,\bar c}^\vee \to \mathcal N_o(K)/\sim_A$.
Similarly, using the bijection from Lemma \ref{lem:pom} we view $d$ as a map $\mathcal N_o^\vee \to \mathcal N_o(\bar k)$.

\begin{theorem}
    [Theorem \ref{thm:arthurwf}]
    Let $X$ be the spherical Arthur representation with parameter $\OO^\vee\in \mathcal N_o^\vee$.
    Then $^K\WF(X)$ is a singleton and 
    \begin{equation}
        \hphantom{ }^{K}\WF(X) = d_A(\OO^\vee,1),\quad \hphantom{ }^{\bar k}\WF(X) = d(\OO^\vee).
    \end{equation}
\end{theorem}

\subsection*{Acknowledgements}
I would like to thank Dan Ciubotaru and Kevin McGerty for their invaluable guidance and many helpful conversations.
I would also like to thank Konstantin Ardakov, Jessica Fintzen, Lucas Mason-Brown and Beth Romano for their careful proof reading of this paper, Maarten Solleveld for helpful comments and corrections on an earlier draft of this paper, Marcelo De Martino for teaching me about the Bruhat-Tits building, and Jonas Antor and Ruben La for many enlightening discussions on perverse sheaves.
The author was supported by the Aker Scholarship.

\section{The Wavefront Set}
\paragraph{Basic Notation}
\label{par:basicnotation1}
Let $k$ be a non-archimedean local field of characteristic $0$ with residue field $\mathbb{F}_q$ of sufficiently large characteristic and ring of integers $\mathfrak{o} \subset k$. 
\nomenclature{$k$}
\nomenclature{$\mf o$}{}
\nomenclature{$\mathbb F_q$}{}
\nomenclature{$\bar k$}{}
\nomenclature{$K$}{}
\nomenclature{$\mf O$}{}
\nomenclature{$\barF_q$}{}
\nomenclature{$\chi$}{}
Let $\mf p\subset \mf o$ be the maximal ideal of $\mf o$, fix an algebraic closure $\bar{k}$ of $k$ and let $K \subset \bar{k}$ be the maximal unramified extension of $k$ in $\bar{k}$. 
Let $\mf O$ be the ring of integers of $K$.
The residue field of $K$ is an algebraic closure for $\mathbb F_q$ so we write $\barF_q$ for the residue field of $K$.
Let $\chi:k\to \C^\times$ be an additive character of $k$ that is trivial on $\mf p$ and non-trivial on $\mf o$.
$\chi$ descends to a character of $\mathbb F_q$ and we will refer to the resulting character also as $\chi$.

\nomenclature{$\bfG$}{}
\nomenclature{$\mf g$}{}
\nomenclature{$\bfG_K$}{}
\nomenclature{$\bfT_K$}{}
\nomenclature{$\bfG_K$}{}
\nomenclature{$X^*(\bfT_K,\bar k)$}{}
\nomenclature{$X_*(\bfT_K,\bar k)$}{}
\nomenclature{$\Phi(\bfT_K,\bar k)$}{}
\nomenclature{$\Phi^\vee(\bfT_K,\bar k)$}{}
Let $\bfG$ be a connected reductive algebraic group defined over $k$, that splits over an unramified extension of $k$.
Let $\mf g$ denote its Lie algebra.
Let $\bfG_K$ denote the base change of $\bfG$ along $$\Spec(K)\to \Spec(k).$$
Note that $\bfG_K$ is a split group.
Let $\bfT_K \subset \mathbf{G}_K$ be a $K$-split maximal torus. 
For any field extension $F$ of $k$, we write $\bfG(F)$, $\mf g(F)$ etc. for the $F$ rational points.

\nomenclature{$\mathcal R(\bfG_K,\bfT_K)$}{}
Write $X^*(\bfT_K,\bar k)$ (resp. $X_*(\bfT_K,\bar k)$) for the lattice of algebraic characters (resp. co-characters) of $\mathbf{T}_K$, $\langle,\rangle$ for the canonical pairing between $X^*(\bfT_K,\bar k)$ and $X_*(\bfT_K,\bar k)$, and write $\Phi(\mathbf{T}_K,\bar k)$ (resp. $\Phi^{\vee}(\mathbf{T}_K,\bar k)$) for the set of roots (resp. co-roots) of $\bfG_K$. 
Write
$$\mathcal R(\bfG_K,\bfT_K)=(X^*(\mathbf{T}_K,\bar k), \ \Phi(\mathbf{T}_K,\bar k),X_*(\mathbf{T}_K,\bar k), \ \Phi^\vee(\mathbf{T}_K,\bar k), \ \langle \ , \ \rangle)$$
for the absolute root datum of $\mathbf{G}$, and let $W$ be the associated (finite) Weyl group. 
\nomenclature{$\bfG_\Z$}{}
\nomenclature{$\bfT_\Z$}{}
Let $\mathbf{G}_\Z$ be the connected reductive algebraic group defined (and split) over $\Z$ with split maximal torus $\bfT_\Z$ such that the root datum of $\bfG_\Z$ with respect to $\bfT_\Z$ is isomorphic to $\mathcal R$.
Let $\mathbf{G}_\Z^\vee$ be the Langlands dual group of $\bfG$, i.e. the connected reductive algebraic group corresponding to the root datum 
$$\mathcal R^\vee=(X_*(\mathbf{T}_K,\bar k), \ \Phi^{\vee}(\mathbf{T}_K,\bar k),  X^*(\mathbf{T}_K,\bar k), \ \Phi(\mathbf{T}_K,\bar k), \ \langle \ , \ \rangle)$$
defined (and split) over $\Z$.
\nomenclature{$\bfG^\vee_\Z$}{}
\nomenclature{$\bfT^\vee_\Z$}{}
Set $T^\vee=X^*(\bfT_K,\bar k)\otimes_\Z \C^\times$, regarded as a maximal torus in $G^\vee:=\bfG_\Z^\vee(\C)$ with Lie algebra $\mathbf{\mathfrak t}^\vee=X^*(\bfT_K,\bar k)\otimes_{\mathbb Z} \mathbb C$, a Cartan subalgebra of the Lie algebra $\mathbf{\mathfrak g}^\vee$ of $\bfG^\vee$. 
Define
\begin{align}\label{eq:real}
\begin{split}
    T^\vee_{\mathbb R} &=X^*(\bfT_K,\bark)\otimes_{\mathbb Z} {\mathbb R}_{>0}\\
    \mathbf{\mathfrak t}_{\mathbb R}^\vee &= X^*(\bfT_K,\bark)\otimes_{\mathbb Z} \mathbb R\\
    T^\vee_c &=X^*(\bfT_K,\bark)\otimes_{\mathbb Z} S^1.
\end{split}
\end{align}
There is a polar decomposition $T^\vee=T^\vee_c T ^\vee_{\mathbb R}$.
\nomenclature{$T^\vee$}{}
\nomenclature{$T_\R^\vee$}{}
\nomenclature{$T_c^\vee$}{}
\nomenclature{$\mf t^\vee$}{}
\nomenclature{$\mf t_\R^\vee$}{}

\nomenclature{$\mathrm{Field}_k$}{}
\nomenclature{$\mathcal N(\bullet)$}{}
\nomenclature{$\mathcal N_o(\bullet)$}{}
Let $\mathrm{Field}_k$ denote the category of field extensions of $k$.
Let $\mathcal N$ be the functor from $\mathrm{Field}_k$ to $\mathrm{Set}$ which takes a field extension $F$ of $k$ to the set of nilpotent elements of $\mf g(F)$.
By nilpotence in this context we mean the unstable points (in the sense of GIT) with respect to the adjoint action of $\bfG(F)$, see \cite[Section 2]{debacker}.
For $F$ algebraically closed this coincides with all the usual notions of nilpotence.
Let $\mathcal N_o$ be the functor which takes $F$ to the set of orbits in $\mathcal N(F)$ under the adjoint action of $\bfG(F)$.
For $\ms H\in \bfG(F)$ and $\ms x\in \mf g(F)$ we write $\ms H.\ms x$ for the adjoint action of $\ms H$ on $\ms x$.
We briefly remark how $\mathcal N_o$ behaves on morphisms.
Given field extensions $F_1,F_2\in\mathrm{Field}_k$ and a morphism $F_1\to F_2$ we have natural inclusion maps 
$$\mf g(F_1)\to \mf g(F_2) \text{ and } \bfG(F_1)\to \bfG(F_2).$$
Thus given a $\bfG(F_1)$ orbit $\OO\subset \mathcal N(F_1)$ we can form the orbit 
$$\bfG(F_2).\OO\subset \mathcal N(F_2).$$
We define $\mathcal N_o(F_1\to F_2)(\OO)$ to be this orbit.
When we wish to emphasis the group we are working with we include it as a superscript e.g. $\mathcal N_o^{\bfG_\Z}$.

When $F$ is algebraically closed, we view $\mathcal N_o(F)$ as a partially ordered set with respect to the closure ordering in the Zariski topology.
When $F$ is $k$ or $K$, we view $\mathcal N_o(F)$ as a pre-ordered set with respect to the closure ordering in the topology induced by the topology on $F$.
When $F=k$ it is well known that the pre-order is a partial order \cite[Section 2.5]{debacker_homog}.
When $F=K$ we will show in \ref{cor:partialorder} that the pre-order is a partial order.
For brevity we will write $\mathcal N(F'/F)$ (resp. $\mathcal N_o(F'/F)$) for $\mathcal N(F\to F')$ (resp. $\mathcal N_o(F\to F')$) where $F\to F'$ is a morphism of fields.

Recall the following classical result.
\begin{lemma}[\cite{Pommerening},\cite{Pommerening2}]\label{lem:pom}
    Let $F \in \mathrm{Field}_k$ be algebraically closed with good characteristic for $\bfG$.
    Then there is canonical isomorphism of partially ordered sets $\Lambda_{\bfT_K}^F:\mathcal N_o(F)\xrightarrow{\sim}\mathcal N_o^{\bfG_\Z}(\C)$.
    \nomenclature{$\Lambda_{\bfT_K}^F$}{}
\end{lemma}
\begin{remark}
    We include the $\bfT_K$ as a subscript because the definition of the group $\bfG_\Z$ depends on the choice of torus.
    However $\Lambda_{\bfT_K}^{F}$ is natural in $\bfT_K$ in an analogous sense to Theorem \ref{thm:naturality}.
\end{remark}

When $F$ is algebraically closed let $\mathcal N_{o,sp}(F)$ denote the set of special orbits in the sense of Lusztig \cite[Definition 13.1.1]{chars}.
\nomenclature{$\mathcal N_{o,sp}(\bullet)$}{}

\paragraph{Buildings, Parahorics and Associated Notation}
\label{sec:buildings}
\nomenclature{$\mathcal B(\bfG,k),\mathcal B(\bfG,K)$}{}
\nomenclature{$\mathcal A(\bfT,k),\mathcal A(\bfT,K)$}{}
\nomenclature{$\mathcal A(c,\mathcal A)$}{}
Let $\mathcal B(\bfG,k)$ (resp. $\mathcal B(\bfG,K)$) denote the (enlarged) Bruhat--Tits building for $\bfG(k)$ (resp. $\bfG(K)$).
We identify $\mathcal B(\bfG,k)$ with the $\Gal(K/k)$-fixed points of $\mathcal B(\bfG,K)$.
We use the notation  $c\subseteq \mathcal B(\bfG)$ to indicate that $c$ is a face of $\mathcal B$.
Given a maximal $k$-split torus $\bfT$, write $\mathcal A(\bfT,k)$ for the corresponding apartment in $\mathcal B(\bfG,k)$.
For an apartment $\mathcal A$ of $\mathcal B(\bfG,k)$ and any subset $\Omega\subseteq \mathcal A$ we write $\mathcal A(\Omega,\mathcal A)$ for the smallest affine subspace of $\mathcal A$ containing $\Omega$.
\nomenclature{$\bfP_c^\dagger$}{}
\nomenclature{$\bfP_c$}{}
\nomenclature{$\bfU_c$}{}
\nomenclature{$\bfL^\dagger_c$}{}
\nomenclature{$\bfL_c$}{}
\nomenclature{$\bfp_c$}{}
\nomenclature{$\bfu_c$}{}
\nomenclature{$\bfl_c$}{}
For a face $c\subseteq \mathcal B(\bfG,k)$ there is a group scheme $\bfP_c^\dagger$ defined over $\Spec(\mf o)$ such that $\bfP_c^\dagger(\mf o)$ identifies with the stabiliser of $c$ in $\bfG(k)$. There is an exact sequence \cite[Section 1.2]{reeder}
\begin{equation}
    1 \to \bfU_c(\mf o) \to  \bfP_c^\dagger(\mf o) \to  \bfL_c^\dagger(\mathbb F_q) \to 1,
\end{equation}
where $\bfU_c(\mf o)$ is the pro-unipotent radical of $\bfP_c^\dagger(\mf o)$ and $\bfL_c^\dagger$ is the reductive quotient of the special fibre of $\bfP_c^\dagger$.
Let $\bfL_c$ denote the identity component of $\bfL_c^\dagger$, and let $\bfP_c$ be the subgroup of $\bf P_c^\dagger$ defined over $\mf o$ such that $\bfP_c(\mf o)$ is the inverse image of $\bfL_c(\mathbb F_q)$ in $\bfP_c^\dagger(\mf o)$.
The groups $\bfP_c$ are called \emph{parahoric} subgroups of $\bfG(k)$.
We have analogous short exact sequences
\begin{equation}
    1 \to \bfU_c(\mf o) \to  \bfP_c(\mf o) \to  \bfL_c(\mathbb F_q) \to 1,
\end{equation}
and one on the level of the Lie algebra
\begin{equation}
    0 \to \bfu_c(\mf o) \to  \bfp_c(\mf o) \to  \bfl_c(\mathbb F_q) \to 0.
\end{equation}
When $c$ is a chamber in the building, then we call $\bfP_c$ an \emph{Iwahori subgroup} of $G$. 
Let 
$$\bfu = \bigcup_{c\subseteq \mathcal B(\bfG,k)}\bfu_c(\mf o)\subseteq \mf g(k), \quad \bfU = \bigcup_{c\subseteq \mathcal B(\bfG,k)}\bfU_c(\mf o)\subseteq \bfG(k).$$
These are the \emph{topologically nilpotent} and \emph{topologically unipotent} elements of $\mf g(k)$ and $\bfG(k)$ respectively.
\nomenclature{$\bfU$}{}
\nomenclature{$\bfu$}{}

\paragraph{Fourier Transforms}
By \cite[Proposition 4.1]{adler_roche}, for $p$ sufficiently large (see the reference for a precise bound for $p$), there exists a symmetric, non-degenerate $\bfG(k)$-invariant bilinear form 
\nomenclature{$\ms {B}$}{}
$$\ms B:\mf g(k)\times \mf g(k) \to k$$ 
such that for every face $c$ of $\mathcal B(\bfG,k)$ we have 
$$\bfp_c(\mf o) = \set{\ms X\in \mf g(k): \ms B(\ms X,\ms Y) \in \mf p, \forall Y\in \bfu_c(\mf o)}.$$
Such a bilinear form naturally descends for each face $c$ of $\mathcal B(\bfG,k)$ to a symmetric, non-degenerate $\bfL_c(\mathbb F_q)$-invariant bilinear form 
$$\ms B_c:\bfl_c(\mathbb F_q)\times \bfl_c(\mathbb F_q)\to \mathbb F_q.$$
Fix a Haar measure $\mu_{\mf g(k)}$ on $\mf g(k)$.
\nomenclature{$\FT(f),\hat f$}{}
For a function $f\in C_c^\infty(\mf g(k))$ we define \emph{the Fourier transform of $f$} to be
\begin{equation}
    \hat f(\ms X) := \FT_{\mf g(k)}(f)(\ms X) := \int_{\mf g(k)}\chi(\ms B(\ms X,\ms Y)) f(\ms Y)d\mu_{\mf g(k)}(\ms Y).
\end{equation}
Let $c$ be a face of $\mathcal B(\bfG,k)$ and $h:\bfl_c(\mathbb F_q)\to \C$ a function.
We define \emph{the Fourier transform of $h$} to be
\begin{equation}
    \hat h(\ms x) := \FT_{\bfl_c(\mathbb F_q)}(h)(\ms x) := \sum_{\ms y\in\bfl_c(\mathbb F_q)}\chi(\ms B_c(\ms x,\ms y)) h(\ms y).
\end{equation}
We define $\tilde h:\mf g(k)\to \C$ to be the function given by 
\begin{equation}
    \tilde h(\ms X) = 
    \begin{cases}
        h(\ms X+\bfu_c(\mf o)) & \text{if } \ms X\in \bfp_c(\mf o) \\
        0 & \text{ otherwise}.
    \end{cases}
\end{equation}
We say \emph{$\tilde h$ is inflated form $h$} and we have that 
$$\FT_{\mf g(k)}(\tilde h)(\ms X) = \mu_{\mf g(k)}(\bfu_c(\mf o))\cdot \tilde{\FT_{\bfl_c(\mathbb F_q)}(h)}(\ms X).$$

\paragraph{The Harish-Chandra-Howe Local Character Expansion}
\label{sec:hchlce}
Let $\exp:\bfu\to \bfU$ be the exponential map defined in \cite[Lemma 3.2]{barbaschmoy} and \cite[Section 3.3]{whomg} (the exponential map exists since we are assuming $p$ is sufficiently large. See the references for the precise bounds on $p$).
\nomenclature{$\exp$}{}
The map $\exp:\bfu\to \bfU$ has the property that for every face $c$ of $\mathcal B(\bfG,k)$, we have 
$$\exp(\bfu_c(\mf o)) = \bfU_c(\mf o),$$
and $\exp$ descends to the exponential from the nilpotent elements of $\bfl_c(\mathbb F_q)$ to the unipotent elements of $\bfL_c(\mathbb F_q)$.
For a function $f\in C_c^\infty(\bfU)$ let $f\circ \exp$ denote the function in $C_c^\infty(\mf g(k))$ given by 
\begin{equation}
    f\circ \exp(\ms X) = 
    \begin{cases}
        f(\exp(\ms X)) & \text{if } \ms X \in \bfu \\
        0 & \text{if } \ms X\not\in \bfu.
    \end{cases}
\end{equation}
For $\OO\in \mathcal N_o(k)$ let $\mu_{\OO}$ denote the corresponding nilpotent orbital integral.
We have the following result due to DeBacker (building on work by Waldspurger \cite{whomg}).
\begin{theorem}
    \cite[Theorem 3.5.2]{debacker_homog}
    Let $(\pi,X)$ be a depth-$0$ admissible representation of $\bfG(k)$.
    Then there exists $c_{\OO}(X)\in \C$ for each $\OO\in\mathcal N_o(k)$ such that for $f\in C_c^\infty(\bfU)$ we have
    \begin{equation}
        \Theta_X(f) = \sum_{\OO\in \mathcal N_o(k)} c_{\OO}(X)\hat \mu_{\OO}(f\circ \exp).
    \end{equation}
\end{theorem}
We remark that the local character expansion, and in particular the coefficients $c_\OO(X)$, always exists for admissible smooth representations.
The point of this theorem is that for depth-$0$ representations the expansion is valid for functions supported on $\bfU$.

\paragraph{The Wavefront Set of Representations of \texorpdfstring{$p$}{p}-adic Groups}
Let $(\pi,X)$ be a smooth admissible representation of $G$.
\nomenclature{$\WF(X)$}{}
\nomenclature{$^K\widetilde \WF(X)$}{}
\nomenclature{$^{\bar k}\WF(X)$}{}
The \emph{($p$-adic) wavefront set} is 
$$\WF(X) := \max_{\OO:c_\OO(X)\ne 0}\OO,$$
the \emph{unramified wavefront set} is
$$^K\widetilde\WF(X) := \max_{\OO:c_\OO(X)\ne 0}\mathcal N_o(K/k)(\OO),$$
and the \emph{geometric wavefront set} is
$$^{\bar k}\WF(X) := \max_{\OO:c_\OO(X)\ne 0}\mathcal N_o(\bar k/k)(\OO).$$
\begin{remark}
    In analogy with real groups and finite groups of Lie type it is expected that $^{\bar k}\WF(X)$ consists of a single nilpotent orbit - $\OO$ say.
    Moreover it is expected that for all $\OO'\in \WF(X)$, $\mathcal N_o(\bar k/k)(\OO') = \OO$ (this is a strictly stronger condition than $^{\bar k}\WF(X)$ being a singleton since a priori there might exist $\OO_1,\OO_2\in \mathcal N_o(k)$ which are incomparable, but $\mathcal N_o(\bar k/k)(\OO_1)<\mathcal N_o(\bar k/k)(\OO_2)$).
\end{remark}

\subsection{Lifting Nilpotent Orbits and Closure Relations}
\paragraph{The Lifting Map}
\label{sec:debacker_param}
Let $h$ be the Coxeter number of the absolute Weyl group for $\mathbf G$.
Since we are assuming $p$ is sufficiently large, we in particular require $p>3(h-1)$ so $p$ satisfies the conditions of section \ref{sec:hchlce} and we can apply the results of \cite{debacker} to $\mf g(k)$.
Let $\mathcal A$ be an apartment of $\mathcal B(\bfG,k)$.
For faces $c_1,c_2$ in $\mathcal A$ with $\mathcal A(c_1,\mathcal A) = \mathcal A(c_2,\mathcal A)$ the projection maps 
$$\bfP_{c_1}(\mf o)\cap \bfP_{c_2}(\mf o) \to \bfL_{c_1}(\mathbb F_q), \quad \bfP_{c_1}(\mf o)\cap \bfP_{c_2}(\mf o) \to \bfL_{c_2}(\mathbb F_q)$$
are both surjective with kernel $\bfU_{c_1}(\mf o)\cap \bfU_{c_2}(\mf o)$ and so there is an isomorphism 
\nomenclature{$i_{c_2,c_1}$}{}
$$i_{c_2,c_1}:\bfL_{c_1}(\mathbb F_q)\to \bfL_{c_2}(\mathbb F_q).$$
We similarly obtain an isomorphism 
\nomenclature{$j_{c_2,c_1}$}{}
$$j_{c_2,c_1}:\bfl_{c_1}(\mathbb F_q) \to \bfl_{c_2}(\mathbb F_q)$$
which is compatible with $i_{c_2,c_1}$ in the following sense:
\begin{equation}
    j_{c_2,c_1}(\ms h.\ms x) = i_{c_2,c_1}(\ms h).j_{c_2,c_1}(\ms x)
\end{equation}
for all $\ms h\in \bfL_{c_1}(\mathbb F_q),\ms x\in \bfl_{c_1}(\mathbb F_q)$.
For $\ms H \in \bfG(k)$ and $\ms x\in \bfl_c(\mathbb F_q)$ let $\ms H.\ms x$ denote the image of $\ms H.\ms X \in \bf p_{\ms Hc}(\mf o)$ in $\bfl_{\ms Hc}(\mathbb F_q)$ where $\ms X$ is any lift of $\ms x$ to $\bfp_c(\mf o)$.
This is well defined because $\ms H.\bfu_c(\mf o) = \bfu_{\ms Hc}(\mf o)$.
Let 
\nomenclature{$I^k$}{}
\nomenclature{$I_d^k$}{}
$$I^k = \{(c,\ms x):c\subseteq \mathcal B(\bfG,k),\ms x \in \mathcal N^{\bfL_c}(\mathbb F_q)\}.$$
Let $I_d^k$ denote the set of pairs $(c,\ms x)\in I^k$ where $\ms x$ is a distinguished nilpotent element of $\bfl_c(\mathbb F_q)$.
For $(c,\ms x)\in I^k$ let $\mathcal C(c,\ms x)$ denote the preimage of $\ms x$ in $\bfp_c(\mf o)$.
\nomenclature{$\sim_k$}{}
\nomenclature{$\mathcal C(c,x)$}{}
For $(c_1,\ms x_1),(c_2,\ms x_2)\in I^k$ we define $(c_1,\ms x_1)\sim_k (c_2,\ms x_2)$ if there exists an $\ms H\in \bfG(k)$ and an apartment $\mathcal A$ such that 
$$\mathcal A(c_2,\mathcal A) = \mathcal A(\ms Hc_1,\mathcal A), \text{ and } \ms x_2 = j_{c_2,\ms Hc_1}({\ms H}.\ms x_1).$$
\nomenclature{$\mathcal L_c$}{}
Given an $(c,\ms x)\in I^k$ one can attach to it, as in \cite{barbaschmoy} and \cite{debacker}, a well defined \emph{nilpotent orbit $\mathcal L_c(\ms x)\in\mathcal N_o(k)$ called its lift}.
It has the following two useful equivalent characterisations (due to DeBacker in \cite[Lemma 5.3.3]{debacker}):
\begin{enumerate}
    \item If $\ms x$ is included into an $\lsl_2$-triple $\ms x,\ms h, \ms y\in \bfl_c(\mathbb F_q)$, and $\ms X,\ms H,\ms Y\in \bfp_c(\mf o)$ is an $\lsl_2$-triple such that their images in $\bfl_c(\mathbb F_q)$ are $\ms x,\ms h,\ms y$ respectively (such an $\lsl_2$-triple always exists), then $\mathcal L_c(\ms x) = {\bfG(k)}.\ms X$.
    \item $\mathcal L_c(\ms x)$ is the \emph{unique minimal element} of $\{\OO \in \mathcal N_o(k):\OO \cap \mathcal C(c,\ms x)\ne \emptyset\}$.
\end{enumerate}
Let $\OO\in \mathcal N_o^{\bfL_c}(\mathbb F_q)$.
The nilpotent orbit $\mathcal L_c(\ms x)$ is independent of the choice of $\ms x\in \OO$; we write $\mathcal L_c(\OO)$ for the resulting nilpotent orbit.
\nomenclature{$I_o^k$}{}
Define 
$$I_o^k = \{(c,\OO):c\subseteq \mathcal B(\bfG,k), \OO\in \mathcal N_o^{\bfL_c}(\mathbb F_q)\}$$
and define $I_{o,d}^k$ to be the subset of $I^k$ consisting of pairs $(c,\OO)$ where $\OO$ is a distinguished nilpotent orbit of $\bfl_c(\mathbb F_q)$.
\nomenclature{$I^k(\OO)$}{}
\nomenclature{$I_o^k(\OO)$}{}
\nomenclature{$I_d^k(\OO)$}{}
\nomenclature{$I_{o,d}^k(\OO)$}{}
For $\OO\in \mathcal N_o(k)$ write $I^k(\OO)$ for the set $\set{(c,\ms x)\in I^k:\mathcal L_c(\ms x) = \OO}$.
Analogously define $I_o^k(\OO), I_d^k(\OO), I_{o,d}^k(\OO)$.
We have the following result due to 
Barbasch and Moy \cite[Corollary 3.7]{barbaschmoy} and DeBacker \cite[Theorem 5.6.1]{debacker} classifying the nilpotent orbits of $\mf g(k)$.
\begin{theorem}
    \label{thm:nilorbit}
    The map $I_d^k\to \mathcal N_o(k), (c,\ms x)\mapsto \mathcal L_c(\ms x)$ descends to a bijective correspondence between $I_d^k/\sim_k$ and $\mathcal N_o(k)$.
\end{theorem}
Note that for all the results in this section we are using the results from \cite{debacker} with $r=0$.

We can similarly define $I^K$, $I_d^K$, $I_o^K$, $I_{o,d}^K$, $I^K(\OO)$, $I_d^K(\OO)$, $I_o^K(\OO)$, $I_{o,d}^K(\OO)$, $\sim_K$, $\mathcal C$, and $\mathcal L_c$ for $\bfG(K)$ and the results in this section hold verbatim for these objects too.
We must be careful however since $K$ is not complete and this is a necessary condition in \cite{debacker}.
The only time this property is used however is in \cite[Lemma 5.2.1]{debacker}.
We give a proof for this result for $r=0$ and base field $K$ which means that the results in this section do indeed hold verbatim for $\bfG(K)$.
\begin{lemma}
    \label{lem:unramlift}
    Let $c$ be a face of $\mathcal B(\bfG,K)$ and let $\ms X,\ms H,\ms Y$ be an $\lsl_2$-triple contained in $\bfp_c(\mf O)$.
    Then 
    \begin{equation}
        {\bfU_c(\mf O)}.(\ms X+\ms c_{\bfu_c(\mf O)}(\ms Y)) = \ms X + \bfu_c(\mf O)
    \end{equation}
    where $\ms c_{\bfu_c(\mf O)}(\ms Y)$ denotes the centraliser of $\ms Y$ in $\bfu_C(\mf O)$.
\end{lemma}
\begin{proof}
    Since we are only looking at $\bfG(K) = \bfG_K(K)$, and $\bfG_K$ is split, we may as well assume that $\mb G$ is also split over $k$.
    Since $\bfG(K)$ acts transitively on the apartments of $\mathcal B(\bfG,K)$ we may also assume that $c\subseteq\mathcal B(G,k)$.
    There is nothing to prove for the $\subseteq$ direction.
    For the $\supseteq$ direction let $\ms Z\in \ms X+\bfu_c(\mf O)$.
    Since $\ms Z,\ms X,\ms Y\in \mf g(K)$ there is a finite unramified extension $F$ of $k$ such that $\ms Z,\ms X,\ms Y\in \mf g(F)$.
    Let $\mf o_F$ be the ring of integers for $F$.
    Then since $\mf O\cap F = \mf o_F$ we have that $Z\in X+\bfu_c(\mf o_F)$.
    Since $F$ is complete we can apply \cite[Lemma 5.2.1]{debacker} to $\bfG(F)$ and so 
    $${\bfU_c(\mf o_F)}.(\ms X+\ms c_{\bfu_c(\mf o_F)}(\ms Y)) = \ms X + \bfu_c(\mf o_F)$$
    Thus 
    $$Z\in {\bfU_c(\mf o_F)}.(\ms X+\ms c_{\bfu_c(\mf o_F)}(\ms Y)) \subset {\bfU_c(\mf O)}.(\ms X+\ms c_{\bfu_c(\mf O)}(\ms Y))$$
    as required.
\end{proof}

\paragraph{Closure relations}
\label{sec:closurerels}
The set $I_o^K$ comes with additional structure that $I_o^k$ does not.
For $(c_1,\OO_1),(c_2,\OO_2)\in I_o^K$ define 
$$(c_1,\OO_1)\le (c_2,\OO_2) \text{ if } c_1 = c_2 \text{ and } \OO_1 \le \OO_2.$$
The following result is implied by the proofs in \cite[Section 3.14]{barbaschmoy}.
\begin{proposition}
    \label{prop:lifted_rel}
    Let $(c_1,\OO_1),(c_2,\OO_2)\in I_o^K$ and suppose $(c_1,\OO_1)\le (c_2,\OO_2)$.
    Then $\mathcal L_{c_1}(\OO_1)\le \mathcal L_{c_2}(\OO_2)$.
\end{proposition}
\nomenclature{$\mathcal L$}{}
In other words, the map $\mathcal L:I_o^K\to \mathcal N_o(K), (c,\OO)\mapsto \mathcal L_c(\OO)$ is non-decreasing.
In section \ref{sec:pseudo-levis} we prove that $\mathcal L$ is in fact strictly increasing.
Let $\OO_1,\OO_2\in \mathcal N_o(K)$ and suppose $\OO_1\le \OO_2$.
\begin{enumerate}[(1)]
    \item We say $\OO_1 \le \OO_2$ is a \emph{lifted closure relation} if there exists a face $c$ of $\mathcal B(\bfG,K)$ and $(c,\OO_i')\in I_o^K(\OO_i)$ for $i=1,2$ such that $(c,\OO_1')\le (c,\OO_2')$ (cf. proposition \ref{prop:lifted_rel}).
    \item We say $\OO_1 \le \OO_2$ is a \emph{degenerate closure relation} if there exists $(c,\ms x)\in I^K(\OO_1)$ such that $\OO_2 \cap \mathcal C(c,\ms x)\ne \emptyset$ (cf. property 2 of section \ref{sec:debacker_param}).
\end{enumerate}
The following proposition shows that every closure relation in $\mathcal N_o(K)$ can be broken down into a lifted closure relation and a degenerate closure relation.
\begin{theorem}
    \label{thm:closurerels}
    Let $\OO_1,\OO_2\in \mathcal N_o(K)$, and suppose $\OO_1 \le \OO_2$.
    Then there exists a $\OO_{1.5}\in \mathcal N_o(K)$ such that $\OO_1\le \OO_{1.5} \le \OO_2$ where $\OO_1\le \OO_{1.5}$ is a lifted closure relation and $\OO_{1.5}\le \OO_2$ is a degenerate closure relation.
\end{theorem}
\begin{proof}
    Let $(c,\OO_1')\in I_o^K(\OO_1)$.
    Let $\ms x \in \OO_1'$ and $\ms X\in \mathcal C(c,\ms x)\cap \OO_1$.
    Since $\bfp_c(\mf O)$ is open in $\mf g(K)$, $\bfp_c(\mf O)\cap \OO_2 \ne \emptyset$.
    Let $\OO'$ be the image of $\OO_2\cap \bfp_c(\mf O)$ in $\bfl_c(\overline{\mathbb F}_q)$.
    We claim that $\OO_1' \subseteq \overline{\OO'}$.
    Let $U$ be an open subset of $\bfl_c(\overline{\mathbb F}_q)$ containing $\ms x$.
    Let $\tilde U$ be the preimage in $\bfp_c(\mf O)$.
    Since $\bfu_c(\mf O)$ is open, the projection map $\bfp_c(\mf O)\to \bfl_c(\overline{\mathbb F}_q)$ is continuous and so $\tilde U$ is open and contains $\ms X$.
    Thus $\tilde U\cap \OO_2 \ne \emptyset$ and so $U\cap \OO' \ne \emptyset$.
    This proves the claim.
    Write $\OO' = \cup_i \OO'^{(i)}$ as a union of $\bfL_c(\overline{\mathbb F}_q)$ nilpotent orbits.
    Then since $\OO_1' \subseteq \overline{\OO'}$ there exists an $i$ such that $\OO_1' \le \OO'^{(i)}$.
    Let $\OO_{1.5}' = \OO'^{(i)}$.
    By construction $\OO_{1.5} := \mathcal L_c(\OO_{1.5}')$ has the required properties.
\end{proof}

For $\OO\in \mathcal N_o(k)$ we say $\ms X, \ms H, \ms Y$ is an $\lsl_2$-triple for $\OO$ if they are an $\lsl_2$-triple and $\ms X\in \OO$.
We now show that $\mathcal N_o(K/k):\mathcal N_o(k) \to \mathcal N_o(K)$ and $\mathcal N_o(\bar k/k):\mathcal N_o(k) \to \mathcal N_o(\bar k)$ are strictly increasing.
\begin{lemma}
    \label{lem:slodowy}
    Let $\OO \in \mathcal N_o(k)$ and let $\ms X,\ms H,\ms Y$ be an $\lsl_2$-triple for $\OO$.
    Let $\ms s = \ms X + \ms c_{\mf g(k)}(\ms Y)$ (a Slodowy slice for $\OO$).
    Then
        \begin{enumerate}[(1)]
            \item $\OO \cap \ms s = \set{\ms X}$,
            \item if $\OO' \in \mathcal N_o(k)$ and $\overline{\OO'} \cap \ms s \ne \emptyset$ then $\OO' \cap \ms s \ne \emptyset$.
        \end{enumerate}
\end{lemma}
\begin{proof}
    We have the decomposition
    \begin{equation}
        \label{eq:slodowy}
        \mf g(k) = [\mf g(k),\ms X] \oplus \ms c_{\mf g(k)}(\ms Y)
    \end{equation}
    and $[\mf g(k), \ms X]$ is the tangent space of $\OO$ at $\ms X$.
    Thus 
    $$T_{\ms X}(\ms s \cap \OO) \hookrightarrow T_{\ms X} \ms s \cap T_{\ms X} \OO = 0$$ 
    and so $\ms s \cap \OO$ is discrete.
    However, if $\ms X'$ is in $\ms s\cap \OO$, let $\lambda_{X'}:\mathbf G_m\to \mathbf G$ be a 1-parameter $k$-subgroup such that 
    $${\lambda_{\ms X'}(t)}.\ms X' = t^{-2}\ms X'$$
    and write $\ms X' = \ms X + \ms Z$ where $\ms Z\in \ms c_{\mf g(k)}(\ms Y)$.
    Let $\lambda$ be the 1-parameter $k$-subgroup of $\bfG(k)$ attached to $\ms H$.
    Write $\mf g(k)(i)$ for the set of $\ms W\in \mf g(k)$ such that $\lambda(t).\ms W = t^i \ms W$ and let 
    $$\mf g(k)(\le 0) = \bigoplus_{i\le 0}\mf g(k)(i).$$
    Since $\ms c_{\mf g(k)}(\ms Y) \subseteq \mf g(k)(\le 0)$, write $\ms Z = \sum_{i\le 0}\ms Z_i$ where $\ms Z_i\in \mf g(k)(i)$.
    Then
    \begin{equation}
        {\lambda(t^{-1})\lambda_{\ms X'}(t^{-1})}.\ms X' = \ms X + \sum_{i\le 0}t^{2-i}\ms Z_i \in \OO \cap \ms s 
    \end{equation}
    for all $t\in k$ and $\to \ms X$ as $t\to 0$.
    Since $\ms s \cap \OO$ is discrete this means that $\ms X' = \ms X$.
    This proves (1).

    Let $\Ad:\bfG(k) \times \ms s \to \mf g(k)$ be the restriction of the adjoint map.
    $\Ad$ is smooth with differential $T_{(1,\ms X)}$ which is onto by equation \ref{eq:slodowy}.
    All the varieties in question are smooth and so there exists a Zariski open (and hence open in the topology induced by $k$) subset $V$ in ${\bfG(k)}.\ms s$ containing $\ms X$.
    We have that $\overline{\OO'}\cap \ms s \ne \emptyset$ and so $\overline{\OO'} \cap {\bfG(k)}.\ms s \ne \emptyset$.
    Since $\ms s \cap \overline{\OO'}$ is closed and non-empty the argument for part (1) shows that $\ms X \in \ms s \cap \overline{\OO'}$.
    It follows that $\ms X \in \overline{\OO'} \cap V$.
    But $V$ is open and so $\OO' \cap V \ne \emptyset$.
    It follows that $\OO' \cap {\bfG(k)}.\ms s \ne \emptyset$ and so $\OO' \cap \ms s \ne \emptyset$ as required.
\end{proof}
\begin{corollary}
    \label{cor:slodowy}
    Let $\OO, \OO'\in \mathcal N_o(k)$.
    Let $\ms X,\ms H,\ms Y$ be an $\lsl_2$-triple for $\OO$ and $\ms s = \ms X + \ms c_{\mf g(k)}(\ms Y)$.
    Then
    \begin{enumerate}
        \item $\OO = \OO'$ iff $\OO' \cap \ms s$ is a singleton,
        \item $\OO < \OO'$ (i.e. $\OO\le \OO'$ and $\OO\ne \OO'$) iff $\OO' \cap \ms s$ has more than one element,
    \end{enumerate}
\end{corollary}
\begin{proof}
    By Lemma \ref{lem:slodowy} (1), if $\OO = \OO'$ then $\OO' \cap \ms s = \set{\ms X}$.
    If $\OO'\cap \ms s$ is a singleton then it is closed and the same argument as Lemma \ref{lem:slodowy} (1) gives that $\OO' \cap \ms s = \set{\ms X}$ and so $\OO = \OO'$.
    If $\OO < \OO'$ then $\ms X \in \overline{\OO'} \cap \ms s$.
    By Lemma \ref{lem:slodowy} (2), $\OO' \cap \ms s\ne \emptyset$.
    It cannot consist of a single element since by part 1 this would imply $\OO = \OO'$.
    Thus $\OO'\cap \ms s$ consists of more than one element.
    If $\OO' \cap \ms s$ consists of more than one element then $\overline{\OO'} \cap \ms s \ne \emptyset$ and so contains $\ms X$.
    Thus $\OO \le \OO'$.
    But $\OO \ne \OO'$ since $\OO' \cap \ms s$ is not a singleton and so $\OO < \OO'$.
\end{proof}
Analogous results to Lemma \ref{lem:slodowy} and Corollary \ref{cor:slodowy} hold for nilpotent orbits of $\mf g(\bar k)$, though different proof methods must be used (see \cite[Lemma 5.10]{ggg-np} for details).
\begin{theorem}
    \label{thm:saturation-increasing}
    Let $\OO,\OO'\in \mathcal N_o(k)$.
    \begin{enumerate}[(1)]
        \item If $\OO < \OO'$ then $\mathcal N_o(K/k)(\OO) < \mathcal N_o(K/k)(\OO')$.
        \item If $\OO < \OO'$ then $\mathcal N_o(\bar k/k)(\OO) < \mathcal N_o(\bar k/k)(\OO')$.
    \end{enumerate}
\end{theorem}
\begin{proof}
    Clearly 2 implies 1 so it suffices to show that $\OO < \OO' \implies \mathcal N_o(\bar k/k)(\OO) < \mathcal N_o(\bar k/k)(\OO')$.
    Let $\ms X,\ms H, \ms Y$ be an $\lsl_2$-triple for $\OO$ and let $\ms s = \ms X + \ms c_{\mf g(k)}(\ms Y)$.
    If $\OO < \OO'$ then $\OO' \cap \ms s$ has more than one element.
    Let $\mf s = \ms X + c_{\mf g(\bar k)}(\ms Y) \supseteq \ms s$.
    Then $\mathcal N_o(\bar k/k)(\OO') \cap \mf s$ contains $\OO'\cap \ms s$ and so also has more than one element.
    Thus $\mathcal N_o(\bar k/k)(\OO) < \mathcal N_o(\bar k/k)(\OO')$.
\end{proof}

\subsection{The Wavefont Set of Representations of Finite Groups of Lie Type}
\label{sec:ffltwf}
\paragraph{Generalised Gelfand--Graev Representations}
\label{sec:kwfs}
In section \ref{sec:ffltwf} only let $\mathbf G$ be a connected reductive group defined over $\mathbb F_q$.
Fix an algebraic closure $\overline{\mathbb F}_q$ of $\mathbb F_q$ and let $F:\bfG(\barF_q)\to \bfG(\barF_q)$ be the associated geometric Frobenius (so that $\bfG(\barF_q)^F = \bfG(\mathbb F_q)$).
Let $h$ be the Coxeter number of the absolute Weyl group for $\mathbf G$ and suppose $p>3(h-1)$.
Then the nilpotent cone $\mathcal N(\barF_q)$ of $\mf g(\barF_q)$ may be identified with the unipotent cone of $\mb G(\barF_q)$ via the exponential map $\exp:\mathcal N(\barF_q) \to \mathcal U(\barF_q)$ \cite[Section 2.1]{barbaschmoy}.
Let $\ms B:\mf g(\mathbb F_q)\times \mf g(\mathbb F_q)\to \mathbb F_q$ be a symmetric, non-degenerate $\mb G$-invariant bilinear form and $\chi:\mathbb F_q\to \C^\times$ be a non-trivial character. 
Recall that for a function $f:\mf g(\mathbb F_q)\to \C$ \emph{the Fourier transform of $f$} is 
$$\hat f(\ms x) = \sum_{\ms x\in \mf g(\mathbb F_q)}\chi(\ms B(\ms x,\ms y))f(\ms y).$$

\nomenclature{$\Gamma_{\ms n},\Gamma_{\OO}$}{}
\nomenclature{$\gamma_{\ms n},\gamma_{\OO}$}{}
For a nilpotent element $\ms n\in \mathcal N(\mathbb F_q)$ we may associate to it a representation $\Gamma_{\ms n}$ of $\bfG(\mathbb F_q)$ called \emph{the associated Generalised Gelfand--Graev Representation} or GGGR for short (see \cite[Section 2]{lusztig} for details on its construction).
Write $\gamma_{\ms n}$ for the character of $\Gamma_{\ms n}$.
$\gamma_{\ms n}$ has the following key properties (due to Kawanaka \cite{kawanaka}, \cite{kawanakab})
\begin{enumerate}[(1)]
    \item $\gamma_{\ms n}$ only depends on the $\Ad(\bfG(\mathbb F_q))$-orbit of $\ms n$. If $\OO = \bfG(\mathbb F_q).\ms n$ write $\Gamma_\OO$ (resp. $\gamma_\OO$) for the resulting representation (resp. character),
    \item the support of $\gamma_{\ms n}$ is contained in the closure of $\mb G(\barF_q).\exp(\ms n)$.
\end{enumerate}
Let $\ms n=\ms e,\ms h,\ms f$ be an $\lsl_2$-triple in $\mf g(\barF_q)$ and $\Sigma = -\ms f + \ms c_{\mf g(\barF_q)}(\ms e)$.
Let $r(\ms n) = \frac12 (\dim \mf g(\barF_q)-\dim \ms c_{\mf g(\barF_q)}(\ms n))$.
We have the following result due to Lusztig about the Fourier transform of $\gamma_{\ms n}\circ \exp$ (which we will also refer to as $\gamma_{\ms n}$).
\begin{proposition}
    \label{prop:gggvals}
    \cite[Proposition 2.5, Proposition 6.13]{lusztig}
    Let $\ms n,\ms n'\in \mathcal N(\mathbb F_q)$.
    \begin{enumerate}[(1)]
        \item $\hat\gamma_{\ms n}(\ms y) = q^{r(\ms n)}\#\set{\ms g\in \bfG(\mathbb F_q):{\ms g}.\ms y\in \Sigma}$ for all $\ms y\in \mf g(\mathbb F_q)$,
        \item if $\hat \gamma_{\ms n}(\ms n') \ne 0$, then $\ms n$ must lie in the closure of $\mb G(\barF_q).\ms n'$,
        \item if $\ms n\in \mb G(\barF_q).\ms n'$ and $\hat\gamma_{\ms n}(\ms n') \ne 0$, then $\ms n'\in\bfG(\mathbb F_q).\ms n$,
        \item $\hat \gamma_{\ms n}(\ms n) = q^{r(\ms n)}\#\ms C_{\bfG(\mathbb F_q)}(\ms n)$.
    \end{enumerate}
\end{proposition}


\paragraph{The Kawanaka Wavefront Set}
\label{sec:kawanaka}
The Kawanaka wavefront set is the analogous notion of the $p$-adic wavefront set for complex representations of $\bfG(\barF_q)$.
In this the wavefront set was introduced by Kawanaka and so we refer to it as the \emph{Kawanaka} wavefront set.
Let $(\rho, V)$ be an irreducible representation of $\bfG(\mathbb F_q)$ and $\chi_V$ be the character afforded by $V$. 
\nomenclature{$^{\barF_q}\WF(V)$}{}
The \emph{Kawanaka wavefront set $^{\barF_q}\WF(V)$ of $V$} is defined to be the nilpotent orbit $\OO\in \mathcal N_o(\barF_q)$ satisfying
\begin{enumerate}[(1)]
    \item there exists an $\OO'\in \mathcal N_o(\mathbb F_q)$ such that $\langle \gamma_{\OO'},\chi_V\rangle \ne 0$ and $\mathcal N_o(\barF_q/\mathbb F_q)(\OO') = \OO$;
    \item if $\OO'\in \mathcal N_o(\mathbb F_q)$ and $\langle \gamma_{\OO'},\chi_V\rangle \ne 0$ then $\mathcal N_o(\barF_q/\mathbb F_q)(\OO') \le \OO$.
\end{enumerate}
It is not clear a priori that such an orbit exists, but if it does, then it is clear that it is unique.

The existence of the Kawanaka wavefront set has a somewhat long and complicated history.
Originally conjectured to always exist by Kawanaka, he proved that this is indeed the case for adjoint groups of type $A_n$, or of exceptional type \cite{kawanaka}.
He also gave a conjectural description of $\hphantom{ }^{\barF_q}\WF(V)$ in terms of Lusztig's classification of the irreducible representations of $\bfG(\mathbb F_q)$.
In particular, let $\mb G^*$ be the dual group of $\mb G$ defined over $\mathbb F_q$ with corresponding Frobenius $F'$ and suppose $V$ corresponds to the $F'$-stable special $G^*(\barF_q)$-conjugacy class $C$ (in the sense of \cite[Section 13.2]{lusztig}).
Pick an element $g$ of $C^{F'}$ and let $g=su$ be its Jordan decomposition.
We can attach to the Weyl group $W(s)$ of $C_{G^*}(s)$ an irreducible (special) representation $E$ via the Springer correspondence applied to $u$ and the trivial local system.
The representation $E'=j_{W(s)}^WE$ obtained via truncated induction then corresponds to an $F$-stable nilpotent orbit $\OO$ and the trivial local system with respect to the Springer correspondence on $\mf g(\barF_q)$.
The orbit $\OO$ was Kawanaka's candidate for $\hphantom{ }^{\barF_q}\WF(V)$.
This conjecture was partially proved by Lusztig in his paper \cite{lusztig}.
In this paper Lusztig attached to $V$ an $F$-stable unipotent class $C$ of $\mb G(\barF_q)$ called the unipotent support of $V$ which is the unipotent class of $\mb G(\barF_q)$ of maximal dimension satisfying
\begin{equation}
    \label{eq:unipsupp}
    \sum_{g\in C^F}\chi_V(g) \ne 0.
\end{equation}
He then showed that the (log of the) unipotent support of the Alvis--Curtis dual of $V$ is the unique nilpotent orbit of maximal dimension satisfying condition 1. above. 
This essentially settled the existence claim, modulo the slight weakening of condition 2. above.
This however was fixed in later work by Achar and Aubert in \cite{achar_aubert} (and Taylor \cite{wavefront} with weakened conditions on the characteristic) - finally settling the matter fully.

In section 3 we will need to know the Kawanaka wavefront set for principal series unipotent representations of $\bfG(\mathbb F_q)$ when $\bfG$ is split.
We record a precise formula for $^{\barF_q}\WF$ for this case.

First recall that for unipotent principal series representations we have a $q\to 1$ operation arising from Lusztig's isomorphism \cite[Theorem 3.1]{lusztigdeformation} 
\begin{equation}
    \C[W]\to \mathcal H(\mb B(\mathbb F_q)\backslash \bfG(\mathbb F_q)/\mb B(\mathbb F_q))
\end{equation}
that gives a bijection between the irreducible constituents of $\Ind_{\mb B(\mathbb F_q)}^{\bfG(\mathbb F_q)}1$ and irreducible $W$-representations.
For $V$ a constituent of $\Ind_{\mb B(\mathbb F_q)}^{\bfG(\mathbb F_q)}1$ we write $V_{q\to 1}$ for the corresponding irreducible representation of $W$.
Second, recall Lusztig's partition of $\mathrm{Irr}(W)$ into families so that each family contains a unique special representation.
Each special representation of $W$ corresponds via the Springer correspondence to a special nilpotent orbit $\OO$ and the trivial local system.
\nomenclature{$\OO^s$}{}
For an irreducible representation $E$ of $W$ we write $\OO^{s}(E)$ for the special nilpotent orbit corresponding to the special representation in the same family as $E\otimes \sgn$.
Unravelling the above recipe for the Kawanaka wavefront set we get that for principal series unipotent representations of $\bfG(\mathbb F_q)$, the Kawanaka wavefront set is given by 
\begin{equation}
    \label{eq:kawanakawf}
    ^{\barF_q}\WF(V) = \OO^s(V_{q\to 1}).
\end{equation}

We also record the following proposition.
\begin{proposition}
    \label{prop:contra}
    Let $(\rho,V)$ be an irreducible representation of $\bfG(\mathbb F_q)$ and $\rho^*$ denote its contragredient.
    Then $\hphantom{ }^{\barF_q}\WF(V) = \hphantom{ }^{\barF_q}\WF(V^*)$.
\end{proposition}
\begin{proof}
    The unipotent support of a representation and its contragredient are trivially the same (see equation \ref{eq:unipsupp}).
    The result then follows from the fact that Alvis--Curtis duality commutes with taking contragredients.
\end{proof}
Finally, we make the following convenient definitions.
When $\rho$ is not necessarily irreducible, \emph{the Kawanaka wavefront set of $\rho$} is the collection of maximal orbits among the wavefront sets of the irreducible constituents.
An element of the Kawanaka wavefront set is called a Kawanaka wavefront-set nilpotent.

\subsection{Relating the \texorpdfstring{$\mathbb Q_p$}{Qp} and \texorpdfstring{$\mathbb F_p$}{Fp} Wavefront Sets}
\paragraph{Inflated Generalised Gelfand--Graev Representations}
Let $(c,\OO)\in I_o^k$.
\nomenclature{$f_{c,\OO}$}{}
Define the function $f_{c,\OO} = \tilde \gamma_{\OO}$ where $\gamma_{\OO}$ is the character of the GGGR of $\bfL_c(\mathbb F_q)$ attached to the orbit $\OO$.
We will also write $f_{c,\OO}$ for $f_{c,\OO}\circ \exp$.
The following result is essentially due to Barbasch and Moy in \cite{barbaschmoy}, but is a sharper result than in loc. cit.
\begin{theorem}
    \label{thm:bm}
    Let $(c,\OO)\in I_o^k$.
    Then
    \begin{enumerate}[(1)]
        \item $f_{c,\OO}$ is supported on the topologically unipotent elements $\bfu$,
        \item $\hat \mu_{\OO'}(f_{c,\OO}) = 0$ unless $\mathcal L_c(\OO) \le \OO'$.
        \item Suppose $\OO'\in \mathcal N_o(k)$ is such that $\mathcal N_o(\bar k/k)(\OO') = \mathcal N_o(\bar k/k)(\mathcal L_c(\OO))$. Then
        \begin{enumerate}
                \item If $\mathcal L_c(\OO)\ne \OO'$, then $\hat \mu_{\OO'}(f_{c,\OO}) = 0$.
                \item If $\mathcal L_c(\OO) = \OO'$, then $\hat \mu_{\OO'}(f_{c,\OO}) \ne 0$.
            \end{enumerate}
        \item For any irreducible smooth admissible representation $(\pi,X)$ of $\bfG(k)$, we have $$\Theta_X(f_{c,\OO}) = \langle\Gamma_{c,\OO},\check X^{\bfu_c(\mf o)}\rangle.$$
    \end{enumerate}
\end{theorem}
\nomenclature{$\check X$}{}
Here $\check X$ denotes the contragredient (i.e. the smooth dual) of $X$.
The proof for (2) in \cite{barbaschmoy} however only shows that $\mathcal L_c(\OO)$ lies in the closure of $\mathcal N_o(K/k)(\OO') \cap \mf g(k)$ in $\mf g(k)$.
We now give a complete proof of (2) using ideas from \cite{wast}.
\begin{proof}
    Let $\ms x,\ms h,\ms y$ be an $\lsl_2$-triple for $\OO$.
    Let $\ms X,\ms H,\ms Y$ denote a lift of $\ms x,\ms h,\ms y$ to an $\lsl_2$-triple of $\mf g(k)$.
    We proceed by first showing that 
    \begin{equation}
        \label{eq:supp}
        \supp(\hat f_{c,\OO}) = \set{{\ms h}.(-\ms Y + \ms Z):\ms h\in \bfP_c(\mf o),\ms Z\in \ms Z_{\bfp_c(\mf o)}(\ms X)}.
    \end{equation}
    We have that $\hat f_{c,\OO}(\ms W) \ne 0$ iff $\ms W\in \bfp_c(\mf o)$ and $\hat\gamma_{\OO}(\ms w) \ne 0$ where $\ms w$ is the image of $\ms W$ in $\bfl_c(\mathbb F_q)$.
    By \cite[Equation 2.4 (a)]{lusztig}, $\hat\gamma_{\OO}(\ms w)\ne 0$ iff $\ms w \in {\bfL_c(\mathbb F_q)}.(-\ms y + \ms Z_{\bfl_c(\mathbb F_q)}(\ms e))$.
    By the proof of \cite[Lemma IX.3]{wast}, we know that the image of $\ms Z_{\bfp_c(\mf o)}(X)$ in $\bfl_c(\mathbb F_q)$ is $\ms Z_{\bfl_c(\mathbb F_q)}(\ms x)$.
    Thus the support consists of those $\ms W$ in ${\bfP_c(\mf o)}.(-\ms Y+\ms Z_{\bfp_c(\mf o)}(\ms X))+\bfu_c(\mf o) = {\bfP_c(\mf o)}.(-\ms Y+\ms Z_{\bfp_c(\mf o)}(\ms X)+\bfu_c(\mf o))$.
    But from the same proof in \cite{wast} we also know that for $\ms Z\in \ms Z_{\bfp_c(\mf o)}(\ms X)$
    \begin{equation}
        -\ms Y + \ms Z + \bfu_c(\mf o) = \set{{\ms h}.(-\ms Y+\ms Z+\ms Z'):\ms Z'\in \ms Z_{\bfu_c(\mf o)},\ms h\in \bfU_c(\mf o)}.
    \end{equation}
    Thus $-\ms Y + \ms Z_{\bfp_c(\mf o)}(\ms X) + \bfu_c(\mf o) = {\bfU_c(\mf o)}.(-\ms Y+\ms Z_{\bfp_c(\mf o)})$ and so equation \ref{eq:supp}. holds.

    Now let $\OO'$ be a nilpotent orbit with $\hat \mu_{\OO'}(f_{c,\OO})\ne0$. 
    Then $\supp(\hat f_{c,\OO})\cap \OO' \ne \emptyset$.
    Thus there is a $\ms X' \in \OO'$, $\ms h\in \bfP_c(\mf o)$ and $\ms Z\in \ms Z_{\bfp_c(\mf o)}(\ms X)$ such that $\ms X' = {\ms h}.(-\ms Y+\ms Z)$.
    Since ${\ms h^{-1}}.\ms X'\in \OO'$ we can assume $\ms h = 1$ and so $\ms X' = -\ms Y+\ms Z$.
    Let $\lambda_{\ms X'}: \mathbf G_m\to \mathbf G$ be a 1-parameter $k$ subgroup so that $\lambda_{\ms X'}(t).\ms X' = t^2\ms X'$, and $\lambda$ be the 1-parameter subgroup determined by $\ms H$.
    Since $\ms Z_{\mf g(k)}(\ms X) \subseteq \mf g(k)(\ge 0)$, write $\ms Z = \sum_{i\ge 0}\ms Z_i$ where $\ms Z_i\in \mf g(k)(i)$.
    Then
    \begin{equation}
        {\lambda(t)\lambda_{\ms X'}(t)}.\ms X' = -\ms Y + \sum_{i\ge 0}t^{2+i}\ms Z_i \to -\ms Y
    \end{equation}
    as $t\to 0$.
    Thus $\OO = {\bfG(k)}.(-Y) \le \OO'$.
\end{proof}
Note that Theorem \ref{thm:saturation-increasing} together with (2) of this theorem immediately imply (3) (a) of this theorem.

\begin{proposition}
    Suppose $(\pi,X)$ has depth-$0$.
    \nomenclature{$\Xi(X),\Xi^{max}(X)$}{}
    Let 
    $$\Xi(X) = \{\OO\in \mathcal N_o(k):\text{there exists } (c,\OO')\in I_o^k(\OO) \text{ such that }\Theta_X(f_{c,\OO'})\ne 0\}.$$
    Write $\Xi^{max}(X)$ for the maximal orbits of $\Xi(X)$.
    Then
    \begin{enumerate}[(1)]
        \item $\WF(X) = \Xi^{max}(X)$,
        \item if $\OO \in \Xi^{max}(X)$, then $\Theta_X(f_{c,\OO'}) \ne 0$ for all $(c,\OO')\in I_o^k(\OO)$.
    \end{enumerate}
\end{proposition}
\begin{proof}
    Let $\OO\in \WF(X)$ and $(c,\OO')\in I_o^k(\OO)$.
    Since $X$ has depth-$0$ and the inflated GGGRs have support in $\bfU$ we have
    \begin{equation}
        \Theta_X(f_{c,\OO'}) = \sum_{\OO''\in \mathcal N_o(k)} c_{\OO''}\hat\mu_{\OO''}(f_{c,\OO'}).
    \end{equation}
    Then by Theorem \ref{thm:bm} part 2 we have that
    \begin{equation}
        \Theta_X(f_{c,\OO'}) = \sum_{\OO\le \OO''} c_{\OO''}\hat\mu_{\OO''}(f_{c,\OO'}).
    \end{equation}
    But if $\OO < \OO''$ then $c_{\OO''} = 0$ and so $\Theta_X(f_{c,\OO'}) = c_{\OO'}(X)\hat \mu_{\OO}(f_{c,\OO'}) \ne 0$.
    Thus $\WF(X)\subseteq \Xi(X)$.

    Now suppose $\OO\in \WF(X)$ and $\OO_1$ is a nilpotent orbit with $\OO < \OO_1$.
    Let $(c_1,\OO_1')\in I_o^k(\OO_1)$.
    Then $\Theta_X(f_{c_1,\OO_1}) = \sum_{\OO_1\le \OO_2} c_{\OO_2}\hat\mu_{\OO_2}(f_{c_1,\OO_1'})$.
    But if $\OO_1\le \OO_2$ then $\OO< \OO_2$ and so $c_{\OO_2}(X) = 0$ and so $\Theta_X(f_{c_1,\OO_1'}) = 0$.
    Thus we get that $\WF(X)\subseteq \Xi^{max}(X)$.

    Finally, suppose $\OO$ is in $\Xi^{max}(X)$ and $(c,\OO')\in I_o^k(\OO)$.
    $\OO$ must be $\le \OO_1$ for some $\OO_1 \in \WF(X)$ (since otherwise $c_{\OO_2}(X) = 0$ for all $\OO\le \OO_2$ and so $\Theta_X(f_{c,\OO'}) = 0$).
    But $\WF(X)\subseteq \Xi(X)$ and so by maximality we must have $\OO = \OO_1$ and so $\Xi^{max}(X) \subseteq \WF(X)$. This establishes 1.

    To establish 2., note by the first part we have that $\Theta_X(f_{c,\OO'}) \ne 0$ for all $\OO \in \WF(X)$ and $(c,\OO')\in I_o^k(\OO)$.
    Then use the fact that $\Xi^{max}(X) = \WF(X)$.
\end{proof}
\begin{proposition}
    \label{prop:wfs}
    $\hphantom{ }$ 
    \begin{enumerate}
        \item $\hphantom{ }^K\widetilde\WF(X) = \max\set{\mathcal N_o(K/k)(\OO):\OO\in \WF(X)}$,
        \item $^{\bar k}\WF(X) = \max\set{\mathcal N_o(\bar k/K)(\OO):\OO\in \hphantom{ }^K\widetilde\WF(X)}$.
    \end{enumerate}
\end{proposition}
\begin{proof}
    This follows from $\mathcal N_o(K/k)$ and $\mathcal N_o(\bar k/K)$ being non-decreasing.
\end{proof}
\begin{corollary}
    \label{cor:maxl}
    Suppose $(\pi,X)$ has depth-$0$.
    Then $\hphantom{ }^K\widetilde\WF(X)$ is equal to the set of maximal orbits of $\mathcal N_o(K/k)(\Xi(X))$.
\end{corollary}

\paragraph{Local Wavefront Sets}
\label{sec:locwf}
\nomenclature{$\WF_c(X)$}{}
\nomenclature{$^K\WF_c(X)$}{}
We now define local wavefront sets $\WF_c$ and $\hphantom{ }^K\WF_c$ for $c$ a face of $\mathcal B(\bfG,k)$.
Let $\OO_1,\dots,\OO_k \in \mathcal N_o^{\bfL_c}(\barF_q)$ be the Kawanaka wavefront-set nilpotents of $W_c := \check V^{\bfU_c(\mf o)}$, a representation of $\bfL_c(\mathbb F_q)$.
Let $\OO_1',\dots,\OO_l' \in \mathcal N_o^{\bfL_c}(\mathbb F_q)$ be the nilpotent orbits in $\cup_i\OO_i$ such that $\langle\Gamma_{\OO_i'},W_c\rangle \ne 0$.
Define 
$$\WF_c(X) =\set{\mathcal L_c(\OO_i'):1\le i\le l} \text{ and } \hphantom{ }^K\WF_c(X) = \set{\mathcal L_c(\OO_i):1\le i\le k}.$$
It is clear that 
$$^K\WF_c(X) = \set{\mathcal N_o(K/k)(\OO):\OO\in \WF_c(X)}.$$
\nomenclature{$\mathscr C$}{}
Let $\mathscr C$ denote a collection of faces in $\mathcal B(\bfG,k)$ such that for all $\OO\in \mathcal N_o(k)$ there exists a $(c,\OO')\in I_o^k(\OO)$ such that $c\in \mathscr C$.
Examples of $\mathscr C$ that satisfy this property are:
\begin{enumerate}[(1)]
    \item the faces of a fixed chamber $c_0$;
    \item the vertices of a fixed chamber $c_0$;
    \item a choice of vertex from each $\bfG(k)$ orbit of the vertices of a fixed chamber $c_0$.
\end{enumerate}

\begin{theorem}
    \label{lem:liftwf}
    Let $(\pi,X)$ be depth-$0$.
    Then
    $$^K\widetilde\WF(X) = \max_{c\in \mathscr C}\hphantom{ }^K\WF_c(X).$$
\end{theorem}
\begin{proof}
    Let $^K\Xi(X) = \mathcal N_o(K/k)(\Xi(X))$.
    We will show that $^K\WF_{c}(X)\subseteq \hphantom{ }^K\Xi(X)$ for all $c\in \mathscr C$, and that if $\OO$ is a maximal element of $^K\Xi(X)$ then $\OO\in \hphantom{ }^K\WF_c(X)$ for some $c\in \mathscr C$.
    Corollary \ref{cor:maxl} then implies the result.

    The first part is straightforward.
    Let $c$ be any face of $\mathcal B(\bfG,k)$ and $\OO \in \hphantom{ }^K\WF_c(X)$.
    Write $\OO$ as $\mathcal N_o(K/k)(\OO_1)$ where $\OO_1$ is in $\WF_c(X)$.
    By definition of $\WF_c(X)$, $\OO_1$ is in $\Xi(X)$.
    Thus we have $\hphantom{ }^K\WF_c(X) \subseteq \hphantom{ }^K\Xi(X)$.

    For the second part let $\OO$ be a maximal element of $^K\Xi(X)$.
    Write 
    $$\OO = \mathcal N_o(K/k)(\OO_1)$$
    where $\OO_1\in \Xi(X)$.
    Let $(c,\OO_1')\in I_o^k(\OO_1)$ be such that $c\in\mathscr C$.
    Since $\Theta_X(f_{c,\OO_1'})\ne 0$, there is an irreducible constituent $W$ of $\check V^{\bfU_c(\mf o)}$ with 
    $$\langle \Gamma_{\OO_1'},W\rangle \ne 0.$$
    Let $\OO_2'\in \mathcal N_o^{\bfL_c}(\barF_q)$ be a Kawanaka wavefront set nilpotent of $\check V^{\bfU_c(\mf o)}$ such that 
    $$\mathcal N_o(\barF_q/\mathbb F_q)(\OO_1') \le \OO_2'.$$
    Let $\OO_2 = \mathcal L_c(\OO_2')$.
    $\OO_2$ is an element of $^K\Xi(X)$.
    By Proposition \ref{prop:lifted_rel}, 
    $$\mathcal N_o(\barF_q/\mathbb F_q)(\OO_1') \le \OO_2'$$
    implies that
    $$\OO = \mathcal N_o(K/k)(\OO_1)  = \mathcal N_o(K/k)(\mathcal L_c(\OO_1')) = \mathcal L_c(\mathcal N_o(\barF_q/\mathbb F_q)(\OO_1'))\le \mathcal L_c(\OO_2') = \OO_2.$$
    By maximality of $\OO$ in $^K\Xi(X)$ we get that $\OO = \OO_2$.
    It follows that $\OO \in \hphantom{ }^K\WF_c(X)$.
\end{proof}
\begin{corollary}
    Let $(\pi,X)$ be an admissible representation of $\bfG(k)$ and let $c$ be a face of $\mathcal B(\bfG,k)$.
    Then 
    \begin{equation}
        \label{eq:locwfexp}
        ^K\WF_c(X) = \set{\mathcal L_c(\OO):\OO \in \hphantom{ }^{\barF_q}\WF(X^{\bfU_c(\mf o)})}.
    \end{equation}
\end{corollary}
\begin{proof}
    We have that $\check X^{\bfU_c(\mf o)} = \left( X^{\bfU_c(\mf o)}\right)^*$.
    Thus by Proposition \ref{prop:contra}, 
    $$^{\barF_q}\WF(\check X^{\bfU_c(\mf o)}) = \hphantom{ }^{\barF_q}\WF(X^{\bfU_c(\mf o)}).$$
\end{proof}
We will always use the expression in equation \ref{eq:locwfexp} to compute the local wavefront sets .
\begin{corollary}
    Let $(\pi,X)$ be an admissible representation of $\bfG(k)$ and let $(\check \pi,\check X)$ be its contragredient.
    Then 
    $$^K\widetilde\WF(X) = \hphantom{ }^K\widetilde\WF(\check X) \text{ and } \hphantom{ }^{\bar k}\WF(X) = \hphantom{ }^{\bar k}\WF(\check X).$$
\end{corollary}
\begin{proof}
    By Proposition \ref{prop:wfs}, it suffices to show the first equality.
    By Lemma \ref{lem:liftwf} it suffices to show that $^K\WF_c(X) = \hphantom{ }^K\WF_c(\check X)$ for any face $c$ of $\mathcal B(\bfG,k)$.
    This last equality follows from the previous corollary.
\end{proof}
In the following example we use the local wavefront sets to show that the geometric wavefront set of an anti-spherical representation is the regular nilpotent orbit.
\begin{example}
    Let $(\pi,X)$ be a spherical representation of $\bfG(k)$.
    Let $\mathcal A$ be an apartment of $\mathcal B(\bfG,k)$ and $c$ be a hyperspecial face of $\mathcal A$.
    By definition, $V^{\bfU_c(\mf o)}$ contains the trivial $\bfL_c(\mathbb F_q)$-representation.
    \nomenclature{$\ms{AZ}(X)$}{}
    Let $\ms {AZ}(X)$ denote the Aubert--Zelevinsky dual of $X$ \cite[Definition 1.5]{aubertdual}.
    Then $\ms {AZ}(X)^{\bfU_c(\mf o)}$ contains the Steinberg representation.
    By Equation \ref{eq:kawanakawf}, the Kawanaka wavefront set of the Steinberg is $\OO_{reg}$ - the regular nilpotent orbit of $\bfL_c(\mathbb F_q)$.
    This is the unique maximal nilpotent orbit of $\bfL_c(\mathbb F_q)$ and so the Kawanaka wavefront set of $V^{\bfU_c(\mf o)}$ is $\OO_{reg}$.
    Let $\OO = \mathcal L_c(\OO_{reg})$.
    Then $^K\WF_c(\ms {AZ}(X)) = \OO$.
    By Theorem \ref{thm:lift}, $\mathcal N_o(\bar k/K)(\OO)$ is the regular nilpotent orbit of $\mf g$ and so $^{\bar k}\WF(\ms {AZ}(X))$ must also be the regular nilpotent orbit.
\end{example}

\section{Unramified Nilpotent Orbits}
\paragraph{Basic Notation}
\label{par:basicnotation2}
\nomenclature{$\mathcal N_o$}{}
\nomenclature{$T$}{}
\nomenclature{$X^*$}{}
\nomenclature{$X_*$}{}
\nomenclature{$\widetilde W$}{}
\nomenclature{$E$}{}
\nomenclature{$V$}{}
\nomenclature{$\mathbb T$}{}
\nomenclature{$X_S$}{}
Recall from section \ref{par:basicnotation1} the definitions of $\bfG_\Z,\bfT_\Z$ and $\bfT_K$.
Let $G = \bfG_\Z(\C)$, $\mathcal N_o := \mathcal N_o^{\bfG_\Z}(\C)$, and $T = \bfT_\Z(\C)$.
Write $X^*,X_*$ for the common character and co-character lattices of $\bfT_K$ and $\bfT_\Z$, and $\Phi$ for $\Phi(\bfT_K,\bar k)$.
Let 
$$\tilde W=W\ltimes X_*$$
be \emph{the extended affine Weyl group of $\bf G_K$} and let 
$$\cdot:\tilde W\to W, w\mapsto \dot w$$
be the projection map along $X_*$.

Let $E = X^*\otimes_{\Z} \R$, $V = X_*\otimes_{\Z} \R$ and define the compact torus ${\mathbb T} = V/X_*$.
Extend the natural pairing $\langle -,- \rangle$ between $X^*$ and $X_*$ to one between $E$ and $V$.
This allows us to view $X^*$ as linear functions on $V$ via the embedding $\chi \mapsto \langle \chi,-\rangle \in V^*$.
It is moreover clear that these maps descend to continuous maps $V/X_*\to \R/\Z$ and so we obtain a homomorphism
$$X^* \to \hat {\mathbb T} = \Hom_{cts}(\mathbb T,\R/\Z).$$
This map is in fact a $W$ equivariant bijection \cite[Section 9]{sommersmcninch}.
We now define the vanishing set for three different scenarios.
\begin{enumerate}[(1)]
    \item For a subset $S\subseteq T$ define 
    $$X_S^* = \{\chi\in X^*:\chi(s) = 1, \ \forall s\in S\};$$
    \item For a subset $S\subseteq V$, view $X^*$ as linear functions on $V$ and define 
    $$X_S^* = \{\chi\in X^*:\exists n\in \mathbb \Z \text{ such that } \chi(s) = n, \ \forall s\in S\};$$
    \item For a subset $S\subseteq \mathbb T$, view $X^*$ as continuous homomorphisms from $\mathbb T$ to $\R/\Z$ and define 
    $$X_S^* = \{\chi\in X^*:\chi(s) = 0 + \Z, \ \forall s\in S\}.$$
\end{enumerate}
\nomenclature{$\Phi_S$}{}
For $S\subseteq T, V$, or $\mathbb T$, define $\Phi_S := \Phi\cap X_S$.

\nomenclature{$\Psi$}{}
Define the set of \emph{affine roots}
$$\Psi := \{\alpha + n:\alpha\in \Phi\}$$
where $\alpha+n$ denotes the affine function $v\mapsto \alpha(v)+n$ on $V$.
For $a = \alpha+n\in \Psi$ write $\dot a := \alpha$ for the linear part.
The extended affine Weyl group $\tilde W$ acts on $V$ by affine transformations.
This induces an action of $\tilde W$ on the set $\Psi$ via $w:a \mapsto a\circ w^{-1}$.
\nomenclature{$\Psi_S$}{}
For a subset $S\subset V$ write $\Psi_S$ for the set of all affine roots that vanish on $S$.
Note that $\Phi_S = \dot \Psi_S$ and that $\Psi_{wS} = w.\Psi_S$ for all $w\in \tilde W$.

The hyperplanes defined by the $a\in \Psi$ endow $V$ with a chamber complex structure and $\tilde W$ acts by chamber complex automorphisms.
Let $\mathcal A := \mathcal A(\bfT_K,K)$ be the apartment of $\mathcal B(\bfG,K)$ associated to $\bfT_K$.
The apartment $\mathcal A$ is the underlying affine space of the vector space $V$ and so affords an action of $V$ by translations.
\nomenclature{$\kappa_{x_0}$}{}
For this section we will need to make a choice of identification between $V$ and $\mathcal A$ - in particular we must choose a hyperspecial point $x_0\in \mathcal A$ (see \cite[Section 1.10]{tits} for the definition) to send the origin of $V$ to.
Let us fix such a choice $x_0$ and write $\kappa_{x_0}:V\to \mathcal A$ for the resulting identification.
This induces a chamber complex isomorphism between $V$ and $\mathcal A$, identifies $\Psi$ with the affine roots of $\mathcal A$, and fixes an action of $\tilde W$ on $\mathcal A$.

\nomenclature{$\mathcal A$}{}
\nomenclature{$\bar\bfT_K$}{}
\nomenclature{$\Phi_c(\bar\bfT_K,\barF_q)$}{}
For $c$ a face of $\mathcal A$ (resp. $V$) write $\mathcal A(c)$ for $\mathcal A(c,\mathcal A)$ (resp. $\mathcal A(c,V)$).
For a $c\subseteq \mathcal A$ let $W_c$ denote the subgroup of $\tilde W$ generated by the reflections through hyperplanes containing $c$.
The torus $\bfT_K$ is in fact defined over $\mf O$, is a subgroup of $\bfP_c$ for each $c\subseteq \mathcal A$, and the special fibre of $\bfT_K$, denoted $\bar\bfT_K$, is an $\barF_q$-split maximal torus of $\bfL_c(\barF_q)$.
Write $\Phi_c(\bar\bfT_K,\barF_q)$ for the root system of $\bfL_c(\barF_q)$ with respect to $\bar\bfT_K$.
Then $\Phi_c(\bar\bfT_K,\barF_q)$ naturally identifies with the set of $\psi\in\Psi(\bfT,k)$ that vanish on $\mathcal A(c,\mathcal A)$, and the Weyl group of $\bfL_c$ with respect to $\bar \bfT_K$ is naturally isomorphic to $W_c$. 
\nomenclature{$\bfB_K$}{}
\nomenclature{$\Delta$}{}
\nomenclature{$\tilde\Delta$}{}
\nomenclature{$\alpha_0$}{}
Let $\bfB_K$ be a Borel of $\bfG_K$ containing $\bfT_K$ and let $\Delta := \Delta(\bfT_K,\bfB_K)\subseteq \Phi$ be the simple roots determined by $\bfB_K$.
Let $\tilde\Delta$ be the simple roots of $\Psi$ corresponding to $\Delta$.
In particular, when $\Phi(\bfT,K)$ is irreducible, $\tilde\Delta = \Delta \cup \{1-\alpha_0\}$ where $\alpha_0$ is the highest root of $\Phi(\bfT,K)$ with respect to $\Delta$.
In general, for $\Phi(\bfT,K)=\coprod_{i=1}^l\Phi_i$ where each $\Phi_i$ is irreducible and $l\ge 1$, then $\Delta = \coprod_{i=1}^l\Delta_i$ where $\Delta_i = \Delta\cap \Phi_i$, and $\tilde\Delta := \coprod_{i=1}^l \tilde\Delta_i$.
\nomenclature{$c_0$}{}
Let $c_0$ be the chamber of $\mathcal A$ cut out by $\tilde\Delta$.
\nomenclature{$\mathbf P(\tilde\Delta)$}{}
\nomenclature{$\mathsf{Type}_{x_0,\Delta}$}{}
There is a unique labelling function $\mathsf{Type}_{x_0,\Delta}$ of the faces of $\mathcal B(\bfG,K)$ in terms of the set 
$$\mathbf P(\tilde \Delta):=\{J\subsetneq\tilde \Delta: J\cap\Delta_i\subsetneq \Delta_i, 1\le i\le l\},$$ 
such that for $c\subseteq \overline{c_0}$, we have that $\mathcal A(c)$ is the vanishing set of $\mathsf{Type}_{x_0,\Delta}(c)$ (see \cite[Section 5.2]{garrett}).

\paragraph{Pseudo-Levis}
\nomenclature{$N$}{}
\nomenclature{$W$}{}
\nomenclature{$C_G(\bullet)$}{}
\nomenclature{$\Phi_L$}{}
Let $N$ denote the normaliser of $T$ in $G$ and $W$ be the Weyl group $N/T$.
A \emph{pseudo-Levi subgroup $L$} of $G$ is a connected centraliser $\mathbf C_{G}^\circ(s)$ of a semisimple element $s\in G$.
Pseudo-Levi subgroups have the following convenient characterisation.
\begin{proposition}
    \cite[Lemma 14]{sommersmcninch}
    Let $S\subseteq T$ be a subset.
    Then $C_{G}^\circ (S)$ is a reductive subgroup of $G$ and is generated by $T$ together with the root subgroups $\mf X_\alpha$ for $\alpha(s) = 1$ for all $s\in S$.
\end{proposition}
For a pseudo-Levi $L\subset G$ containing $T$, write $\Phi_{L}$ for $\Phi_{Z}$ where ${Z}\subseteq T$ is the center of $L$.

\subsection{Parameterising Unramified Nilpotent Orbits}
\paragraph{Affine Bala--Carter Theory}
\label{par:abc}
\nomenclature{$I_{d,\mathcal A}^K$}{}
\nomenclature{$I_{o,d,\mathcal A}^K$}{}
\nomenclature{$\mathscr N$}{}
\nomenclature{$\sim_{\mathscr N}$}{}
Let 
$$I_{d,\mathcal A}^K = \set{(c,\ms x)\in I_{d}^K:c\subseteq \mathcal A} \text{ and } I_{o,d,\mathcal A}^K = \set{(c,\OO)\in I_{o,d}^K:c\subseteq \mathcal A}.$$
Since $\bfG(K)$ acts transitively on ordered pairs of apartments and chambers the inclusion map $I_{o,d,\mathcal A}^K\to I_{o,d}^K$ descends to a bijection on $\sim_K$ equivalence classes.
Let $\mathscr N$ denote the normaliser of $\bfT_K(K)$ in $\bfG(K)$ (or equivalently the stabiliser of $\mathcal A$ in $\bfG(K)$).
For $(c_1,\OO_1),(c_2,\OO_2)\in I_{o,d,\mathcal A}^K$ declare $(c_1,\OO_1)\sim_{\mathscr N} (c_2,\OO_2)$ if there exists an $\ms n\in \mathscr N$ such that $\mathcal A(c_1) = \mathcal A(\ms nc_2)$ and $\OO_1 = j_{c_1,nc_2}(\ms n.\OO_2)$.
Note that $\ms n.\OO_2$ is a single $\bfL_{\ms nc_2}(\barF_q)$ orbit.
Indeed, if $\ms x_2 \in \OO_2$ then 
$$\ms n.\OO_2 = \ms n \bfL_{c_2}(\barF_q) . \ms x_2 = \ms n \bfL_{c_2}(\barF_q) \ms n^{-1} \ms n. \ms x_2 = \bfL_{\ms n c_2} \ms n. \ms x_2.$$
\begin{proposition}
    The map 
    $$I_{d,\mathcal A}^K\to I_{o,d,\mathcal A}^K, \quad (c,\ms x) \mapsto (c,\bfL_c(\barF_q).\ms x)$$
    descends to a bijection between $I_{d,\mathcal A}^K/\sim_K$ and $I_{o,d,\mathcal A}^K/\sim_{\mathscr N}$.
\end{proposition}
\begin{proof}
    The map $I_{d,\mathcal A}^K\to I_{o,d,\mathcal A}^K$ is clearly surjective.
    Suppose there is a $\ms h\in \bfG(K)$ such that $\mathcal A(c_1) = \mathcal A(\ms hc_2)$ and $\ms x_1 = j_{c_1,\ms hc_2}(\ms h.\ms x_2)$.
    Then we can write $\ms h = \ms n\ms h_0$ where $\ms n\in \mathscr N,\ms h_0\in \bfP_{c_2}(\mf O)$.
    Thus
    \begin{align}
        \bfL_{c_1}(\barF_q).\ms x_1 &= \bfL_{c_1}(\barF_q).j_{c_1,\ms hc_2}(\ms h.\ms x_2) = \bfL_{c_1}(\barF_q)j_{c_1,\ms nc_2}(\ms n\ms h_0.\ms x_2) \\
        &= j_{c_1,\ms nc_2}(\bfL_{\ms nc_2}(\barF_q)\ms n\ms h_0.\ms x_2) = j_{c_1,\ms nc_2}(\ms n(\bfL_{c_2}(\barF_q).(\ms h_0.\ms x_2))).
    \end{align}
    But since $\ms h_0\in \bfP_{c_2}(\mf O)$, we have that $\bfL_{c_2}(\barF_q)(\ms h_0.\ms x_2) = \bfL_{c_2}(\barF_q)\ms x_2$ and so 
    $$(c_1,\bfL_{c_1}(\barF_q).\ms x_1)\sim_{\mathscr N} (c_2,\bfL_{c_2}(\barF_q).\ms x_2).$$
    Thus $I_{d,\mathcal A}^K\to I_{o,d,\mathcal A}^K/\sim_{\mathscr N}$ descends to a well defined map $I_{d,\mathcal A}^K/\sim_K\to I_{o,d,\mathcal A}^K/\sim_{\mathscr N}$.
    Now suppose there is an $\ms n\in \mathscr N$ such that $\mathcal A(c_1) = \mathcal A(\ms nc_2)$ and 
    $$\bfL_{c_1}(\barF_q)\ms x_1 = j_{c_1,\ms nc_2}(\ms n(\bfL_{c_2}(\barF_q)\ms x_2)).$$
    Then there exists an $\ms h_0'\in \bfL_{c_2}(\barF_q)$ such that $\ms x_1 = j_{c_1,\ms nc_2}(^{\ms n\ms h_0'}\ms x_2)$.
    Let $\ms h_0\in \bfP_{c_2}(\mf O)$ be a lift of $\ms h_0'$ and $\ms h = \ms n\ms h_0$. 
    Then $\mathcal A(c_1) = \mathcal A(\ms hc_2)$ and 
    \begin{equation}
        \ms x_1 = j_{c_1,\ms nc_2}(\ms n\ms h_0.\ms x_2) = j_{c_1,\ms hc_2}(\ms h.\ms x_2).
    \end{equation}
    Thus $I_{d,\mathcal A}^K/\sim_K\to I_{o,d,\mathcal A}^K/\sim_{\mathscr N}$ is injective and hence a bijection as required.
\end{proof}
Note that the choice of base point $x_0$ induces a surjection $\pi:\mathscr N\to \tilde W$ with kernel $\bfT_K(\mf O^\times)$.
Since additionally 
$$\bfT_K(\mf O^\times)\subseteq \bigcap_{c\subseteq \mathcal A}\bfP_c(\mf O),$$
if $\ms n_1,\ms n_2\in \mathscr N$ and $\pi(\ms n_1) = \pi(\ms n_2)$, then for any facet $c$ of $\mathcal A$ and nilpotent orbit $\OO\in \mathcal N_o^{\bfL_c}(\barF_q)$, we have $\ms n_1c = \ms n_2c$ and $\ms n_1.\OO = \ms n_2.\OO$.
Thus for $w\in \tilde W$, $c\subseteq \mathcal A$ and $\OO\in \mathcal N_o^{\bfL_c}(\barF_q)$, defining $w.\OO$ to be $\ms n.\OO$ where $\ms n$ is any element of $\mathscr N$ such that $\pi(\ms n) = w$, is well defined.
\nomenclature{$\sim_{\mathcal A}$}{}
Now define a relation $\sim_{\mathcal A}$ on $I_{o,d,\mathcal A}$ by declaring $(c_1,\OO_1)\sim_{\mathcal A}(c_2,\OO_2)$ if there exists $w\in \tilde W$ such that $\mathcal A(c_1) = \mathcal A(wc_2)$ and $\OO_1 = i_{c_1,wc_2}(w.\OO_2)$.
Then clearly $\sim_{\mathscr N}$ and $\sim_{\mathcal A}$ are the same equivalence relation on $I_{o,d,\mathcal A}^K$.

\nomenclature{$\mb{ABC}_{\bfT_K}(\tilde\Delta)$}{}
\nomenclature{$\mathsf{BC}_{x_0,\Delta}(c,\OO)$}{}
\nomenclature{$I_{o,d,c_0}^K$}{}
Let $\mb {ABC}_{\bfT_K}(\tilde\Delta)$ be the set of pairs $(J,J')$ where $J\in \mathbf P(\tilde\Delta)$ and $J' \subseteq J$ is a distinguished subset of $J$ in the sense of Bala--Carter \cite[Section 1]{balacarter}\cite[Section 8.2]{collingwoodmcgovern} (we can think of $J$ as the simple roots of a crystallographic root system).
Given $c\subseteq \mathcal B(\bfG,K)$, $\ms{Type}_{x_0,\Delta}(c)$ is a set of simple roots for $\Phi_c(\bar\bfT_K,\barF_q)$, and so given a distinguished $\OO\in \mathcal N_o^{\bfL_c}(\mathbb F_q)$, let $\ms {BC}_{x_0,\Delta}(c,\OO)$ be the corresponding distinguished subset of $\ms{Type}_{x_0,\Delta}(c)$ prescribed by the Bala--Carter classification of distinguished nilpotent orbits \cite[Section 8]{collingwoodmcgovern}.
Recall that $c_0$ is the chamber of $\mathcal A$ cut out by $\tilde \Delta$.
Let 
$$I_{o,d,c_0}^K = \set{(c,\OO)\in I_{o,d,\mathcal A}^K:c\subseteq \overline{c_0}}.$$
Since $\tilde W$ acts transitively on the chambers in $\mathcal A$, the inclusion map $I_{o,d,c_0}^K\to I_{o,d,\mathcal A}^K$ descends to a bijection.
\nomenclature{$\mb{ABC}_{x_0,\bfT_K}(c,\OO)$}{}
There is a bijection between $I_{o,d,c_0}^K$ and $\mb {ABC}_{\bfT_K}(\tilde\Delta)$ given by 
$$\mb {ABC}_{x_0,\bfT_K}:(c,\OO)\mapsto (\ms{Type}_{x_0,\Delta}(c),\ms {BC}_{x_0,\Delta}(c,\OO)).$$
Write $I_{o,d,c_0}^K:\mb {ABC}_{\bfT_K}(\tilde\Delta)\to I_{o,d,c_0}^K$ for the inverse map.
\nomenclature{$\sim_{\widetilde{W}}$}{}
For $(J_1,J_1'),(J_2,J_2')\in \mb {ABC}_{\bfT_K}(\tilde\Delta)$ define 
$$(J_1,J_1')\sim_{\tilde W}(J_2,J_2') \text{ if there exists } w\in \tilde W: J_2 = w.J_1, J_2' = w.J_1'.$$
\begin{proposition}
    Let $(c_1,\OO_1),(c_2,\OO_2)\in I_{o,d,c_0}^K$.
    Then 
    $$(c_1,\OO_1)\sim_{\mathcal A} (c_2,\OO_2) \text{ if and only if } \mb {ABC}_{x_0,\bfT_K}(c_1,\OO_1)\sim_{\tilde W}\mb{ABC}_{x_0,\bfT_K}(c_2,\OO_2).$$
\end{proposition}
\begin{proof}
    Let $(J_i,J_i') = \mb {ABC}_{x_0,\bfT_K}(c_i,\OO_i)$ for $i=1,2$.
    $(\Rightarrow)$ Suppose there is a $w\in \tilde W$ such that $\mathcal A(c_1) = \mathcal A(wc_2)$ and $\OO_1 = i_{c_1,wc_2}(w.\OO_2)$.
    Note that $J_i$ is a root basis for $\psi_{\mathcal A(c_i)}$.
    Thus since $\mathcal A(c_1) = \mathcal A(wc_2) = w\mathcal A(c_2)$, we have that $J_1$ and $w J_2$ are both root bases for $\psi_{\mathcal A(c_1)}$.
    Thus there exists $w_0\in W_{c_1}$ such that $J_1 = w_0wJ_2$.
    Now, the Bala--Carter data for $i_{c_1,wc_2}(w.\OO_2)$ with respect to $wJ_2$ is $wJ_2'$.
    With respect to the root basis $J_1$ it is thus $w_0wJ_2'$.
    Thus $(J_1,J_1')\sim_{\tilde W} (J_2,J_2')$.
    $(\Leftarrow)$ Suppose there is some $w\in \tilde W$ such that $(J_1,J_1') = w.(J_2,J_2')$.
    Then $J_1 = wJ_2$ implies that $\mathcal A(c_1) = \mathcal A(wc_2)$.
    Moreover, the Bala--Carter data of $i_{c_1,wc_2}(w.\OO_2)$ is $wJ_2'$ with respect to $wJ_2 = J_1$.
    But $w.J_2' = J_1'$ and so $\OO_1$ and $i_{c_1,wc_2}({w}.\OO_2)$ have the same Bala--Carter data and so must be equal.
\end{proof}
\begin{theorem}
    \label{thm:affine_bala_carter}
    \nomenclature{$\mathcal L_{x_0,\bfT_K}$}{}
    The map 
    $$\mathcal L_{x_0,\bfT_K}:\mb {ABC}_{\bfT_K}(\tilde\Delta) \to \mathcal N_o(K), \quad (J,J')\mapsto \mathcal L\circ I_{o,d,c_0}^K(J,J')$$
    descends to a bijection $\mb {ABC}_{\bfT_K}(\tilde\Delta)/\sim_{\tilde W} \ \to \mathcal N_o(K)$.
\end{theorem}
\begin{proof}
    This follows immediately from the previous Proposition.
\end{proof}
\nomenclature{$\mb{BC}_{\bfT_K}(\Delta)$}{}
\nomenclature{$\mb{BC}_{x_0,\bfT_K}(\Delta)$}{}
\nomenclature{$\sim_W$}{}
Recall that regular Bala--Carter theory states that there is a map from the set $\mb {BC}_{\bfT_K}(\Delta)$ of pairs $(J,J')$ where $J\subseteq \Delta$ and $J'$ is distinguished in $J$, to $\mathcal N_o({\bar k})$ that descends to a bijection $\mb{BC}_{x_0,\bfT_K}:\mb {BC}_{\bfT_K}(\Delta)/\sim_W\ \to \mathcal N_o(\bar k)$ (where $\sim_W$ is the obvious analogue of $\sim_{\tilde W}$).
So Theorem \ref{thm:affine_bala_carter} is an affine version of the combinatorial Bala--Carter Theorem in a very literal sense ($\mb {ABC}$ stands for Affine Bala--Carter).
In Proposition \ref{prop:trivccl} we give a precise statement about how these two parameterisations relate.

\paragraph{Properties of Lifting of Nilpotent Orbits}
\nomenclature{$\bfL$}{}
\nomenclature{$\bfl$}{}
\nomenclature{$\bar{\mb t}_K$}{}
\nomenclature{$\bar{\Phi}$}{}
\nomenclature{$\bar{\alpha}$}{}
\nomenclature{$\bar{\omega}$}{}
\nomenclature{$i_{c,\bar\omega}$}{}
\nomenclature{$j_{c,\bar\omega}$}{}
\nomenclature{$j_{c,o}$}{}
\nomenclature{$\Lambda_{x_0}^{\barF_q}$}{}
\nomenclature{$\Lambda_{\bfT_K}^{\bar k}$}{}
Let $\bfL = \bfL_{x_0}$, $\bfl = \bfl_{x_0}$.
Note that $\bfL(\barF_q) = \bfG_\Z(\barF_q)$ since $\bfG_K$ is split.
Let $\bar{\mb t}_K$ be the Lie algebra of $\bar \bfT_K$.
There is a natural identification between the character lattices of $\bfT_K$ and $\bar \bfT_K$.
For $\alpha\in \Phi$ let $\bar \alpha$ denote the image under this identification.
Let $\bar \Phi = \set{\bar \alpha:\alpha\in \Phi}$.
Let $\mf P$ be the maximal ideal of $\mf O$.
Recall that a choice of uniformiser of $K$ induces an isomorphism of additive groups between $\mf P^i/\mf P^{i+1}\to \mf P^j/\mf P^{j+1}$ for any $i,j\in \Z$.
For any facet $c$ of $\mathcal A$ this in turn induces an isomorphism from $\bfL_c(\barF_q)$ onto the pseudo-Levi of $\bfL(\barF_q)$ corresponding to $\bar\bfT_K(\barF_q)$ and $\bar \Phi_c$.
For a uniformiser $\bar \omega$ let $i_{c,\bar\omega}:\bfL_c\to \bfL$ denote the corresponding homomorphism and $j_{c,\bar\omega}:\bfl_c\to \bfl$ the associated morphism of Lie algebras.
One important property of this map is that the following diagram commutes
\begin{equation}
    \begin{tikzcd}
        & \bar\bfT_K \arrow[dl] \arrow[dr] & \\
        \bfl_c \arrow[rr,"j_{c,\bar \omega}"] & & \bfl.
    \end{tikzcd}
\end{equation}
Moreover, since $(\bfl_c)_{\bar \alpha}$ maps to $\bfl_{\bar \alpha}$ the resulting map of nilpotent orbits does not depend on the choice of uniformiser for $K$ (this follows from Bala--Carter theory).
Thus we obtain a canonical map $j_{c,o}:\mathcal N_o^{\bfL_c}(\barF_q) \to \mathcal N_o^{\bfL}(\barF_q)$.
Note since $\bfL$ and $\bfG_\Z$ have the same root data, by Lemma \ref{lem:pom} there is an order preserving isomorphism 
$$\Lambda_{x_0}^{\barF_q}: \mathcal N_o^{\bfL}(\barF_q) \to \mathcal N_o.$$
Let $\Lambda_{\bfT_K}^{\bar k}:\mathcal N_o(\bar k) \to \mathcal N_o$ be the isomorphism from Lemma \ref{lem:pom} applied to $F = \bar k$.
\begin{theorem}
    \label{thm:lift}
    Let $(c,\OO)\in I_{o,\mathcal A}^K$.
    Then 
    $$\Lambda_{\bfT_K}^{\bar k}\circ\mathcal N_o(\bar k/K)\circ \mathcal L_c(\OO) = \Lambda_{x_0}^{\barF_q}\circ j_{c,o}(\OO).$$
\end{theorem}
\begin{proof}
    Let $\ms x,\ms h,\ms y$ be an $\lsl_2$-triple for $\OO$ and let $\ms X,\ms H,\ms Y$ be a lift to an $\lsl_2$-triple for $\mathcal L_c(\OO)$.
    We have that $\alpha(\ms H)\in \Z$ for all $\alpha \in \Phi$.
    There exists a $w\in W$ such that $\alpha(wH)\ge 0$ for all $\alpha\in \Delta$.
    In particular then $\alpha(wH)\in\set{0,1,2}$ for all $\alpha\in\Delta$.
    We also have $j_{c,\bar \omega}(\OO) = \hphantom{ }w j_{c,\bar \omega}(\OO) = j_{wc,\bar \omega}(w.\OO)$ and $w.\mathcal L_c(\OO) = \mathcal L_c(\OO)$.
    Thus by replacing $(c,\OO)$ with $(wc,\hphantom{ }{w}\OO)$ we can assume that $\alpha(H)\in\set{0,1,2}$ for all $\alpha\in \Delta$.
    Then $j_{c,\bar \omega}(x),j_{c,\bar \omega}(h),j_{c,\bar \omega}(y)$ is an $\lsl_2$-triple for $j_{c,o}(\OO)$.
    But since $\bar \alpha(j_{c,\bar \omega}(\ms h)) = \bar\alpha(\ms h)$ equals to the image of $\alpha(\ms H)$ in $\mathbb F_q$ for all $\alpha\in\Delta$, $j_{c,o}(\OO)$ and $\mathcal N_o(\bar k/K)(\mathcal L_c(\OO))$ have the same weighted Dynkin diagram with respect to $\bar \Delta$ and $\Delta$ respectively and so $\Lambda_{\bfT_K}^{\bar k}\circ\mathcal N_o(\bar k/K)\circ\mathcal L_c(\OO) = \Lambda_{x_0}^{\barF_q}\circ j_{c,o}(\OO)$. 
\end{proof}
\begin{corollary}
    \label{cor:alginc}
    Let $(c,\OO_1),(c,\OO_2)\in I_o^K$ and suppose that $\OO_1 < \OO_2$.
    Then 
    $$\mathcal N_o(\bar k/K)(\mathcal L_c(\OO_1)) < \mathcal N_o(\bar k/K)(\mathcal L_c(\OO_2)).$$
\end{corollary}
\begin{proof}
    By \cite[Theorem 5.5]{ggg-np}, if $\OO_1 < \OO_2$, then $j_{c,o}(\OO_1) < j_{c,o}(\OO_2)$.
    Thus 
    $$\mathcal N_o(\bar k/K)(\OO_1) = \Lambda_{x_0}^{\barF_q}\circ j_{c,o}(\OO_1) < \Lambda_{x_0}^{\barF_q}\circ j_{c,o}(\OO_2) = \mathcal N_o(\bar k/K)(\OO_2).$$
\end{proof}
\begin{corollary}
    \label{cor:strictly_inc}
    The map $\mathcal L:I_o^K\to \mathcal N_o(K)$ is strictly increasing.
\end{corollary}
\begin{proof}
    Suppose that $(c,\OO_1') < (c,\OO_2')$.
    Let $\OO_i = \mathcal L_c(\OO_i')$ for $i=1,2$.
    Since $\mathcal L$ is non-decreasing we have that $\OO_1\le \OO_2$.
    By corollary \ref{cor:alginc} we have that $\mathcal N_o(\bar k/K)(\OO_1) \ne \mathcal N_o(\bar k/K)(\OO_2)$.
    Therefore $\OO_1 \ne \OO_2$ and so $\OO_1< \OO_2$ as required.
\end{proof}
\begin{corollary}
    \label{cor:partialorder}
    The closure ordering on $\mathcal N_o(K)$ is a partial order.
\end{corollary}
\begin{proof}
    Let $\OO_1,\OO_2\in \mathcal N_o(K)$ and suppose $\OO_1\le \OO_2 \le \OO_1$.
    By Theorem \ref{thm:closurerels} there exists $\OO_{1.5}\in \mathcal N_o(K)$ with $\OO_1\le \OO_{1.5}\le \OO_2$ such that $\OO_1\le\OO_{1.5}$ is lifted and $\OO_{1.5}\le \OO_2$ is degenerate.
    Since $\mathcal N_o(\bar k/K)(\OO_1) = \mathcal N_o(\bar k/K)(\OO_2)$, by corollary \ref{cor:strictly_inc}, we must have $\OO_1 = \OO_{1.5}$.
    Let $(c,\OO'_{1.5})\in I_o^K(\OO_{1.5})$ be such that $\mathcal C(c,\OO'_{1.5}) \cap \OO_2 \ne \emptyset$ and let $\ms X,\ms H,\ms Y\in \bfP_c(\mf O)$ be an $\lsl_2$-triple for $\OO_{1.5}$.
    By Lemma \ref{lem:unramlift} we have that 
    $$\ms X+\bfu_c(\mf O) = \bfU_c(\mf O).(\ms X+\ms c_{\bfu_c(\mf O)}(\ms Y)).$$
    It follows that 
    $$\ms X+\ms c_{\bfu_c(\mf O)}(\ms Y) \cap \OO_2 \ne \emptyset.$$
    However, 
    $$\ms X+\ms c_{\bfu_c(\mf O)}(\ms Y) \subseteq \ms X+\ms c_{\mf g(\bar k)}(\ms Y)$$ 
    and since $\mathcal N_o(\bar k/K)(\OO_1) = \mathcal N_o(\bar k/K)(\OO_2)$ we have that 
    $$\left(\ms X+\ms c_{\mf g(\bar k)}(\ms Y)\right)\cap \mathcal N_o(\bar k/K)(\OO_2)  = \ms X.$$
    Thus $\left(\ms X+\ms c_{\bfu_c(\mf O)}(\ms Y)\right)\cap \OO_2 = \ms X$ and so $\OO_1 = \OO_2$.
\end{proof}

\paragraph{McNinch and Sommers' Parameterisation of \texorpdfstring{$\mathcal N_{o,c}$}{Noc}}
\label{sec:pseudo-levis}
\nomenclature{$\mathcal N_{o,c}$}{}
\nomenclature{$\mathcal C(G_0)$}{}
For a finite group $G_0$ write $\mathcal C(G_0)$ for the set of conjugacy classes of $G_0$.
For $\OO \in \mathcal N_o$ let $A(\OO)$ denote the component group of the centraliser of an element of $\OO$.
This group is well defined up to inner automorphism of $G$.
Let 
$$\mathcal N_{o,c} = \{(\OO,C):\OO\in \mathcal N_o,C\in \mathcal C(A(\OO))\}.$$
Note that the $c$ here is short for `conjugacy class' and does not refer to a face of the building.
By \cite[Remark 2]{sommersmcninch} this set is in bijection with the set 
$$\{(n,gC_G^\circ(n)):n\in \mathcal N^G,g\in C_G(n)\}/G.$$
This is perhaps a more canonical description of the $N_{o,c}$, but the first description has the benefit of being easier to understand. 
We will pass interchangeably between the two.

\nomenclature{$\mathscr F$}{}
\nomenclature{$I_d^\C$}{}
Let $\mathscr F$ denote the set of pairs $(L,tZ^\circ)$ such that $L$ is a pseudo-Levi of $G$ and $tZ^\circ$ is an element of $Z/Z^\circ$ where $Z$ is the center of ${L}$ and $L = C_{G}^\circ(tZ^\circ)$.
Let $I_d^\C$ denote the set of all triples $({L},tZ^\circ,x)$ such that $(L,tZ^\circ)\in \mathscr F$, and $x$ is a distinguished nilpotent element of $\mf l$, the Lie algebra of $L$.
McNinch and Sommers prove the following result.
\begin{theorem}
    \label{thm:somninch}
    (\cite{sommersmcninch})
    The map $\mathrm{MS}:({L},tZ^\circ,x)\mapsto (x,tC_{G}^\circ(x))$ yields a bijection between $I_d^\C/G$ and $\mathcal N_{o,c}$.
\end{theorem}

\nomenclature{$I_{d,T}^\C$}{}
\nomenclature{$I_{o,d,T}^\C$}{}
\nomenclature{$\mathscr F_T$}{}
Note that every semisimple element can be conjugated to lie in $T$.
Thus if we define $I_{d,T}^\C$ to be the subset of $I_d^\C$ consisting of triples $(L,tZ^\circ,x)$ such that $T\subseteq L$, then the map in Theorem \ref{thm:somninch} descends to a bijection between $I_{d,T}^\C/G$ and $\mathcal N_{o,c}$.
Define $\mathscr F_{T} = \set{(L,tZ^\circ)\in \mathscr F:T\subseteq L}$.
We will additionally find it convenient to work with orbits rather than elements.
Define $I_{o,d,T}^\C$ to be the set of all triples $({L},tZ^\circ,\OO)$ where $({L},tZ^\circ) \in \mathscr F_{T}$, and $\OO$ is a distinguished nilpotent orbit of $\mf l$.
\begin{proposition}
    \label{prop:unip}
    The map $({L},tZ^\circ,x)\mapsto ({L},tZ^\circ,L.x)$ induces a bijection between $I_{d,T}^\C/G$ and $I_{o,d,T}^\C/N$.
\end{proposition}
\begin{proof}
    The induced map $\phi:I_{d,T}^\C\to I_{o,d,T}^\C/N$ is clearly a surjection. 
    Suppose $g\in G$ is such that $\hphantom{ }{g}(L_1,t_1Z_1^\circ,x_1) = (L_2,t_2Z_2^\circ,x_2)$.
    Then $T,\hphantom{ }{g}T$ are both maximal tori of $L_2$ and so there is a $l\in L_2$ so that $\hphantom{ }{lg}T = T$ and so $lg = n$ for some $n \in N$.
    Clearly $\hphantom{ }nL_1 = L_2$ and $\hphantom{ }n(t_1Z_1^\circ) = t_2Z_2^\circ$.
    Also ${n}(\hphantom{ }{L_1}x_1) = \hphantom{ }{L_2n}x_1 = \hphantom{ }{L_2g}x_1 = \hphantom{ }{L_2}x_2$.
    Thus $\phi$ factors through $\mathscr F_{T}/G$.
    Now suppose there exists a $n\in N$ such that ${n}(L_1,t_1Z_1^\circ,\hphantom{ }{L_1}x_1) = (L_2,t_2Z_2^\circ,\hphantom{ }{L_2}x_2)$.
    Then $\hphantom{ }{n}L_1 = L_2$ and $\hphantom{ }{n}\hphantom{ }{L_1}x_1 = \hphantom{ }{L_2}x_2$.
    Thus $\hphantom{ }{L_2n}x_1 = \hphantom{ }{L_2}x_2$ and so there exists an $l\in L_2$ such that $\hphantom{ }{ln}x_1 = x_2$.
    Clearly $\hphantom{ }{ln}(t_1Z_1^\circ) = t_2Z_2^\circ$ since $t_2Z_2^\circ$ lies in the center of $L_2$.
    Thus $\hphantom{ }{ln}(L_1,t_1Z_1^\circ,x_1) = (L_2,t_2Z_2^\circ,x_2)$.
    Thus $\phi$ descends to a bijection as required.
\end{proof}
Note that $N$ stabilises $I_{o,d,T}^\C$ and $T$ acts trivially on $I_{o,d,T}^\C$. 
Thus $W$ acts on $I_{o,d,T}^\C$ and $I_{o,d,T}^\C/W = I_{o,d,T}^\C/N$.
We thus have bijections
\begin{equation}
    \label{eq:Wequivtoccl}
    \begin{tikzcd}
        & I_d^\C/G \arrow[dl,swap,"\sim"] \arrow[dr,"\sim"] & \\
        I_{o,d,T}^\C/W & & \mathcal N_{o,c}.
    \end{tikzcd}
\end{equation}
\nomenclature{$\mathrm{MS}_{o,T}$}{}
Write $\mathrm{MS}_{o,T}$ for the composition $I_{o,d,T}^\C\to I_{o,d,T}^\C/W \xrightarrow{\sim} N_{o,c}$.

\paragraph{From Faces of the Apartment to Pseudo-Levis}
Recall from section \ref{par:basicnotation2} the definitions of $T$ and $\mathbb T$.
We view $T$ as a complex algebraic group with the Zariski topology and $\mathbb T$ as a topological group with the topology induced from the classical topology on $V \cong \R^{\mathrm{rank}(X^*)}$.
Recall the following standard results about closed subgroups of $T$ and $\mathbb T$.
\begin{proposition}
    \label{prop:bij1}
    \cite[Section 16]{humphreys}
    There is a $W$-equivariant bijective correspondence 
    \begin{align}
        \set{\text{closed subgroups } H\le T} &\leftrightarrow \left\{\Z \text{-submodules } M\le X^*\right\} \\
        H &\rightarrow X^*_{H} = \set{\chi\in X^*:\chi(H) = 1} \\
        \set{t\in T: \chi(t) = 1 \ \forall \chi \in M} &\leftarrow M.
    \end{align}
\end{proposition}
\begin{proposition}
    \label{prop:bij2}
    \cite[Chapter 4]{pontryaginduality}
    There is a $W$-equivariant bijective correspondence 
    \begin{align}
        \set{\text{closed subgroups } {\mathbb H} \le {\mathbb T}} &\leftrightarrow \left\{\Z \text{-submodules } M\le \hat{\mathbb T} \right\} \\
        {\mathbb H} &\rightarrow X^*_{\mathbb H} = \set{\chi\in X^*:\chi(\mathbb H) = 0 + \Z} \\
        \set{t\in \mathbb T: \chi(t) = 1 \ \forall \chi \in M} &\leftarrow M.
    \end{align}
    Moreover, for every closed subgroup ${\mathbb H} \le {\mathbb T}$ the connected component of ${\mathbb H}$ containing the identity, $\mathbb H^\circ$, coincides with the annihilator of the torsion elements of the Pontryagin dual of $\mathbb H$.
\end{proposition}

For an affine subspace $A\subseteq V$, call $A$ admissible if it is equal to the vanishing set of a subset of the affine roots $\Psi$.
\begin{proposition}
    \label{prop:pseudolevis}
    For every admissible $A\subseteq V$, there is a pseudo-Levi $L$ with $\Phi_L = \Phi_A$.
\end{proposition}
\begin{proof}
    This follows easily from \cite[Remark 5.2 (b)]{steinberg}.
\end{proof}
\nomenclature{$\mathscr A$}{}
Write $\mathscr A$ for the set of admissible affine subspaces of $V$.
Cocharacters pair integrally with the roots $\Phi$ and so ${X_*}$ acts on the collection of admissible affine subspaces by translation.
For $A\in \mathscr A$ write $[A]$ for the orbit of $A$ in $\mathscr A$ under the action of $X_*$.
For $A\in \mathscr A$ write $L_{A}$ for the pseudo-Levi containing $T$ with $\Phi_{L_A} = \Phi_A$.
For $v\in {X_*}$ and $w\in W$ we have $L_{A+v} = L_A$ and $L_{w.A} = \hphantom{ }w.L_A$ (where $w.L_A$ denotes $wL_Aw^{-1}$).
\begin{remark}
    \label{rmk:roots}
    For a closed subgroup $H\le T$, $X^*_{H^\circ} = \R X^*_{H}\cap X^*$.
    Moreover, for $tH^\circ \in H/H^\circ$ we have $X^*_{H}\subseteq X^*_{tH^\circ} \subseteq X^*_{H^\circ}$ and $\Phi_{H}\subseteq \Phi_{tH^\circ} \subseteq \Phi_{H^\circ}$.
    Identical statements hold for closed subgroups ${\mathbb H}\le {\mathbb T}$.
\end{remark}
\begin{lemma}
    \label{lem:lift}
    Let ${\mathbb H} \le {\mathbb T}$ be a closed subgroup, $\pi:V \to {\mathbb T}$ be the projection map and $z+{\mathbb H}^\circ \in {\mathbb H}/{\mathbb H}^\circ$.
    Then
    \begin{enumerate}
        \item there is an affine subspace $A\subseteq V$ such that $\pi^{-1}(z+{\mathbb H}^\circ) = A+{X_*}$, $X^*_A = X^*_{z+{\mathbb H}^\circ}$ and $\dim A = \rk {\mathbb H}^\circ$.
        \item if $B$ is any other affine subspace with $\pi^{-1}(z+{\mathbb H}^\circ) = B+{X_*}$ then $B\in [A]$, $X^*_B = X^*_{z+{\mathbb H}^\circ}$ and $\dim B = \rk {\mathbb H}^\circ$.
    \end{enumerate}
\end{lemma}
\begin{proof}
    Let $N$ be annihilator of ${\mathbb H}^\circ$ in $X^*$. 
    By Proposition \ref{prop:bij2}, $X^*/N$ is a free $\Z$-module and so the short exact sequence 
    \begin{equation}
        0\longrightarrow N \longrightarrow X^* \longrightarrow X^*/N \longrightarrow 0
    \end{equation}
    splits.
    We can thus find a basis $\chi_1,\dots,\chi_n$ of $X^*$ such that $\chi_1,\dots,\chi_k$ is a basis for $N$ and $\chi_{k+1},\dots,\chi_n$ project onto a basis for $X^*/N$.
    Let $\gamma_1,\dots,\gamma_n$ be the dual basis in ${X_*}$.
    Then ${\mathbb H}^\circ$ is evidently equal to the image of $W = \sum_{i=k+1}^n\R\gamma_i = \cap_{i=1}^k\ker\chi_i$ under $\pi:V\to {\mathbb T}$.
    Let $w\in V$ be such that $\pi(w) \in v + {\mathbb H}^\circ$.
    Then $A = w+W$ maps onto $z+{\mathbb H}^\circ$ and so $\pi^{-1}(z+{\mathbb H}^\circ) = A+{X_*}$.
    It is clear that $\chi\in X^*$ takes integer values on $A$ iff it does so on $A+{X_*}$ iff it does so on $z+{\mathbb H}^\circ$ and so $X^*_A = X^*_{z+{\mathbb H}^\circ}$.
    Also $\dim A = \dim W = \rk X^*/N = \rk{\mathbb H}^\circ$.
    This is establishes (1).
    To establish (2) note that $A+{X_*}$ is a disjoint union of translates of $A$ by elements of the lattice $\bigoplus_{i=1}^k\Z\gamma_i$ and that $\bigoplus_{i=1}^k\R\gamma_i \cap A$ is a single point.
    Thus the connected components of $A+{X_*}$ are exactly $[A]$.
    If $B$ is an affine subspace such that $\pi^{-1}(z+{\mathbb H}^\circ) = B + {X_*}$ then as $B$ is connected it must be contained in some $C\in [A]$.
    The property that $A+{X_*} = B+{X_*}$ then forces $C = B + \bigoplus_{i=k+1}^n\Z\gamma_i$.
    But this can clearly only happen if $B=C$.
    Finally, it is obvious that if $B\in [A]$ then $\dim B = \dim A$ and $X^*_B = X^*_A$.
\end{proof}
\begin{proposition}
    \label{prop:adm}
    Let $L$ be a pseudo-Levi with center ${Z}$.
    There is a canonical bijection
    \begin{align}
        \label{eq:Zbij}
        {Z}/Z^\circ & \leftrightarrow \set{A\in \mathscr A: \Phi_{Z} \subseteq \Phi_A \subseteq \Phi_{Z^\circ}}/{X_*}
    \end{align}
    with the property that if $[A]$ is the image of $tZ^\circ$ then $\Phi_A = \Phi_{C_{G}^\circ(tZ^\circ)}$.
    In particular, there is a bijection
    \begin{align}
        \set{tZ^\circ\in Z/Z^\circ:C_{G}^\circ(tZ^\circ) = L} &\leftrightarrow \set{A\in \mathscr A: \Phi_A = \Phi_{L}}/{X_*}.
    \end{align}
\end{proposition}
\begin{proof}
    Fix a pseudo-Levi $L$ and let $Z$ denote its center.
    The annihilator of $Z$ in $X^*$ is $\Z \Phi_{L}$.
    Write $N$ for the annihilator for $Z^\circ$ in $X^*$.
    Recalling the identification between $X^*$ and $\hat{\mathbb T}$, let $\mathbb H$ denote the annihilator of $\Z\Phi_{L}$ in $\mathbb T$.
    Then both the character group of $Z/Z^\circ$ and the Pontryagin dual of ${\mathbb H}/{\mathbb H}^\circ$ naturally identify with $N/\Z\Phi_{L} = \tor(X^*/\Z\Phi_{L})$ - the torsion subgroup of $X^*/\Z\Phi_{L}$.
    But $Z/Z^\circ$ and ${\mathbb H}/{\mathbb H}^\circ$ are both finite Abelian groups and since the Pontryagin dual and character group (which we denote by $X^*$) coincide for finite groups we obtain a canonical isomorphism $f_{L}$ given by the composition
    \begin{equation}
        \label{eq:canon}
        Z/Z^\circ \to X^*(\tor(X^*/\Z\Phi_{L})) \to (\tor(X^*/\Z\Phi_{L}))^\wedge \to {\mathbb H}/{\mathbb H}^\circ.
    \end{equation}
    \begin{claim}
        The isomorphism $f_{L}$ has the property that for all $tZ^\circ\in Z/Z^\circ$, $X^*_{tZ^\circ} = X^*_{f_{L}(tZ^\circ)}$.
    \end{claim}
    \begin{proof}
        Suppose $\chi\in X^*$ vanishes on $tZ^\circ$.
        Then since $\chi$ is multiplicative, it must vanish on $Z^\circ$ too and hence lie in $N$.
        Its image in $N/\Z\Phi_{L} = \tor(X^*/\Z\Phi_{L}) = X^*(Z/Z^\circ)$ - the character group of $Z/Z^\circ$ - must thus lie in $X^*(Z/Z^\circ)_{tZ^\circ}$.
        Conversely any lift of an element $\bar \chi \in X^*(Z/Z^\circ)_{tZ^\circ}$ to $X^*$ clearly lies in $X^*_{tZ^\circ}$.
        Thus $X^*_{tZ^\circ}$ is equal to all possible lifts of elements in $X^*(Z/Z^\circ)_{tZ^\circ}$.
        We may similarly characterise $X^*_{f_{L}(tZ^\circ)}$ and so we obtain the desired result.
    \end{proof}
    Now fix a $tZ^\circ$, let $v+{\mathbb H}^\circ = f_{L}(tZ^\circ)$ and let $\pi:V\to {\mathbb T}$ denote the projection map.
    By Lemma \ref{lem:lift} we may assign to $z+{\mathbb H}^\circ$ a well defined class $[A]$ with the property that for any $B\in [A]$ we have $\pi^{-1}(z+{\mathbb H}^\circ) = B+{X_*}$, $\dim B = \rk {\mathbb H}^\circ$ and $X^*_B = X^*_{z+{\mathbb H}^\circ}$.
    Define $\lambda_{L}(tZ^\circ)$ to be $[A]$
    \begin{claim}
        $[A]$ is an admissible class with $\Phi_{Z}\subseteq \Phi_A \subseteq \Phi_{Z^\circ}$.
    \end{claim}
    \begin{proof}
        The conditions on $\Phi_A$ are evident from the fact that $\Phi_A = \Phi_{z+{\mathbb H}^\circ}$ and remark \ref{rmk:roots}.
        It remains to show that $A$ is admissible.
        We know that $\Phi_{L}\subseteq X^*_{z+{\mathbb H}^\circ} = X^*_A$ and so $A$ lies in the vanishing set of $\set{\alpha-\alpha(A):\alpha\in \Phi_{L}}\subseteq \Psi$.
        It thus suffices to check that $\dim A = \dim V - \dim \R\Phi_{L}$.
        This follows from $\dim A = \rk{\mathbb H}^\circ = \rk X^* - \rk \Z\Phi_{L} = \dim V - \dim \R\Phi_{L}$.
    \end{proof}
    Thus $\lambda_{L}$ gives us a well defined forwards map for equation \ref{eq:Zbij}.
    Moreover $\Phi_A = \Phi_{tZ^\circ} = \Phi_{C_{G}^\circ(tZ^\circ)}$ and so $\lambda_{L}$ has the advertised property.
    It thus remains to show that the map has an inverse.

    Let $[A]$ be an admissible class with $\Phi_{Z}\subseteq \Phi_A \subseteq \Phi_{Z^\circ}$.
    Then $\pi(A)\subseteq {\mathbb H}$ and is connected so $\pi(A) \subseteq z + {\mathbb H}^\circ$ for some $z\in {\mathbb H}$.
    But also
    \begin{equation}
        \Phi_{Z}\subseteq \Phi_A \subseteq \Phi_{Z^\circ} \subseteq X^*_{Z^\circ} = \R X^*_{Z} \cap X^* \subseteq \R X^*_{Z} = \R \Phi_{Z}
    \end{equation}
    and so $\R\Phi_A = \R\Phi_{Z}$.
    Thus, since $A$ is admissible, $\dim A = \dim V - \dim \R\Phi_A = \dim V - \dim \R\Phi_{Z} = \rk{\mathbb H}^\circ$.
    We must therefore have that $\pi(A) = z+{\mathbb H}^\circ$.
    Applying $f_{L}^{-1}$ we obtain a backwards map which is clearly inverse to $\lambda_{L}$.
\end{proof}
\begin{remark}
    \label{rmk:cs}
    Let $L_1\subseteq L_2$ be pseudo-Levis containing $T$ with centers $Z_1,Z_2$ respectively.
    Let $\mathbb H_1,\mathbb H_2$ be the corresponding subgroups of $\mathbb T$.
    We have $Z_2\subseteq Z_1$ which induces a map $Z_2/Z_2^\circ \to Z_1/Z_1^\circ$.
    Similarly we have a map $\mathbb H_2/\mathbb H_2^\circ \to \mathbb H_1/\mathbb H_1^\circ$.
    By the naturality of the construction of $f_{L_1},f_{L_2}$ we have that 
    \begin{equation}
        \begin{tikzcd}
            Z_2/Z_2^\circ \arrow[r,"f_{L_2}"] \arrow[d] & \mathbb H_2/\mathbb H_2^\circ \arrow[d] \\
            Z_1/Z_1^\circ \arrow[r,"f_{L_1}"] & \mathbb H_1/\mathbb H_1^\circ
        \end{tikzcd}
    \end{equation}
    commutes.
\end{remark}

\begin{proposition}
    \label{prop:admissiblesubset}
    \nomenclature{$\lambda$}{}
    There is a canonical $W$-equivariant bijection
    \begin{align}
        \mathscr F_{T} \leftrightarrow \mathscr A/{X_*} 
    \end{align}
    where the forwards map is given by $\lambda:(L,tZ^\circ) \mapsto \lambda_{L}(tZ^\circ)$.
\end{proposition}
\begin{proof}
    We show first that $\lambda$ is injective.
    Let $(L_1,t_1Z_1^\circ),(L_2,t_2Z_2^\circ) \in \mathscr F_{T}$ and write $[A_i] = \lambda(L_i,t_iZ_i^\circ)$. 
    Suppose we have $[A_1] = [A_2]$.
    Then 
    \begin{equation}
        \Phi_{L_1} = \Phi_{t_1Z_1^\circ} = \Phi_{A_1} = \Phi_{A_2} = \Phi_{t_2Z_2^\circ} = \Phi_{L_2}
    \end{equation}
    and so $L_1 = L_2 =: L$.
    We then get that $t_1Z_1^\circ = t_2Z_2^\circ$ from the fact that $\lambda_{L}$ is a bijection.
    Now let $[A]$ be an admissible class.
    By Proposition \ref{prop:pseudolevis} there exists a pseudo-Levi $L$ with $\Phi_{L} = \Phi_A$.
    Let $Z$ denote the center of $L$.
    Then by Proposition \ref{prop:adm} there exists an $tZ^\circ$ such that $\lambda(L,tZ^\circ) = [A]$ and so $\lambda$ is surjective.
    It remains to show that $\lambda$ is $W$-equivariant.
    Since the projection map $V\to {\mathbb H}$ is $W$ equivariant is suffices to check that for all $w\in W$, the outer square of the following diagram commutes
    \begin{equation}
        \begin{tikzcd}
            Z/Z^\circ \arrow[rrr,"f_{L}"] \arrow{dr} \arrow[ddd,"\hphantom{x}w(-)"] & & & {\mathbb H}/{\mathbb H}^\circ \arrow[ddd,"\hphantom{x}w(-)"] \\
            & X^*(N/\Z\Phi_{L}) \arrow[r,"="] \arrow[d,"\circ w^{-1}"] & (N/\Z\Phi_{L})^\wedge \arrow[ur] \arrow[d,"\circ w^{-1}"] & \\
            & X^*(wN/w\Z\Phi_{L}) \arrow[r,"="] & (wN/w\Z\Phi_{L})^\wedge \arrow[dr] & \\
            \hphantom{x}wZ/wZ^\circ \arrow[rrr,"f_{wL}"] \arrow[ur] & & & \hphantom{x}w{\mathbb H}/w{\mathbb H}^\circ.
        \end{tikzcd}
    \end{equation}
    But it is clear that all the inner squares commute and so the outer square must do too.
\end{proof}

An important consequence of the proposition is the following corollary.
\begin{corollary}
    \label{cor:mfL}
    \nomenclature{$\mf L$}{}
    Let 
    \begin{equation}
        \mf L:\{c \text{ a face of } V\} \rightarrow \mathscr F_T, \quad c\mapsto \lambda^{-1}(\mathcal A(c)+X_*).
    \end{equation}
    Then $\mf L$ is $W$-equivariant surjection and $c_1,c_2$ lie in the same fibre iff $\mathcal A(c_1)+X_*=\mathcal A(c_2)+X_*$.
    Moreover, if $\mf L(c) =(L,tZ_L^\circ)$ then $L$ is a complex reductive group with the same root datum as $\bfL_{\kappa_{x_0}(c)}(\barF_q)$.
\end{corollary}

\nomenclature{$\tau_{y_0}$}{}
Let us investigate how the bijection from Proposition \ref{prop:admissiblesubset} behaves under shift by a hyperspecial point $y_0\in V$ (i.e. a point $y\in V$ such that $\Phi_{y} = \Phi$).
Let $\tau_{y_0}:\mathscr A/X_*\to \mathscr A/X_*$ be the map induced by translating by $y_0$.
Let $A_0\in\mathscr A$ be the vanishing set of $\Psi_{y_0}$.
Then $y_0\in A_0$ and $\tau_y = \tau_{y_0}$ for any $y\in A_0$ (since $\phi(y-y_0) = 0$ for all $y\in A_0$).
Since $\tau_y$ does not depend on the choice of $y\in A_0$ write $\tau_{A_0}$ for $\tau_y$ where $y\in A_0$.
Since $\Phi_{A_0} = \Phi$ we have that $\lambda^{-1}(A_0+X_*) = (G,t_0Z_G^\circ)$ for some $t_0\in Z_G$.
Since $t_0$ is central in $G$, for any $(L,tZ^\circ)\in \mathscr F_T$, we also have $(L,t_0tZ^\circ)\in \mathscr F_T$ (and this procedure is independent of the choice of coset representative of $t_0Z_G^\circ$ since $Z_G^\circ\subseteq Z^\circ$).
\begin{lemma}
    \label{lem:translation}
    Let $(L,tZ^\circ)\in \mathscr F_T$.
    Then 
    $$\lambda(L,t_0tZ^\circ) = \tau_{A_0}(\lambda(L,tZ^\circ)).$$
\end{lemma}
\begin{proof}
    Since $f_L$ is a group homomorphism we have that $f_L(t_0tZ^\circ) = f_L(t_0Z^\circ)+f_L(tZ^\circ)$.
    Let $\pi:V\to \mathbb T$ be the projection map.
    By definition of $t_0$ we have that $\pi(y_0) \in f_G(t_0Z_G^\circ)$.
    By Remark \ref{rmk:cs} we have that $\pi(y_0)\in f_L(t_0Z^\circ)$.
    Thus $f_L(t_0tZ^\circ) = \pi(y_0) + f_L(tZ^\circ)$ and so the lift of $f_L(t_0tZ^\circ)$ to $\mathscr A/X_*$ is $\tau_{y_0}(\lambda(L,tZ^\circ))$.
\end{proof}

\nomenclature{$\mf L_{x_0}$}{}
Now fix a hyperspecial point $x_0\in \mathcal A$ and let 
$$\mf L_{x_0}:\{c\subseteq \mathcal A\} \to \mathscr F_T$$
be the composition $\mf L\circ \kappa_{x_0}^{-1}$.
\begin{lemma}
    Let $x_0$, $x_0'$ be two hyperspecial points in $\mathcal A$ and $c\subseteq \mathcal A$.
    Let $\mf L_{x_0}(c) = (L,tZ_{L}^\circ)$ and $\mf L_{x_0'}(c) = (L',t'Z_{L'}^\circ)$.
    Then $L = L'$.
\end{lemma}
\begin{proof}
    Let $v = x_0-x_0'$.
    For any $y\in \mathcal A$ we have
    $$\kappa_{x_0}^{-1}(y)+v = \kappa_{x_0'}^{-1}(y).$$
    Therefore we have that 
    $$\mathcal A(\kappa_{x_0}^{-1}(c)) + v = \mathcal A(\kappa_{x_0'}^{-1}(c)).$$
    But since $x_0,x_0'$ are both hyperspecial we have that $\alpha(v)\in \Z$ for all $\alpha\in \Phi$.
    Therefore 
    $$\Phi_{\mathcal A(\kappa_{x_0}^{-1}(c))} = \Phi_{\mathcal A(\kappa_{x_0'}^{-1}(c))}$$
    and so $L=L'$ as required.
\end{proof}
For $c\subseteq \mathcal A$ write $L_c$ for $\pr_1\circ\mf L_{x_0}$.
By the above lemma $\pr_1\circ\mf L_{x_0}$ is independent of the choice of $x_0$ which is why we omit it from the notation for $L_c$.
\nomenclature{$\Lambda_c^{\barF_q}$}{}
Since $\bfL_c(\barF_q)$ and $L_c$ have the same root data, by Lemma \ref{lem:pom}, there is an isomorphism of partial orders 
$$\Lambda_c^{\barF_q}:\mathcal N_o^{\bfL_c}(\barF_q)\to \mathcal N_o^{L_c}(\C).$$
Moreover for $c\subseteq \mathcal A, \OO\in \mathcal N_o^{\bfL_c}(\barF_q)$ 
\begin{equation}
    \label{eq:sat}
    \Lambda_{x_0}^{\barF_q}\circ j_{c,o}(\OO) = G. \Lambda_c^{\barF_q}(\OO)
\end{equation}
since saturation of nilpotent orbits can be computed purely in terms of the weighted Dynkin diagram.

\paragraph{A Parameterisation of Unramified Nilpotent Orbits}
\begin{theorem}
    \label{thm:debackbij}
    \nomenclature{$\Gamma_{x_0}$}{}
    \nomenclature{$\tilde\Gamma_{x_0}$}{}
    Fix a hyperspecial point $x_0\in \mathcal A$.
    The map 
    $$\Gamma_{x_0}:I_{o,d,\mathcal A}^K\to I_{o,d,T}^\C, \quad (c,\OO)\mapsto (\mf L_{x_0}(c),\Lambda_c^{\barF_q}(\OO))$$
    induces a bijection
    \begin{equation}
        \tilde\Gamma_{x_0}:I_{o,d,\mathcal A}^K/\sim_{\mathcal A} \to I_{o,d,T}^\C/W.
    \end{equation}
\end{theorem}
\begin{proof}
    The map $I_{o,d,\mathcal A}^K\to I_{o,d,T}^\C$ is clearly a surjection.
    Let 
    $$(c_1,\OO_1),(c_2,\OO_2)\in I_{o,d,\mathcal A}^K$$
    and suppose that $(c_1,\OO_1) \sim_{\mathcal A} (c_2,\OO_2)$.
    Then there exists a $w\in \tilde W$ such that 
    $$\mathcal A(c_1) = \mathcal A(wc_2) \text{ and } \OO_1 = j_{c_1,wc_2}(w.\OO_2).$$
    By corollary \ref{cor:mfL}, we have that $\mf L_{x_0}(c_1) = \dot w.\mf L_{x_0}(c_2)$.
    Moreover
    $$\Lambda_{c_1}^{\barF_q}(\OO_1) = \Lambda_{c_1}^{\barF_q}(j_{c_1,wc_2}(w.\OO_2)) = \Lambda_{wc_2}^{\barF_q}(w.\OO_2) = \dot w\Lambda_{c_2}^{\barF_q}(\OO_2).$$
    Thus 
    $$(\mf L_{x_0}(c_1),\Lambda_c^{\barF_q}(\OO_1)) = \dot w.(\mf L_{x_0}(c_2),\Lambda_c^{\barF_q}(\OO_2))$$
    and so $I_{o,d,\mathcal A}^K \to I_{o,d,T}^\C/W$ descends to a map 
    $$I_{o,d,\mathcal A}^K/\sim_{\mathcal A} \to I_{o,d,T}^\C/W.$$
    To see that this map is injective suppose that $(c_1,\OO_1),(c_2,\OO_2)\in I_{o,d,\mathcal A}^K$ are such that there is a $w_0\in W$ with 
    $$(\mf L_{x_0}(c_1),\Lambda_c^{\barF_q}(\OO_1)) = w_0.(\mf L_{x_0}(c_2),\Lambda_c^{\barF_q}(\OO_2)).$$
    Since $\mf L_{x_0}(c_1) = w_0.\mf L_{x_0}(c_2) = \mf L_{x_0}(w_0c_2)$ we have that 
    $$\mathcal A(c_1) + X_* = \mathcal A(w_0c_2)+X_*$$
    and so there is a $w\in \tilde W$ such that $\dot w = w_0$, and $A(c_1) = A(wc_2)$.
    Moreover
    $$\Lambda_{c_1}^{\barF_q}(\OO_1) = \dot w\Lambda_{c_2}^{\barF_q}(\OO_2) = \Lambda_{wc_2}^{\barF_q}(w.\OO_2) = \Lambda_{c_1}^{\barF_q}(j_{c_1,wc_2}(w.\OO_2))$$
    and so $\OO_1 = j_{c_1,wc_2}(w.\OO_2)$.
    Thus $(c_1,\OO_1)\sim_{\mathcal A}(c_2,\OO_2)$ as required.
\end{proof}

\nomenclature{$\pr_1$}{}
\nomenclature{$\theta_{x_0,\bfT_K}$}{}
Let $\pr_1:\mathcal N_{o,c}\to \mathcal N_o$ be the projection onto the first factor.
Let $\theta_{x_0,\bfT_K}:\mathcal N_o(K)\to \mathcal N_{o,c}$ be the composition
\begin{equation}
    \begin{tikzcd}
        \mathcal N_o(K) \arrow[r,"\sim"] & I_{o,d,\mathcal A}^K/\sim_{\mathcal A} \arrow[r,"\tilde\Gamma_{x_0}"] & I_{o,d,T}^\C/W \arrow[r,"\sim"] & \mathcal N_{o,c}.
    \end{tikzcd}
\end{equation}
where the first isomorphism is the inverse of $\mathcal L:I_{o,d,\mathcal A}^K/\sim_{\mathscr N}\to \mathcal N_o(K)$ (see Section \ref{par:abc}) and the third isomorphism is the composition of the two isomorphisms in Equation \ref{eq:Wequivtoccl}.
\begin{theorem}
    \label{thm:unramifiedparam}
    The map $\theta_{x_0,\bfT_K}:\mathcal N_o(K) \to \mathcal N_{o,c}$ is a bijection and for all $\OO\in \mathcal N_o(K)$ 
    $$\Lambda_{\bfT_K}^{\bar k}(\mathcal N_o(\bar k/K)(\OO)) = \pr_1(\theta_{x_0,\bfT_K}(\OO)).$$
\end{theorem}
\begin{proof}
    Since all the maps involved are bijections, $\theta_{x_0,\bfT_K}$ is also a bijection.
    Let $\OO\in \mathcal N_o(K)$ and $(c,\OO')\in I_{o,d,\mathcal A}^K(\OO)$.
    By theorem \ref{thm:lift}
    $$\Lambda_{\bfT_K}^{\bar k}\circ\mathcal N_o(\bar k/K)(\OO) = \Lambda_{x_0}^{\barF_q}\circ j_{c,o}(\OO').$$
    By equation \ref{eq:sat}
    $$\Lambda_{x_0}^{\barF_q}\circ j_{c,o}(\OO') = G.\Lambda_c^{\barF_q}(\OO').$$
    But 
    $$ G.\Lambda_c^{\barF_q}(\OO') = \pr_1(\theta_{x_0,\bfT_K}(\OO))$$
    as required.
\end{proof}
\nomenclature{$\mathscr H$}{}
\nomenclature{$c(x)$}{}
Let $\mathscr H$ denote the set of hyperspecial faces of $\mathcal B(\bfG,K)$ (these are the faces of $\mathcal B(\bfG,K)$ which contain hyperspecial points).
For a point $x\in \mathcal B(\bfG,K)$ write $c(x)$ for the face containing $x$.
\begin{lemma}
    Suppose $x_0,x_0'$ are hyperspecial points of $\mathcal A$ such that $c(x_0) = c(x_0')$.
    Then $\theta_{x_0,\bfT_K} = \theta_{x_0',\bfT_K}$.
\end{lemma}
\begin{proof}
    It suffices to show that $\mf L_{x_0} = \mf L_{x_0'}$.
    We have that $\mf L_{x_0} = \mf L\circ \kappa_{x_0}^{-1}$ and $\mf L_{x_0} = \mf L\circ \kappa_{x_0'}^{-1}$.
    But if $c(x_0) = c(x_0')$ then the isomorphism of chamber complexes induced by $\kappa_{x_0}$ and $\kappa_{x_0'}$ are the same.
\end{proof}
In the proof of the above lemma we showed that $\mf L_{x_0}$ only depends on the face of $x_0$.
Thus for $h\in \mathscr H\cap \mathcal A$ we will write $\mf L_h$ for $\mf L_{x_0}$ where $x_0$ is any point in $h$.
Similarly we write $\tilde\Gamma_h$ for $\tilde\Gamma_{x_0}$ where $x_0$ is any point in $h$.
Note that when $\bfG_K$ is semisimple and $x_0$ is hyperspecial, $c(x_0) = x_0$ and so there is no content to the lemma.
\begin{lemma}
    Let $c$ be a face of $\mathcal B(\bfG, K)$, $g\in \bfG(K)$ and suppose $c,gc \subseteq\mathcal A$.
    Then there exists a $n\in \mathscr N$ (the normaliser of $\bfT_K$), such that $gc = nc$.
\end{lemma}
\begin{proof}
    Let $c_1$ be a chamber of $\mathcal A$ with $c$ as a face.
    Then $gc_1$ has $gc$ as a face.
    Let $c_2$ be a chamber of $\mathcal A$ containing $gc$ as a face.
    Then $gc_1$ and $c_2$ both share $gc$ as a face and so there exists a $g_1\in \bfP_{gc}(\mf O)$ such that $g_1gc_1 = c_2$ (and necessarily $g_1gc = gc$).
    Let $g_2 = g_1g$.
    Then $g_2\mathcal A$ and $\mathcal A$ contain $c_2$.
    Since the subgroup of $\bfG(K)$ preserving types (in the sense of \cite[Section 5.5]{garrett}) acts transitively on pairs of apartments and chambers of $\mathcal B(\bfG,K)$ there exists a $g_3$ such that $g_3g_2\mathcal A = \mathcal A$ and $g_3c_2 = c_2$ where $g_3$ fixes $c_2$ point-wise. 
    In particular $g_3gc = gc$.
    Since $g_3g_2 \mathcal A = \mathcal A$ we have that $g_3g_2 \in \mathscr N$.
    Let $n = g_3g_2$.
    Then $nc = g_3g_2c = g_3 g_1 g c = g c$ as required.
\end{proof}
\begin{lemma}
    \label{lem:hyperspecialorbit}
    Suppose $x_0,x_0'$ are hyperspecial points of $\mathcal A$ and there exists a $g\in \bfG(K)$ such that $c(x_0') = gc(x_0)$.
    Then $\theta_{x_0,\bfT_K} = \theta_{x_0',\bfT_K}$.
\end{lemma}
\begin{proof}
    By the previous lemma we may assume that $g\in \mathscr N$.
    Let $h = c(x_0),h'=c(x_0')$.
    We need to show that $\tilde\Gamma_{h} = \tilde\Gamma_{h'}$.
    Let $(c,\OO)\in I^K_{o,d,\mathcal A}$.
    It is sufficient to show that $\mf L_h(c) = \mf L_{h'}(c)$.
    But $\mf L_{h'} = \mf L_{gx_0}$ since $c(gx_0) = gc(x_0) = h'$.
    Finally, since $x_0$ is hyperspecial, $gx_0 = x_0+v$ for some $v\in X_*$ (to see this, note that the orbit of $\mathscr N$ on $x_0$ is $\kappa_{x_0}(\tilde W.0) = \kappa_{x_0}(X_*)$) and by Corollary \ref{cor:mfL} we have $\mf L_{x_0+v} = \mf L_{x_0}$.
    Thus $\mf L_{h} = \mf L_{h'}$ as required.
\end{proof}

\nomenclature{$\mathscr O$}{}
\nomenclature{$\theta_{\mathscr O,\bfT_K}$}{}
\nomenclature{$\beta_{\bfT_K,\bfT_K'}$}{}
Let $\mathscr O \in \bfG(K)\backslash \mathscr H$ be a $\bfG(K)$ orbit on $\mathscr H$.
By the previous lemma, the map $\theta_{x_0,\bfT_K}$ is independent of the choice of $x_0 \in h$ for $h\in \mathscr O\cap \mathcal A$.
Thus we write $\theta_{\mathscr O,\bfT_K}$ for $\theta_{x_0,\bfT_K}$ where $x_0\in h$ for some $h\in \mathscr O\cap \mathcal A$.

Now suppose $\bfT_K,\bfT_K'$ are two maximal $K$-split tori of $\bfG_K$.
Let $g\in \bfG(K)$ be such that $\bfT_K' = g\bfT_K g^{-1}$.
This induces an isomorphism or root data 
\begin{equation}
    \label{eq:rd}
    \mathcal R(\bfG_K,\bfT_K)\xrightarrow{\sim}\mathcal R(\bfG_K,\bfT_K'),
\end{equation}
an isomorphism of complex groups
\begin{equation}
    \label{eq:cg}
    (G,T) \xrightarrow{\sim} (G',T'),
\end{equation}
and an isomorphism 
\begin{equation}
    \label{eq:noc}
    \beta_{\bfT_K,\bfT_K'}:\mathcal N_{o,c}\xrightarrow{\sim }\mathcal N_{o,c}'.
\end{equation}
The isomorphisms in Equations \ref{eq:rd}, \ref{eq:cg} differ by conjugation by an element of $W$ for different choices of $g$ with $\bfT_K' = g\bfT_K g^{-1}$ and so the isomorphism in Equation \ref{eq:noc} is independent of choice of $g$.
\begin{theorem}
    \label{thm:naturality}
    Fix an orbit $\mathscr O\in \bfG(K)\backslash \mathscr H$.
    Let $\bfT_K,\bfT_K'$ be two maximal $K$-split tori of $\bfG_K$.
    Then the diagram
    \begin{equation}
        \begin{tikzcd}[row sep = tiny]
            & \mathcal N_{o,c} \arrow[dd,"\beta_{\bfT_K,\bfT_K'}"] \\
            \mathcal N_o(K) \arrow[ru,"\theta_{\mathscr O,\bfT_K}"] \arrow[rd,swap,"\theta_{\mathscr O,\bfT_K'}"] & \\
            & \mathcal N_{o,c}'
        \end{tikzcd}
    \end{equation}
    commutes.
\end{theorem}
\begin{proof}
    Let $g\in \bfG(K)$ be such that $\bfT_K' = g \bfT_K g^{-1}$ and $\mathcal A = \mathcal A(\bfT_K,K),\mathcal A' = \mathcal A(\bfT_K',K)$.
    Then $\mathcal A = g\mathcal A'$.
    Let $h \in \mathscr O \cap \mathcal A$ and pick $x_0\in h$.
    Then $gx_0\in gh \in \mathscr O\cap \mathcal A'$.
    Let 
    $$\mathcal R(\bfG_K,\bfT_K) = (X^*,\Phi,X_*,\Phi^\vee), \quad \mathcal R(\bfG_K,\bfT_K') = (X'^*,\Phi',X'_*,\Phi'^\vee)$$
    and $(\zeta^*,\zeta_*):(X^*,X_*)\to (X'^*,X'_*)$ be the isomorphisms induced by $\Ad(g)$ (these isomorphisms restrict to isomorphisms $\Phi\to \Phi',\Phi^\vee\to \Phi'^\vee$).
    Let $V=X_*\otimes_\Z\R, V'=X_*'\otimes_\Z\R$ and write $\zeta_{*,\R}:V\to V'$ for $\zeta_*\otimes_\Z\R$.
    Write $\eta:(G,T)\to (G',T')$ for the isomorphism of complex algebraic groups induced from $(\zeta^*,\zeta_*)$.
    We claim that 
    \begin{equation}
        \begin{tikzcd}
            \{c\text{ a face of } V\} \arrow[r,"\mf L"] \arrow[d,"\zeta_{*,\R}"] & \mathscr F_T \arrow[d,"\eta"] \\
            \{c\text{ a face of } V'\} \arrow[r,"\mf L'"] & \mathscr F_{T'}.
        \end{tikzcd}
    \end{equation}
    commutes.
    Note that in this diagram, by $\zeta_{*,\R}$ we mean the map on chamber complexes induced by $\zeta_{*,\R}$, and by $\eta$ we mean the map $\mathscr F_T\to \mathscr F_{T'}$ induced by $\eta$.
    We omit the details since this is a slightly tedious, but standard check.
    We remark simply that the isomorphism in Equation \ref{eq:canon} is equivariant with respect to $\eta$ on the left, and $\zeta_{*,\R}$ (descended to the compact torus) on the right and the commutativity of the square follows essentially from this.
    We have also that the first diagram below commutes and so the second one does too.
    \begin{equation}
        \begin{tikzcd}
            V \arrow[d,"\zeta_{*,\R}"] \arrow[r,"\kappa_{x_0}"] &  \mathcal A \arrow[d,"x\mapsto gx"] \\
            V'  \arrow[r,"\kappa_{gx_0}"] & \mathcal A',
        \end{tikzcd}
        \quad
        \begin{tikzcd}
            \mathcal A \arrow[r,"\mf L_{x_0}"] \arrow[d,"x\mapsto gx"] & \mathscr F_T \arrow[d,"\eta"] \\
            \mathcal A' \arrow[r,"\mf L'_{gx_0}"] & \mathscr F_{T'}.
        \end{tikzcd}
    \end{equation}
    Now for $c\subseteq \mathcal A'$, the root system of $\bfL_c(\barF_q)$ is $\bar \Phi_c$, while the root system for $L_c$ is $\Phi_c$.
    Moreover $\Ad(g)$ induces an isomorphism between $\bfL_c(\barF_q)$ and $\bfL_{gc}(\barF_q)$ which identifies $\bar \Phi_c$ with $\bar \Phi'_{gc}$ according to the identification induced by $\zeta^*$.
    Finally $\zeta^*$ restricts to an isomorphism between $\Phi_c$, the root system of $L_c$, and $\Phi'_{gc}$, the root system of $L'_{gc}$.
    In particular the first diagram below commutes and so the second one does too
    \begin{equation}
        \begin{tikzcd}
            \bar \Phi_{c} \arrow[r] \arrow[d] & \bar \Phi'_{gc} \arrow[d] \\
            \Phi_{c} \arrow[r] & \Phi'_{gc},
        \end{tikzcd}
        \quad
        \begin{tikzcd}
            \mathcal N_o^{\bfL_c}(\barF_q) \arrow[r,"\Ad(g)"] \arrow[d,"\Lambda^{\barF_q}_c"] & \mathcal N_o^{\bfL_{gc}}(\barF_q) \arrow[d,"\Lambda^{\barF_q}_{gc}"] \\
            \mathcal N_o^{L_c}(\C) \arrow[r,"\eta"] & \mathcal N_o^{L'_{gc}}(\C).
        \end{tikzcd}
    \end{equation}
    It follows that 
    \begin{equation}
        \label{eq:2nd}
        \begin{tikzcd}
            I_{o,d,\mathcal A}^K \arrow[r,"\Gamma_{x_0}"] \arrow[d,"\Ad(g)"] & I_{o,d,T}^\C \arrow[d,"\eta"] \\
            I_{o,d,\mathcal A'}^K \arrow[r,"\Gamma'_{gx_0}"] & I_{o,d,T'}^\C 
        \end{tikzcd}
    \end{equation}
    also commutes.
    But if $(c',\OO')\in I_{o,d,\mathcal A'}^K$ is equal to $g.(c,\OO)$ for some $(c,\OO\in I_{o,d,\mathcal A}^K)$ then $(c',\OO')\sim_K (c,\OO)$.
    Thus 
    \begin{equation}
        \label{eq:1st}
        \begin{tikzcd}[row sep=tiny]
            & I_{o,d,\mathcal A}^K/\sim_{\mathcal A} \arrow[dd,"\Ad(g)"]\\
            \mathcal N_o(K) \arrow[<-,ur,"\sim"] \arrow[<-,dr,"\sim"] & \\
            & I_{o,d,\mathcal A'}^K/\sim_{\mathcal A'}.
        \end{tikzcd}
    \end{equation}
    commutes.
    Finally 
    \begin{equation}
        \label{eq:3rd}
        \begin{tikzcd}
            I_{o,d,T}^\C \arrow[r] \arrow[d,"\eta"] & \mathcal N_{o,c} \arrow[d,"\beta_{\bfT_K,\bfT_K'}"] \\
            I_{o,d,T'}^\C \arrow[r] & \mathcal N'_{o,c}
        \end{tikzcd}
    \end{equation}
    trivially commutes.
    Stringing together the diagrams in Equations \ref{eq:1st}, \ref{eq:2nd} and \ref{eq:3rd} yields the desired commutative triangle.
\end{proof}

\paragraph{Equivariance of the parameterisation}
\label{par:equiv}
We now finally investigate how the parameterisations corresponding to different $\mathscr O\in \bfG(K)\backslash \mathscr H$ relate.
Fix a $\bfT_K$, let $\mathcal A = \mathcal A(\bfT_K,K)$ and set $\mathscr H_{\mathcal A} = \mathscr H\cap \mathcal A$.
Let us define an action of $Z_G/Z_G^\circ$ on $\bfG(K)\backslash \mathscr H$.
Let $tZ_G^\circ\in Z_G/Z_G^\circ$ and let $A_0+X_* = \lambda(G,tZ_G^\circ)$.
Let $c\in\mathscr H\cap \mathcal A$.
Note that hyperspecial faces are in fact admissible classes i.e. $\mathcal A(c,\mathcal A) = c$ for hyperspecial faces $c$ (since $\Phi_y$ is constant on $\mathcal A(c,\mathcal A)$).
Thus by the discussion preceding Lemma \ref{lem:translation}, for $y\in A_0+X_*$, the translation $\tau_y:\mathscr A/X_*\to \mathscr A/X_*$ restricts to a map $\mathscr H_{\mathcal A}/X_*\to \mathscr H_{\mathcal A}/X_*$ which does not depend on the choice of $y$.
\nomenclature{$\tau_{tZ_G^\circ}$}{}
\nomenclature{$\mathscr H_{\mathcal A}$}{}
Denote the resulting map 
$$\tau_{tZ_G^\circ}:\mathscr H_{\mathcal A}/X_*\to \mathscr H_{\mathcal A}/X_*.$$
We now define the action of $tZ_G^\circ$ on $\bfG(K)\backslash\mathscr H$ as follows.
Let $\mathscr O\in \bfG(K)\backslash \mathscr H$.
Let $c\in \mathscr O\cap \mathcal A$.
We saw in the proof of Lemma \ref{lem:hyperspecialorbit} that $\mathscr O \cap \mathcal A = c+X_*$.
Let $c'+X_* = \tau_{tZ_G^\circ}(c+X_*)$ and $\mathscr O' = \bfG(K).(c')$.
This is well defined because $\bfG(K).c' = c'+X_*$.
Finally we set 
$$(tZ_G^\circ). \mathscr O := \mathscr O'.$$
\begin{proposition}
    \label{prop:staction}
    The map $(tZ_G^\circ):\bfG(K)\backslash \mathscr H \to \bfG(K)\backslash \mathscr H$ defines a simply transitive action of $Z_G/Z_G^\circ$ on $\bfG(K)\backslash \mathscr H$.
\end{proposition}
\begin{proof}
    The fact that this defines an action follows immediately from the fact that the map in Equation \ref{eq:canon} is a group homomorphism.
    To see that the action is simply transitive, first note that the map 
    $$\mathscr H_{\mathcal A}/X_*\to \bfG(K)\backslash \mathscr H, \quad c+X_*\mapsto \bfG(K).c$$
    is a well defined bijection (it is injective since $\bfG(K).c \cap \mathcal A = c+X_*$ and surjective since $\bfG(K)$ acts transitively on apartments).
    Thus it suffices to show that $Z_G/Z_G^\circ$ acts simply transitively on $\mathscr H_{\mathcal A}/X_*$.
    To see that the action is transitive let $c+X_*,c'+X_*\in \mathscr H_{\mathcal A}$.
    Let $x_0\in c+X_*,x_0'\in c'+X_*$ and $y=x_0-x_0'$.
    Let $A_0 = \Psi_{y_0}$ and $(G,tZ_G^\circ) = \lambda^{-1}(A_0+X_*)$. 
    Then $\tau_y = \tau_{tZ_G^\circ}$ and hence $\tau_{tZ_G^\circ}(c+X_*) = c'+X_*$.
    Finally let $c+X_*\in \mathscr H_{\mathcal A},tZ_G^\circ\in Z_G/Z_G^\circ$ and suppose $\tau_{tZ_G^\circ}(c+X_*) = c+X_*$.
    Let $A_0+X_* = \lambda(G,tZ_G^\circ)$ and $y\in A_0+X_*$. 
    By definition $\tau_{tZ_G^\circ} = \tau_{y}$.
    Let $x_0\in c$ and $A_0'$ be the vanishing set of $\Phi$ is $V$.
    Then $c+X_* = \kappa_{x_0}(A_0'+X_*)$ and so since $\tau_y(c+X_*) = c+X_*$ it follows that $y\in A_0' + X_*$. 
    Thus $A_0+X_* = A_0'+X_*$.
    But $\lambda^{-1}(A_0'+X_*) = (G,Z_G^\circ)$ and so $tZ_G^\circ = Z_G^\circ$.
    Thus the action is simply transitive.
\end{proof}

There is also an action of $Z_G/Z_G^\circ$ on $\mathcal N_{o,c}$ given by 
$$tZ_G^\circ:(n,sC_G^\circ(n))/G\mapsto (n,tsC_G^\circ(n))/G.$$
This is well defined since $t\in Z_G$ and $Z_G^\circ \subseteq C_G^\circ(n)$.
We will write this action as $(\OO,C)\mapsto (\OO,tC)$ when we view $\mathcal N_{o,c}$ as pairs $(\OO,C)$.
\begin{proposition}
    \label{prop:equivariance}
    Let $\mathscr O\in \bfG(K)\backslash \mathscr H$ and $tZ_G^\circ\in Z_G/Z_G^\circ$.
    Then 
    $$\theta_{(tZ_G^\circ).\mathscr O,\bfT_K} = (tZ_G^\circ).\theta_{\mathscr O,\bfT_K}.$$
\end{proposition}
\begin{proof}
    This is an immediate consequence of the definitions and Lemma \ref{lem:translation}.
\end{proof}

\paragraph{Compatibility with the Bala-Carter parameterisation}
\begin{proposition}
    \label{prop:trivccl}
    Let $\OO \in \mathcal N_o$ and let $[(J,J')]_W\in \mb{BC}(\Delta)$ be the Bala--Carter parameter for $\OO$ with respect to $\Delta$.
    Then 
    $$\theta_{x_0,\bfT_K}\circ\mb{ABC}_{x_0,\bfT_K}([(J,J')]_{\tilde W}) = (\OO,1).$$
\end{proposition}
\begin{proof}
    For $I\in \mathbf P(\tilde\Delta)$ let $c(I)$ denote the face of $c_0$ with $\ms{Type}$ equal to $I$.
    Let $\OO\in \mathcal N_o$ and let $[(J,J')]_W$ be the Bala--Carter parameter for $\OO$.
    Then $\OO$ intersects non-trivially with $\mf l_{c(J)}$ and the distinguished orbit $\OO_{J'}$ of $\mf l_{c(J)}$ parameterised by $J'$ lies in this intersection.
    Since $L_{c(J)}$ is a Levi it is equal to $C_{G}^\circ(Z_J^\circ)$ where $Z_J$ is the center of $L_{c(J)}$.
    Thus $(L_{c(J)},Z_J^\circ,\OO_{J'}) \in I_{o,d,T}^\C$ and its image in $\mathcal N_{o,c}$ is $(\OO,1)$.
    We also have that $\lambda(L_{c(J)},Z^\circ_J) = [A]$ where $A$ is the vanishing set of $J$, and $(\Lambda_{c(J)}^{\barF_q})^{-1}(\OO_{J'})$ is the distinguished nilpotent orbit of $\bfl_{c(J)}$ corresponding to $J'$.
    It follows that $\theta^{-1}(\OO,1)$ has affine Bala--Carter parameter $[(J,J')]_{\tilde W}$ as required.
\end{proof}
This result implies that the diagram below commutes.
\begin{equation}
    \begin{tikzcd}
        \mb {BC}_{\bfT_K}(\Delta)/\sim_W \arrow[r,"i"] \arrow[d,"\sim"] & \mb {ABC}_{\bfT_K}(\tilde\Delta)/\sim_{\tilde W} \arrow[d,"\sim"] \arrow[r] & \mb {BC}_{\bfT_K}(\Delta)/\sim_W \arrow[d,"\sim"]\\
        \mathcal N_o(\bar k) \arrow[r,"\OO\mapsto \theta_{x_0,\bfT_K}^{-1}(\Lambda_{\bfT_K}^{\bar k}(\OO){,}1)"] \arrow[rr,bend right=15,"\id",swap] & \mathcal N_o(K) \arrow[r,"\mathcal N_o(\bar k/K)"] & \mathcal N_o(\bar k).
    \end{tikzcd}
\end{equation}
Here the map $i$ is the map induced by the inclusion map 
\begin{equation}
    \mb {BC}_{\bfT_K}(\Delta) \to \mb {ABC}_{\bfT_K}(\tilde\Delta).
\end{equation}

We remark that this composition $\theta_{x_0,\bfT_K}\circ \mb{ABC}_{x_0,\bfT_K}$ is independent of $x_0$ as can be easily seen from the construction of the two maps (indeed the composition can be defined entirely by referring only to the vector space $V$ and not the affine space $\mathcal A$).

\paragraph{Example: \texorpdfstring{$G_2$}{G2}}
Consider the case when $\textbf G$ is $K$-split semisimple of type $G_2$ (there is only one isogeny type).
Let $\Delta=\set{\alpha_1,\alpha_2}$ denote the simple roots where $\alpha_2$ is the short root and let $\alpha_0$ be affine simple root in $\tilde\Delta$. 
The extended Dynkin diagram of $G_2$ is $\dynkin[extended,labels*={\alpha_0,\alpha_1,\alpha_2}]G2$.
The set $\mb {ABC}_{\bfT_K}(\tilde\Delta)$ consists of the pairs 
\begin{align}
    (\emptyset,\emptyset), (\set{\alpha_0},\emptyset), (\set{\alpha_1},\emptyset), (\set{\alpha_2},\emptyset), (\set{\alpha_1,\alpha_2},\emptyset), \nonumber \\
    (\set{\alpha_1,\alpha_2},\set{\alpha_2}), (\set{\alpha_0,\alpha_2},\emptyset), (\set{\alpha_0,\alpha_1},\emptyset).
\end{align}
Under the equivalence relation $\sim_{\tilde W}$, all equivalence classes are singletons except $(\set{\alpha_0},\emptyset)\sim_{\tilde W} (\set{\alpha_1},\emptyset)$ so $\mathcal N_o(K)$ has size 7.
We also obtain 7 nilpotent orbits using the parameterisation in terms of $\mathcal N_{o,c}$.
The only nilpotent orbit with non-trivial $G_\C$-equivariant fundamental group is $G_2(a_1)$.
In this case $A(G_2(a_1)) = S_3$ which has 3 conjugacy classes which we denote by representatives $1,(12),(123)$.
Using theorem \ref{thm:unramifiedparam}, the unramified orbits can be parameterised by 
\begin{align}
    (0,1), (A_1,1), (\tilde A_1,1), (G_2(a_1),1), (G_2(a_1),(12)), (G_2(a_1),(123)), (G_2,1)\in \mathcal N_{o,c}.
\end{align}
We thus have two parameterisations 
$$\mb{ABC}_{x_0,\bfT_K}:\mb{ABC}_{\bfT_K}(\tilde\Delta)\to \mathcal N_o(K), \theta_{x_0,\bfT_K}:\mathcal N_o(K)\to \mathcal N_{o,c}.$$
We now demonstrate how they match up i.e. we describe the composition 
$$\theta_{x_0,\bfT_K}\circ \mb{ABC}_{x_0,\bfT_K}:\mb{ABC}_{\bfT_K}(\tilde\Delta)\xrightarrow{\sim}\mathcal N_{o,c}.$$

We already have from Proposition \ref{prop:trivccl} the following matchings
\begin{align}
    (\emptyset,\emptyset) &\leftrightarrow (1,1), \nonumber\\
    (\set{\alpha_1},\emptyset) &\leftrightarrow (A_1,1) \nonumber\\
    (\set{\alpha_2},\emptyset) &\leftrightarrow (\tilde A_1,1) \nonumber\\
    (\set{\alpha_1,\alpha_2},\set{\alpha_1}) &\leftrightarrow (G_2(a_1),1) \nonumber\\
    (\set{\alpha_1,\alpha_2},\emptyset) &\leftrightarrow (G_2,1)
\end{align}
It remains to determine how to match up $(G_2(a_1),(12))$ and $(G_2(a_1),(123))$ together with $(\set{\alpha_0,\alpha_1},\emptyset)$ and $(\set{\alpha_0,\alpha_2},\emptyset)$.
For this we must look at the map $\lambda$.
Figure \ref{fig:apt} is a diagram of the chamber complex on $V$ together with the coroot lattice (in grey), the coroots of the simple roots and minus the highest root (in blue), and a fundamental domain for the topological torus $V/X_*$ (in red).

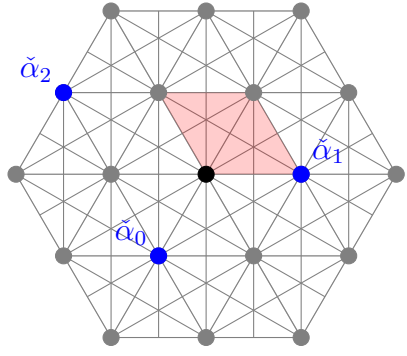
\begin{figure}
    \centering
    \begin{tikzpicture}[baseline=-.5,scale=2.5] 
        \begin{rootSystem}{G}
            \draw[thin,gray] \weight{0}{-1} -- \weight{0}{1};
            \draw[thin,gray] \weight{2}{-2} -- \weight{-2}{2};
            \draw[thin,gray] \weight{3}{-2} -- \weight{-3}{2};
            \draw[thin,gray] \weight{4}{-2} -- \weight{-4}{2};
            \draw[thin,gray] \weight{3}{-1} -- \weight{-3}{1};
            \draw[thin,gray] \weight{2}{0} -- \weight{-2}{0};

            \draw[thin,gray] \weight{-3}{1} -- \weight{0}{1};
            \draw[thin,gray] \weight{3}{-1} -- \weight{0}{-1};
            \draw[thin,gray] \weight{-3}{2} -- \weight{3}{-1};
            \draw[thin,gray] \weight{-3}{1} -- \weight{3}{-2};
            \draw[thin,gray] \weight{-3}{2} -- \weight{0}{-1};
            \draw[thin,gray] \weight{0}{1} -- \weight{3}{-2};

            \draw[thin,gray] \weight{2}{0} -- \weight{-2}{2};
            \draw[thin,gray] \weight{-2}{2} -- \weight{-4}{2};
            \draw[thin,gray] \weight{-4}{2} -- \weight{-2}{0};
            \draw[thin,gray] \weight{-2}{0} -- \weight{2}{-2};
            \draw[thin,gray] \weight{2}{-2} -- \weight{4}{-2};
            \draw[thin,gray] \weight{4}{-2} -- \weight{2}{0};

            \draw[thin,gray] \weight{2.5}{-0.5} -- \weight{-3.5}{1.5};
            \draw[thin,gray] \weight{-2.5}{0.5} -- \weight{3.5}{-1.5};
            \draw[thin,gray] \weight{-1}{1.5} -- \weight{-1}{-0.5};
            \draw[thin,gray] \weight{1}{-1.5} -- \weight{1}{0.5};
            \draw[thin,gray] \weight{-3.5}{2} -- \weight{2.5}{-2};
            \draw[thin,gray] \weight{3.5}{-2} -- \weight{-2.5}{2};

            \draw[thin,gray] \weight{0}{1} -- \weight{-3}{2};
            \draw[thin,gray] \weight{-3}{2} -- \weight{-3}{1};
            \draw[thin,gray] \weight{-3}{1} -- \weight{0}{-1};
            \draw[thin,gray] \weight{0}{-1} -- \weight{3}{-2};
            \draw[thin,gray] \weight{3}{-2} -- \weight{3}{-1};
            \draw[thin,gray] \weight{3}{-1} -- \weight{0}{1};

            \draw[thin,gray] \weight{-2}{2} -- \weight{4}{-2};
            \draw[thin,gray] \weight{2}{0} -- \weight{-4}{2};
            \draw[thin,gray] \weight{-2}{0} -- \weight{-2}{2};
            \draw[thin,gray] \weight{2}{-2} -- \weight{-4}{2};
            \draw[thin,gray] \weight{-2}{0} -- \weight{4}{-2};
            \draw[thin,gray] \weight{2}{0} -- \weight{2}{-2};

            \fill [opacity=0.2,red] \weight{0}{0} -- \weight{1}{0} -- \weight{-1}{1} -- \weight{-2}{1} -- cycle;
            \wt [gray]{1}{0}
            \wt [gray]{0}{1}
            \wt [gray]{-1}{1}
            \wt [gray]{-2}{1}
            \wt [gray]{-3}{1}
            \wt [gray]{-3}{2}

            \wt [gray]{-1}{0}
            \wt [gray]{0}{-1}
            \wt [gray]{1}{-1}
            \wt [gray]{2}{-1}
            \wt [gray]{3}{-1}
            \wt [gray]{3}{-2}

            \wt [gray]{-4}{2}
            \wt [gray]{4}{-2}
            \wt [gray]{-2}{2}
            \wt [gray]{2}{-2}
            \wt [gray]{2}{0}
            \wt [gray]{-2}{0}

            \wt [black]{0}{0}
            \wt [blue]{1}{0}
            \wt [blue]{-3}{1}
            \wt [blue]{1}{-1}
            \node[above right=-0pt] at (hex cs:x=1,y=0) {\(\check\alpha_1\)};
            \node[above left=-0pt] at (hex cs:x=-2,y=1) {\(\check\alpha_2\)};
            \node[above left=-0pt] at (hex cs:x=0,y=-1) {\(\check\alpha_0\)};
        \end{rootSystem}
    \end{tikzpicture}
    \label{fig:apt}
    \caption{The chamber complex for $V$ is displayed alongside the coroot lattice (in grey), the coroots $\check \alpha_0$,$\check \alpha_1$,$\check \alpha_2$ (in blue), and the fundamental domain for the torus $\mathbb T$ (in red).}
\end{figure}
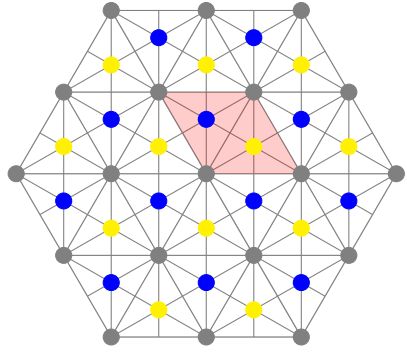
\begin{figure}
    \centering
        \begin{tikzpicture}[scale=2.5] 
            \begin{rootSystem}{G}
                \draw[thin,gray] \weight{0}{1} -- \weight{-3}{2};
                \draw[thin,gray] \weight{-3}{2} -- \weight{-3}{1};
                \draw[thin,gray] \weight{-3}{1} -- \weight{0}{-1};
                \draw[thin,gray] \weight{0}{-1} -- \weight{3}{-2};
                \draw[thin,gray] \weight{3}{-2} -- \weight{3}{-1};
                \draw[thin,gray] \weight{3}{-1} -- \weight{0}{1};

                \draw[thin,gray] \weight{2.5}{-0.5} -- \weight{-3.5}{1.5};
                \draw[thin,gray] \weight{-2.5}{0.5} -- \weight{3.5}{-1.5};
                \draw[thin,gray] \weight{-1}{1.5} -- \weight{-1}{-0.5};
                \draw[thin,gray] \weight{1}{-1.5} -- \weight{1}{0.5};
                \draw[thin,gray] \weight{-3.5}{2} -- \weight{2.5}{-2};
                \draw[thin,gray] \weight{3.5}{-2} -- \weight{-2.5}{2};

                \draw[thin,gray] \weight{2}{0} -- \weight{-2}{2};
                \draw[thin,gray] \weight{-2}{2} -- \weight{-4}{2};
                \draw[thin,gray] \weight{-4}{2} -- \weight{-2}{0};
                \draw[thin,gray] \weight{-2}{0} -- \weight{2}{-2};
                \draw[thin,gray] \weight{2}{-2} -- \weight{4}{-2};
                \draw[thin,gray] \weight{4}{-2} -- \weight{2}{0};

                \draw[thin,gray] \weight{0}{-1} -- \weight{0}{1};
                \draw[thin,gray] \weight{3}{-2} -- \weight{-3}{2};
                \draw[thin,gray] \weight{3}{-1} -- \weight{-3}{1};

                \draw[thin,gray] \weight{-2}{2} -- \weight{4}{-2};
                \draw[thin,gray] \weight{2}{0} -- \weight{-4}{2};
                \draw[thin,gray] \weight{-2}{0} -- \weight{-2}{2};
                \draw[thin,gray] \weight{2}{-2} -- \weight{-4}{2};
                \draw[thin,gray] \weight{-2}{0} -- \weight{4}{-2};
                \draw[thin,gray] \weight{2}{0} -- \weight{2}{-2};

                \draw[thin,gray] \weight{-3}{1} -- \weight{0}{1};
                \draw[thin,gray] \weight{3}{-1} -- \weight{0}{-1};
                \draw[thin,gray] \weight{-3}{2} -- \weight{3}{-1};
                \draw[thin,gray] \weight{-3}{1} -- \weight{3}{-2};
                \draw[thin,gray] \weight{-3}{2} -- \weight{0}{-1};
                \draw[thin,gray] \weight{0}{1} -- \weight{3}{-2};

                \draw[thin,gray] \weight{2}{-2} -- \weight{-2}{2};
                \draw[thin,gray] \weight{4}{-2} -- \weight{-4}{2};
                \draw[thin,gray] \weight{2}{0} -- \weight{-2}{0};

                \fill [opacity=0.2,red] \weight{0}{0} -- \weight{1}{0} -- \weight{-1}{1} -- \weight{-2}{1} -- cycle;
                \foreach \i in {1,2,3}{\wt [yellow]{-\i}{1+1/3}};
                \foreach \i in {1,2,3,4}{\wt [yellow]{2-\i}{0+1/3}};
                \foreach \i in {0,1,2}{\wt [yellow]{\i}{-1+1/3}};
                \foreach \i in {0,1}{\wt [yellow]{2+\i}{-2+1/3}};

                \foreach \i in {0,1}{\wt [blue]{-2-\i}{2-1/3}};
                \foreach \i in {0,1,2}{\wt [blue]{-\i}{1-1/3}};
                \foreach \i in {0,1,2,3}{\wt [blue]{2-\i}{0-1/3}};
                \foreach \i in {1,2,3}{\wt [blue]{\i}{-1-1/3}};

                \foreach \i in {0,1,2}{\wt [gray]{-2-\i}{2}};
                \foreach \i in {0,1,2,3}{\wt [gray]{-\i}{1}};
                \foreach \i in {0,1,2,3,4}{\wt [gray]{2-\i}{0}};
                \foreach \i in {0,1,2,3}{\wt [gray]{\i}{-1}};
                \foreach \i in {0,1,2}{\wt [gray]{2+\i}{-2}};
            \end{rootSystem}
        \end{tikzpicture}
        \caption{The colored disks (grey, blue, yellow) are all the lifts of the vanishing set of $\set{\alpha_0,\alpha_1}$ in $\mathbb T$. Having the same color indicates having the same image in $\mathbb T$.}
        \label{fig:alpha01}
\end{figure}
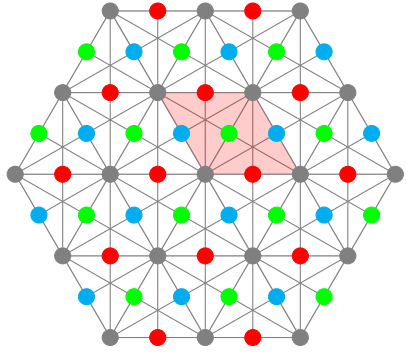
\begin{figure}
    \centering
        \begin{tikzpicture}[scale=2.5] 
            \begin{rootSystem}{G}
                \draw[thin,gray] \weight{0}{1} -- \weight{-3}{2};
                \draw[thin,gray] \weight{-3}{2} -- \weight{-3}{1};
                \draw[thin,gray] \weight{-3}{1} -- \weight{0}{-1};
                \draw[thin,gray] \weight{0}{-1} -- \weight{3}{-2};
                \draw[thin,gray] \weight{3}{-2} -- \weight{3}{-1};
                \draw[thin,gray] \weight{3}{-1} -- \weight{0}{1};

                \draw[thin,gray] \weight{2.5}{-0.5} -- \weight{-3.5}{1.5};
                \draw[thin,gray] \weight{-2.5}{0.5} -- \weight{3.5}{-1.5};
                \draw[thin,gray] \weight{-1}{1.5} -- \weight{-1}{-0.5};
                \draw[thin,gray] \weight{1}{-1.5} -- \weight{1}{0.5};
                \draw[thin,gray] \weight{-3.5}{2} -- \weight{2.5}{-2};
                \draw[thin,gray] \weight{3.5}{-2} -- \weight{-2.5}{2};

                \draw[thin,gray] \weight{2}{0} -- \weight{-2}{2};
                \draw[thin,gray] \weight{-2}{2} -- \weight{-4}{2};
                \draw[thin,gray] \weight{-4}{2} -- \weight{-2}{0};
                \draw[thin,gray] \weight{-2}{0} -- \weight{2}{-2};
                \draw[thin,gray] \weight{2}{-2} -- \weight{4}{-2};
                \draw[thin,gray] \weight{4}{-2} -- \weight{2}{0};

                \draw[thin,gray] \weight{0}{-1} -- \weight{0}{1};
                \draw[thin,gray] \weight{3}{-2} -- \weight{-3}{2};
                \draw[thin,gray] \weight{3}{-1} -- \weight{-3}{1};

                \draw[thin,gray] \weight{-2}{2} -- \weight{4}{-2};
                \draw[thin,gray] \weight{2}{0} -- \weight{-4}{2};
                \draw[thin,gray] \weight{-2}{0} -- \weight{-2}{2};
                \draw[thin,gray] \weight{2}{-2} -- \weight{-4}{2};
                \draw[thin,gray] \weight{-2}{0} -- \weight{4}{-2};
                \draw[thin,gray] \weight{2}{0} -- \weight{2}{-2};

                \draw[thin,gray] \weight{-3}{1} -- \weight{0}{1};
                \draw[thin,gray] \weight{3}{-1} -- \weight{0}{-1};
                \draw[thin,gray] \weight{-3}{2} -- \weight{3}{-1};
                \draw[thin,gray] \weight{-3}{1} -- \weight{3}{-2};
                \draw[thin,gray] \weight{-3}{2} -- \weight{0}{-1};
                \draw[thin,gray] \weight{0}{1} -- \weight{3}{-2};

                \draw[thin,gray] \weight{2}{-2} -- \weight{-2}{2};
                \draw[thin,gray] \weight{4}{-2} -- \weight{-4}{2};
                \draw[thin,gray] \weight{2}{0} -- \weight{-2}{0};

                \fill [opacity=0.2,red] \weight{0}{0} -- \weight{1}{0} -- \weight{-1}{1} -- \weight{-2}{1} -- cycle;

                \foreach \i in {1,2}{\wt [red]{0.5-2-\i}{2}};
                \foreach \i in {1,2,3}{\wt [red]{0.5-\i}{1}};
                \foreach \i in {1,2,3,4}{\wt [red]{0.5+2-\i}{0}};
                \foreach \i in {0,1,2}{\wt [red]{0.5+\i}{-1}};
                \foreach \i in {0,1}{\wt [red]{0.5+2+\i}{-2}};

                \foreach \i in {1,2,3}{\wt [green]{-0.5-\i}{1+0.5}};
                \foreach \i in {1,2,3,4}{\wt [green]{-0.5+2-\i}{0+0.5}};
                \foreach \i in {0,1,2,3}{\wt [green]{-0.5+\i}{-1+0.5}};
                \foreach \i in {0,1,2}{\wt [green]{-0.5+2+\i}{-2+0.5}};

                \foreach \i in {0,1,2}{\wt [cyan]{-1-\i}{1+0.5}};
                \foreach \i in {0,1,2,3}{\wt [cyan]{-1+2-\i}{0+0.5}};
                \foreach \i in {0,1,2,3}{\wt [cyan]{-1+\i}{-1+0.5}};
                \foreach \i in {0,1,2}{\wt [cyan]{-1+2+\i}{-2+0.5}};

                \foreach \i in {0,1,2}{\wt [gray]{-2-\i}{2}};
                \foreach \i in {0,1,2,3}{\wt [gray]{-\i}{1}};
                \foreach \i in {0,1,2,3,4}{\wt [gray]{2-\i}{0}};
                \foreach \i in {0,1,2,3}{\wt [gray]{\i}{-1}};
                \foreach \i in {0,1,2}{\wt [gray]{2+\i}{-2}};
            \end{rootSystem}
        \end{tikzpicture}
        \caption{The colored disks (grey, red, cyan, green) are all the lifts of the vanishing set of $\set{\alpha_0,\alpha_2}$ in $\mathbb T$. Having the same color indicates having the same image in $\mathbb T$.}
        \label{fig:alpha02}
\end{figure}

Figures \ref{fig:alpha01} and \ref{fig:alpha02} respectively show the vanishing sets of $\set{\alpha_0,\alpha_1}$ and $\set{\alpha_0,\alpha_2}$ in $\mathbb T$ (recall the definition from section \ref{par:basicnotation2} (2)) lifted to $V$.
Having the same color indicates having the same image in $\mathbb T$.
Note that the yellow and blue lattices are $W$ conjugate, and the red, cyan and green lattices are also $W$ conjugate.
Thus respectively they give rise to the same element of $\mathcal N_{o,c}$.
It suffices thus to focus on the blue and green lattices respectively.
Clearly the blue lattice gives rise to a conjugacy class of order $3$, while the green one gives rise to a conjugacy class of order $2$.
The remaining matchings are thus
\begin{align}
    (\set{\alpha_0,\alpha_1},\emptyset) &\leftrightarrow (G_2(a_1),(123)), \nonumber\\
    (\set{\alpha_0,\alpha_2},\emptyset) &\leftrightarrow (G_2(a_1),(12)).
\end{align}

\subsection{Equivalence Classes of Unramified Nilpotent Orbits}
\paragraph{A Duality Map for Unramified Nilpotent Orbits}
\label{sec:duality}
\nomenclature{$d_S$}{}
\nomenclature{$d_{S,\mathscr O,\bfT_K}$}{}
\nomenclature{$d_{S,x_0,\bfT_K}$}{}
\nomenclature{$\mathcal N_o^\vee$}{}
\nomenclature{$\Ind_{\bfl_c}^{\mf g^\vee}$}{}
\nomenclature{$d_{\bfl_c}$}{}
\nomenclature{$j_{\dot W_c}^WE$}{}
Fix a torus $\bfT_K$ and corresponding Langlands dual group $\bfG^\vee$ as in Section \ref{par:basicnotation1}.
Write $\mathcal N_o^\vee$ for $\mathcal N_o^{\bfG^\vee}(\C)$ and $\mathcal A$ for $\mathcal A(\bfT_K,K)$.
In \cite[Section 6]{sommersduality}, Sommers defines a duality map
\begin{equation}
    d_S:\mathcal N_{o,c} \to \mathcal N_o^\vee.
\end{equation}
Let $\mathscr O\in \bfG(K)\backslash \mathscr H$ and fix a hyperspecial point $x_0$ in $\mathscr O\cap \mathcal A$ to fix an isomorphism $\mathscr N/\bfT(\mf O^\times)\simeq \tilde W$.
Define
\begin{equation}
    d_{S,\mathscr O,\bfT_K} =  d_S \circ \theta_{\mathscr O,\bfT_K}: \mathcal N_o(K) \to \mathcal N_o^\vee.
\end{equation}
The map $d_S$ is natural under isomorphisms of complex reductive algebraic groups and so $d_{S,\mathscr O,\bfT_K}$ is natural in $\bfT_K$.
We wish to give a recipe for computing $d_{S,\mathscr O,\bfT_K}$ in practice.

Let $c\subseteq \mathcal A$ and let $E$ be a special representation of $\dot W_c$.
Recall from \cite[Theorem 10.7]{lusztig} that the representation $j_{\dot W_c}^WE$ obtained through truncated induction corresponds under the Springer correspondence for $\mf g^\vee(\C)$ to a nilpotent orbit $\OO^\vee\in \mathcal N_o^\vee$ and the trivial local system.
This gives us a map 
$$\Ind_{\bfl_c}^{\mf g^\vee}:\mathcal N_{o,sp}^{\bfL_c}(\barF_q) \to \mathcal N_o^\vee$$
(recall from the end of Section \ref{par:basicnotation1} the definition of $\mathcal N_{o,sp}^{\bfL_c}(\barF_q)$).
Let $d_{\bfl_c}$ denote the Lusztig--Spaltenstein duality 
$$d_{\bfl_c}:\mathcal N_o^{\bfL_c}(\barF_q)\to \mathcal N_{o,sp}^{\bfL_c}(\barF_q).$$
Define
\begin{equation}
    d_{S,x_0,\bfT_K}:I_{o,\mathcal A}^K\to \mathcal N_o^\vee,\quad (c,\OO) \mapsto \Ind_{\bfl_c}^{\mf g^\vee}\circ d_{\bfl_c}(\OO).
\end{equation}
\begin{proposition}
    \label{prop:ds}
    The following diagram commutes
    \begin{equation}
        \begin{tikzcd}[column sep = large]
            I_{o,\mathcal A}^K \arrow[r,"d_{S,x_0,\bfT_K}"] \arrow[d,swap,"\mathcal L"] & \mathcal N_o^\vee \\
            \mathcal N_o(K) \arrow[ur,swap,"d_{S,\mathscr O,\bfT_K}"]
        \end{tikzcd}
    \end{equation}
\end{proposition}
\begin{proof}
    Let $c'\subseteq c$ be faces of $\mathcal A$.
    Then $\bfl_c(\barF_q)$ is the reductive quotient of the parabolic $\bfp_c(\mf O)/\bfu_{c'}(\mf O)$ in $\bfl_{c'}(\barF_q)$.
    For $\OO\in \mathcal N_o^{\bfL_c}(\barF_q)$, write $\Ind_c^{c'}\OO$ for the Lusztig--Spaltenstein induction of $\OO$ from $\bfl_{c}(\barF_q)$ to $\bfl_{c'}(\barF_q)$.
    Let $\ms {Sat}_c^{c'}$ denote the map from $\mathcal N_o^{\bfL_c}(\barF_q)$ to $\mathcal N_o^{\bfL_{c'}}(\barF_q)$ obtained by $\bfL_{c'}(\barF_q)$-saturation (note that although $\bfl_c(\barF_q)$ is naturally only a subquotient of $\bfl_{c'}(\barF_q)$, it is standard to identify it with a Levi factor of $\bfp_c(\mf O)/\bfu_{c'}(\mf O)$ and the saturation map is independent of this choice of Levi factor).
    Recall that 
    $$d_{\bfl_{c'}}\circ\ms{Sat}_c^{c'}\OO = \Ind_c^{c'}\circ d_{\bfl_{c}}\OO.$$
    Then since $\Ind_{\bfl_{c'}}^{\mf g^\vee}\circ\Ind_{\bfl_c}^{\bfl_{c'}} = \Ind_{\bfl_c}^{\mf g^\vee}$ and $\mathcal L_c(\OO) = \mathcal L_{c'}(\ms{Sat}_{c}^{c'}\OO)$ it suffices to show that the diagram commutes when $I_{o,\mathcal A}^K$ is replaced by $I_{o,d,\mathcal A}^{K}$.
    Now consider the following diagram
    \begin{equation}
        \begin{tikzcd}
            & I_{o,d,\mathcal A}^K \arrow[ddl,"\mathcal L"] \arrow[dr,red,"\color{black}\Gamma_{x_0}"] \arrow[dd] \arrow[ddrrr,red,bend left=20,"\color{black}d_{S,x_0,\bfT_K}"] \\
            & & I_{o,d,T}^\C \arrow[dr,red,"\color{black}\mathrm{MS}_{o,T}"] \arrow[d] \\
            \mathcal N_o(K) \arrow[rrr,bend right = 20,swap,"\theta_{\mathscr O,\bfT_K}"] & I_{o,d,\mathcal A}^K/\sim_{\mathcal A} \arrow{l}[swap]{\sim} \arrow{r}{\sim}[swap]{\tilde \Gamma_{x_0}} & I_{o,d,T}^\C/W \arrow[r,"\sim"] & \mathcal N_{o,c} \arrow[r,red,"\color{black}d_S"] & \mathcal N_o^\vee.
        \end{tikzcd}
    \end{equation}
    We wish to show that the outermost square commutes. 
    But we already know all the inner squares and triangles commute except for the square outlined by the red arrows.
    The maps $\Ind_{\bfl_c}^{\mf g^\vee}$ and $d_{\bfl_c}$ can also be defined on $\mathcal N_{o,sp}^{L}(\C)\to \mathcal N_o^\vee$ and $\mathcal N_o^{L}(\C)\to \mathcal N_{o,sp}^{L}(\C)$ where $L$ is a pseudo-Levi of $G$, using using the Springer correspondence over $\C$ instead of $\barF_q$.
    Define 
    $$d_{S,\bfT_K}^\C:I_{o,d,T}^\C\to \mathcal N_{o,c}, \quad (L,tZ_L^\circ,\OO) \mapsto \Ind_{\mf l}^{\mf g^\vee}\circ d_{\mf l}(\OO)$$
    where $\mf l$ is the Lie algebra of $L$.
    Since for good characteristic the Springer correspondence does not depend on the base field, it is clear that $d_{S,x_0,\bfT_K} = d_{S,\bfT_K}^\C \circ \Gamma_{x_0}$.
    It remains to show that $d_{S,\bfT_K}^\C = d_S\circ \mathrm{MS}_{o,T}$.
    This follows from \cite[Proposition 7]{Sommers2001}.
\end{proof}
One remarkable property of $d_S$ is that it does not depend on the isogeny of $G$ (in fact isogeny of semisimple quotient of $G$).
\nomenclature{$\mathcal N_{o,\bar c}$}{}
\nomenclature{$\mf Q$}{}
Let
\begin{equation}
    \mathcal N_{o,\bar c} = \set{(\OO,C):\OO\in \mathcal N_o(\C),C \in \mathcal C(\bar A(\OO))}
\end{equation}
where $\bar A(\OO)$ is Lusztig's canonical quotient of $A(\OO)$ as defined in \cite[Section 5]{sommersduality} and write $\mf Q:\mathcal N_{o,c}\to \mathcal N_{o,\bar c}$ for the map induced by $C\mapsto \bar C$ (where $\bar C$ denotes the image of $C$ in $\bar A(\OO)$).
Then \cite[Theorem 1]{achar} shows that $d_S$ factors through $\mf Q$.
We will call the resulting map from $\mathcal N_{o,\bar c}\to \mathcal N_o^\vee$ also $d_S$.
Let $Z_G/Z_G^\circ$ act trivially on $\mathcal N_{o,\bar c}$.
The argument in \cite[Section 5]{Sommers2001} shows that the map $\mf Q:\mathcal N_{o,c}\to \mathcal N_{o,\bar c}$ is $Z_G/Z_G^\circ$-equivariant.
\begin{proposition}
    \label{prop:indpd}
    Let $\mathscr O,\mathscr O'\in \bfG(K)\backslash \mathscr H$.
    Then $d_{S,\mathscr O,\bfT_K} = d_{S,\mathscr O',\bfT_K}$.
\end{proposition}
\begin{proof}
    By Proposition \ref{prop:staction} there exists a (unique) $tZ_G^\circ$ such that $(tZ_G^\circ).\mathscr O = \mathscr O'$.
    Thus
    \begin{align*}
        d_{S,\mathscr O',\bfT_K} &= d_S \circ \mf Q \circ \theta_{(tZ_G^\circ).\mathscr O,\bfT_K} = d_S \circ \mf Q((tZ_G^\circ).\theta_{\mathscr O,\bfT_K}) \\
        &= d_S\circ \mf Q\circ \theta_{\mathscr O,\bfT_K} = d_{S,\mathscr O,\bfT_K}
    \end{align*}
    where the third equality uses the $Z_G/Z_G^\circ$-equivariance of $\mf Q$.
\end{proof}
\nomenclature{$d_{S,\bfT_K}$}{}
\nomenclature{$\tilde\theta_{\bfT_K}$}{}
Thus we can define $d_{S,\bfT_K}$ to be the map $d_{S,\mathscr O,\bfT_K}$ where $\mathscr O$ is any element of $\bfG(K)\backslash \mathscr H$ and this is well defined.
We will also define $\tilde \theta_{\bfT_K} = \mf Q\circ \theta_{\mathscr O,\bfT_K}$ where $\mathscr O\in \bfG(K)\backslash \mathscr H$.
This is also independent of $\mathscr O$ by the $Z_G/Z_G^\circ$-equivariance of $\mf Q$.

\paragraph{An equivalence relation on $\mathcal N_o(K)$}
\nomenclature{$\le_A$}{}
Recall that in \cite{achar}, Achar defined a pre-order $\le_A$ on $\mathcal N_{o,c}$ by $(\OO_1,C_1)\le_A (\OO_2,C_2)$ if $\OO_1\le \OO_2$ and $d_S(\OO_1,C_1)\ge d_S(\OO_2,C_2)$. 
We can transport it over $\mathcal N_o(K)$ via a $\theta_{\mathscr O,\bfT_K}$ for some $\mathscr O\in \bfG(K)\backslash \mathscr H$ to obtain a pre-order on $\mathcal N_o(K)$.
By Theorem \ref{thm:unramifiedparam}, this pre-order can be expressed as: $\OO_1,\OO_2\in \mathcal N_o(K)$
$$\OO_1\le_A \OO_2 \iff \mathcal N_o(\bar k/K)(\OO_1) \le \mathcal N_o(\bar k/K)(\OO_2) \text{ and } d_{S,\bfT_K}(\OO_1)\ge d_{S,\bfT_K}(\OO_2).$$
By Proposition \ref{prop:indpd} this pre-order is independent of the choice of $\mathscr O$.
By naturality of $d_{S,\mathscr O,\bfT_K}$ in $\bfT_K$, this pre-order also does not depend on the choice of $\bfT_K$.
Now define $\OO_1\sim_A\OO_2$ if $\OO_1\le_A \OO_2$ and $\OO_2\le_A \OO_1$.
\begin{theorem}
    \label{thm:thetabar}
    \nomenclature{$\bar\theta_{\bfT_K}$}{}
    The fibres of the map $\tilde\theta_{\bfT_K}$ are exactly the $\sim_A$-equivalence classes on $\mathcal N_o(K)$.
    In particular $\tilde\theta_{\bfT_K}$ descends to a natural (in $\bfT_K$) bijection 
    $$\bar\theta_{\bfT_K}:\mathcal N_o(K)/\sim_A\to \mathcal N_{o,\bar c}.$$
\end{theorem}
\begin{proof}
    The fibres of $\mf Q$ are precisely the equivalence classes of $\le_A$ on $\mathcal N_{o,c}$ by \cite[Theorem 1]{achar}.
    Naturality in $\bfT_K$ is inherited from naturality of $\theta_{\mathscr O,\bfT_K}$ and $\mf Q$.
\end{proof}
By construction $\le_A$ descends to a partial order on $\mathcal N_o(K)$.
\begin{lemma}
    Let $\OO_1,\OO_2\in \mathcal N_o(K)$.
    If $\OO_1\sim_A\OO_2$ then 
    $$\mathcal N_o(\bar k/K)(\OO_1) = \mathcal N_o(\bar k/K)(\OO_2).$$
\end{lemma}
\begin{proof}
    Obvious.
\end{proof}
The saturation map $\mathcal N_o(\bar k/K)$ is constant on $\sim_A$ classes.
\begin{lemma}
    \label{lem:orderrev}
    Let $(c,\OO_1'),(c,\OO_2')\in I_{o,\mathcal A}^K$ and let $\OO_i = \mathcal L_c(\OO_i')$ for $i=1,2$.
    If $\OO_1'\le \OO_2'$ then $d_{S,\bfT_K}(\OO_2)\le d_{S,\bfT_K}(\OO_1)$.
    In particular $\OO_1\le_A\OO_2$.
\end{lemma}
\begin{proof}
    If $\OO_1'\le \OO_2'$ then $d_{\bfl_c}(\OO_2')\le d_{\bfl_c}(\OO_1')$.
    But $\Ind_{\bfl_c}^{\mf g^\vee}$ is order preserving by \cite[Theorem 12.5]{spaltenstein} and so $d_{S,\bfT_K}(\OO_2)\le d_{S,\bfT_K}(\OO_1)$.
\end{proof}

\paragraph{The Canonical Unramified Wavefront Set}
\nomenclature{$^K\WF(X)$}{}
Let $(\pi,X)$ be a depth-$0$ representation of $\bfG(k)$.
Define the \emph{canonical unramified wavefront set} of $(\pi,X)$ to be
$$^K\WF(X):=\max_{c\subseteq \mathcal B(\bfG,k)}[\mathcal L_c(\hphantom{ }^{\mf F}\WF(X^{\ms U_c}))] \quad (\subseteq \mathcal N_o(K)/\sim_A).$$
The canonical unramified wavefront set has many of the nice properties that $^K\tilde\WF(X)$ has.
\begin{lemma}
    \label{lem:charfreewf}
    Let $(\pi,X)$ be a depth-$0$ representation of $\bfG(k)$.
    Then 
    $$^K\WF(X) = \max [\hphantom{ }^K\Xi(X)].$$
\end{lemma}
\begin{proof}
    The proof is exactly the same as Lemma \ref{lem:liftwf} except we need to be careful for the second part which we now provide the details for.
    Let $\OO$ be a maximal element of $^K\Xi(\pi)$ with respect to $\le_A$.
    Write $\OO = \mathcal N_o(K/k)(\OO_1)$ where $\OO_1\in \Xi(X)$.
    Let $(c,\OO_1')\in I_o^k(\OO_1)$.
    Let 
    $$\OO_2' = \mathcal N_o^{\bfL_c}(\barF_q/\mathbb F_q)(\OO_1').$$
    Let $\OO_3'$ be a wavefront set nilpotent for $V^{\bfU_c(\mf o)}$ such that $\OO_2' \le \OO_3'$.
    Let $\OO_3 = \mathcal L_c(\OO_3')$.
    $\OO_3'$ is an element of $^K\Xi(\pi)$.
    By Lemma \ref{lem:orderrev}, $\OO_2' \le \OO_3'$ implies that $\OO \le_A \OO_3$.
    By maximality of $\OO$ in $^K\Xi(\pi)$ we get that 
    $$\OO\sim_A\OO_3.$$
    In particular
    $$\mathcal N_o(\bar k/K)(\OO) = \mathcal N_o(\bar k/K)(\OO_3).$$
    Thus by corollary \ref{cor:alginc} we have that $\OO_2'= \OO_3'$.
    It follows that $\OO \in \hphantom{ }^K\WF_c(\pi)$.
\end{proof}
\begin{theorem}
    \label{thm:algwf}
    Let $(\pi,X)$ be a depth $0$ representation of $\bfG(k)$.
    Then
    $$^{\bar k}\WF(X) = \max \mathcal N(\bar k /K)(\hphantom{ }^K\WF(X)).$$
\end{theorem}
\begin{proof}
    By Proposition \ref{prop:wfs} and corollary \ref{cor:maxl}, $^{\bar k}\WF(X)$ consists of the maximal elements of $\mathcal N_o(\bar k/K)(\hphantom{ }^K\Xi(X))$.
    Since $\mathcal O_1\le_A \mathcal O_2$ implies that $^{\bar k}\mathcal O_1\le \hphantom{ }^{\bar k}\mathcal O_2$, the results follows by Lemma \ref{lem:charfreewf}.
\end{proof}

\paragraph{Achar's Duality Map}
\label{par:acharduality}
\nomenclature{$\mathcal N_{o,\bar c}^\vee$}{}
\nomenclature{$d_A$}{}
\nomenclature{$d$}{}
\nomenclature{$d_{LS}$}{}
Let $\mathcal N_{o,c}^\vee,\mathcal N_{o,\bar c}^\vee$ be the analogous objects of $\mathcal N_{o,c}$ and $\mathcal N_{o,\bar c}$ for $G^\vee$.
In \cite{achar}, Achar introduces duality maps
\begin{equation}
    \begin{tikzcd}
        \mathcal N_{o,\bar c} \arrow[r,shift left,"d_A"] & \arrow[l,shift left,"d_A"] \mathcal N^\vee_{o,\bar c}.
    \end{tikzcd}
\end{equation}
which extend $d_{S}$ in the sense that the following diagram commutes
\begin{equation}
    \begin{tikzcd}
        \mathcal N_{o,c} \arrow[r,"d_S"] \arrow[d,"\mf Q"] & \mathcal N_o^\vee \\
        \mathcal N_{o,\bar c} \arrow[r,"d_A"] & \mathcal N_{o,\bar c}^\vee. \arrow[u,"\pr_1^\vee"]
    \end{tikzcd}
\end{equation}
Here $\pr_1^\vee:\mathcal N_{o,\bar c}\to\mathcal N_o^\vee$ denotes the projection onto the first factor.

Let $d:\mathcal N_o^\vee\to \mathcal N_o$ denote the Barbasch--Vogan duality \cite{bv} and $d_{LS}:\mathcal N_o\to \mathcal N_o$ denote the Lusztig--Spaltenstein duality \cite{spaltenstein}.

\begin{theorem}
    \label{thm:canoninv}
    \begin{enumerate}
        \item Let $\OO \in \mathcal N_o(K)$. Then $\Lambda_{\bfT_K}^{\bar k}\circ \mathcal N_o(\bar k/K)(\OO) \le d\circ d_{S,\bfT_K}(\OO)$.
        \item Let $\OO^\vee\in \mathcal N_o^\vee$. Then
        \begin{equation}
            \set{(\OO,C)\in \mathcal N_{o,c}:d_S(\OO,C) = \OO^\vee, \OO = d(\OO^\vee)} = \mf Q^{-1}(d_A(\OO^\vee,1)).
        \end{equation} 
        In particular this set is non-empty.
        \item Let $\OO\in \mathcal N_o(K)$ and $\OO^\vee\in \mathcal N_o^\vee$.
        If $d_{S,\bfT_K}(\OO)\ge \OO^\vee$ then $\bar\theta_{\bfT_K}([\OO])\le_A d_A(\OO^\vee,1)$.
    \end{enumerate}
\end{theorem}
\begin{proof}
    In \cite[Proposition 2.3]{achar} Achar proves that $d_S(\OO',C)\le d_S(\OO',1)$ for all $(\OO',C)\in \mathcal N_{o,c}$.
    But $d_S(\OO',1) = d(\OO')$.
    Thus $d\circ d(\OO')\le d\circ d_S(\OO',C)$.
    Since $\OO'\le d_{LS}\circ d_{LS}(\OO') = d\circ d(\OO')$ the first part follows.

    The second part is an immediate consequence of \cite[Remark 14]{sommersduality} and \cite[Proposition 2.8]{achar}.

    For the third part note that by part 1 of this theorem we have that 
    $$\Lambda_{\bfT_K}^{\bar k}\circ\mathcal N_o(\bar k/K)(\OO)\le d(\OO^\vee).$$
    By assumption we also have 
    $$d_S(\OO)\ge d_S(d_A(\OO^\vee,1)) = \OO^\vee.$$
    Since $p_1(d_A(\OO^\vee,1)) = d(\OO^\vee)$, it follows by definition that $\bar \theta([\OO])\le_A d_A(\OO^\vee,1)$.
\end{proof}

When using the duality maps to compute wavefront sets we will often want to consider the following modifications.
Define 
$$d_{A,\bfT_K}^K:\mathcal N_{o,\bar c}^\vee\to \mathcal N_{o}(K)/\sim_A, \quad d_A^K = (\bar\theta_{\bfT_K})^{-1} \circ d_A$$
and
$$d_{\bfT_K}^{\bar k}:\mathcal N_{o}^\vee\to \mathcal N_{o}(\bar k), \quad d^{\bar k} =  (\Lambda_{\bfT_K}^{\bar k})^{-1} \circ d.$$

\section{The Wavefront Set for the Principal Block}
\paragraph{Basic Notation}
\nomenclature{$\mathrm{Rep}(\bfG(k))$}{}
\nomenclature{$\mathcal B^\vee,\mathcal B^\vee_x$}{}
\nomenclature{$H^i(\mathcal B^\vee_S)^\rho$}{}
\nomenclature{$\mathrm{Irr}(A(S))_0$}{}
Let $\mathrm{Rep}(\bfG(k))$ denote the category of smooth representations of $\bfG(k)$.

If $H$ is a complex reductive group and $x$ is an element of $H$ or its Lie algebra $\mf h$, we write $H(x)$ for the centraliser of $x$ in $H$, and $A_H(x)$ for the group of connected components of $H(x)$. If $S$ is a subset of $H$ or $\mf h$ (or indeed, of $H \cup \mf h$), we can similarly define $H(S)$ and $A_H(S)$. We will sometimes write $A(x)$, $A(S)$ when the group $H$ is implicit. 

Write $\mathcal B^\vee$ for the flag variety of $G^\vee$, i.e. the variety of Borel subgroups $B^{\vee} \subset G^{\vee}$. Note that $\mathcal{B}^{\vee}$ has a natural left $G^{\vee}$-action. 
For $g\in G^\vee$, write
$$\mathcal B^\vee_g = \{B^\vee\in \mathcal B^\vee \mid g\in B^\vee \}.$$
(this coincides with the subvariety of Borels fixed by $g$). Similarly, for $x\in \mathfrak g^\vee$, write
$$\mathcal B^\vee_x = \{B^\vee\in \mathcal B^\vee \mid x\in \mathfrak b^\vee \}.$$
If $S$ is a subset of $G^{\vee}$ or $\mf g^{\vee}$ (or indeed of $G^{\vee} \cup \mf g^{\vee}$), write
$$\mathcal B^\vee_S = \bigcap_{x\in S} \mathcal{B}^{\vee}_x.$$
The singular cohomology group $H^i(\mathcal B^\vee_S,\C) = H^i(\mathcal{B}^{\vee}_S)$ carries an action of $A(S)=A_{G^\vee}(S)$. For an irreducible representation $\rho\in\mathrm{Irr}(A(S)))$, let 
$$H^i(\mathcal B^\vee_S)^\rho := \Hom_{A(S)}(\rho,H^i(\mathcal{B}^{\vee}_S)).$$ 
We will often consider the subset
\begin{equation}\label{eq:defofIrr0}
    \mathrm{Irr}(A(S))_0 := \{\rho \in \mathrm{Irr}(A(S)) \mid H^{\mathrm{top}}(\mathcal{B}_S^{\vee})^{\rho} \neq 0\}.
\end{equation}

\subsection{Local Wavefront Sets for the Principal Block}
\paragraph{Representations with Unipotent Cuspidal Support}
\label{s:unip-cusp}
\nomenclature{$\Pi^{Lus}(\bfG(k))$}{}
\begin{definition} Let $X$ be an irreducible smooth $\mathbf{G}({k})$-representation. We say that $X$ has \emph{unipotent cuspidal support} if there is a parahoric subgroup $\bfP_c \subset \bfG$ such that $X^{\mathbf U_c(\mf o)}$ contains an irreducible Deligne--Lusztig cuspidal unipotent representation of ${\mathbf L}_c(\mathbb F_q)$. Write $\Pi^{{Lus}}(\bfG( k))$ for the set of all such representations.
\end{definition}

\nomenclature{$\mathsf L(\bfG(k))$}{}
\nomenclature{$\mf g_q^\vee$}{}
\nomenclature{$s_0$}{}
Recall that an irreducible $\mathbf{G}({k})$-representation $V$ is \emph{Iwahori spherical} if $V^{\mathbf{I}(\mathfrak {o})} \neq 0$ for an Iwahori subgroup $\bfI(\mf o)$ of $\bfG(k)$. We note that all such representations have unipotent cuspidal support, corresponding to the case $\mathbf P_c=\mathbf I$ and the trivial representation of $\mathbf T(\mathbb F_q)$.

We will now recall the classification of irreducible representations of unipotent cuspidal support. Write $\ms L(\bfG( k))$ for the set of $G^\vee$-orbits (under conjugation) of triples $(s,n,\rho)$ such that
\begin{itemize}
    \item $s\in G^\vee$ is semisimple,
    \item $n\in \mathfrak g^\vee$ such that $\operatorname{Ad}(s) n=q n$,
    \item $\rho\in \mathrm{Irr}(A_{G^{\vee}}(s,n))$ such that $\rho|_{Z(G^\vee)}$ acts as the identity.
\end{itemize}
Without loss of generality, we may assume that $s\in T^\vee$. Note that $n\in\mathfrak g^\vee$ is necessarily nilpotent. The group $G^\vee(s)$ acts with finitely many orbits on the $q$-eigenspace of $\Ad(s)$
$$\mathfrak g_q^\vee=\{x\in\mathfrak g^\vee\mid \operatorname{Ad}(s) x=qx\}$$
In particular, there is a unique open $G^\vee(s)$-orbit in $\mathfrak g_q^\vee$.

Fix an $\mathfrak{sl}(2)$-triple $\{n^-,h,n\} \subset \mf g^{\vee}$ with $h\in \mathfrak t^\vee_{\mathbb R}$ and set
\begin{equation}
    \label{eq:s0}
    s_0:=sq^{-\frac{h}{2}}.
\end{equation}
Then $\operatorname{Ad}(s_0)n=n$.
Recall the definition of $T^\vee_c$ from section \ref{par:basicnotation1}.

\begin{theorem}[{Local Langlands correspondence, \cite[Theorem 7.12]{kl}\cite[Theorem 5.21]{Lu-unip1}}]\label{thm:Langlands} Suppose that $\bfG$ is adjoint and $ k$-split. There is a natural bijection
$$\ms L(\bfG( k))\xrightarrow{\sim} \Pi^{{Lus}}(\bfG( k)),\qquad (s,n,\rho)\mapsto X(s,n,\rho),$$
such that
\begin{enumerate}
    \item $X(s,n,\rho)$ is tempered if and only if $s_0\in T_c^\vee$ and $\overline {G^\vee(s)n}=\mathfrak g_q^\vee$,
    \item $X(s,n,\rho)$ is square integrable (modulo the center) if and only if it is tempered and $A_{G^{\vee}}(s,n)$ contains no nontrivial torus.
\item $X(s,n,\rho)^{\mathbf I(\mf o)}\neq 0$ if and only if $\rho\in \mathrm{Irr}(A(s,n))_0$, i.e., $\rho$ is such that 
$$H^{\mathrm{top}}(\mathcal B^\vee_{s,n})^\rho\neq 0.$$
\end{enumerate}
\end{theorem}

For each parameter $(s,n,\rho) \in \ms L(\mathbf{G}({k}))$, there is an associated \emph{standard representation} $Y(s,n,\rho) \in \mathrm{Rep}(\mathbf{G}({k}))$.
For Iwahori spherical representations, the relevant results are \cite[Theorems 7.12, 8.2, 8.3]{kl}. For the more general setting of representations with unipotent cuspidal support, see \cite[\S10]{Lu-gradedII}. 
The standard modules are of finite length and have the property that $X(s,n,\rho)$ is the unique simple quotient of $Y(s,n,\rho)$.

\paragraph{The Borel--Casselman Equivalence}
\nomenclature{$\mathcal H_{\bfI}$}{}
\nomenclature{$\mathrm{Rep}_\bfI(\bfG(k))$}{}
\nomenclature{$m_{\bfI}$}{}
\nomenclature{$R_\bfI(\bfG(k))$}{}
\nomenclature{$R(\mathcal H_\bfI)$}{}
In this section we assume that $\mb G$ is $k$-split.
We fix a maximal $K$-split torus $\bfT_K$, a root basis $\Delta$ for $\Phi$, and a hyperspecial point $x_0\in \mathcal A(\bfT_K,K)$.
Recall the definition of $c_0$ from the end of Section \ref{par:basicnotation2} and let $\bfI$ be the Iwahori subgroup of $\bfG$ corresponding to $c_0$.
Recall the \emph{Iwahori--Hecke algebra} associated to $\mathbf{G}(k)$ and $\bfI$,
$$\mathcal H_{\mathbf{I}}=\{f\in C^\infty_c(\bfG( k))\mid f(i_1gi_2)=f(g),\ i_1,i_2\in \mathbf I(\mf o)\}.$$
Multiplication in $\mathcal H_{\mathbf{I}}$ is given by convolution with respect to a fixed Haar measure of $\bfG( k)$. Let $\mathrm{Rep}_{\mathbf{I}}(\bfG( k))$ denote the Iwahori category, i.e. the full subcategory of $\mathrm{Rep}(\bfG( k))$ consisting of representations $X$ such that $X$ is generated by $X^{\mathbf I(\mf o)}$. The simple objects in this category are the (irreducible) Iwahori spherical representations. 
By the Borel--Casselman Theorem \cite[Corollary 4.11]{Bo}, there is an exact equivalence of categories
\begin{equation}\label{eq:mI}
    m_{\mathbf{I}}: \mathrm{Rep}_{\mathbf{I}}(\bfG(k))\to \mathrm{Mod}(\mathcal H_{\mathbf{I}}), \qquad m_{\mathbf{I}}(X) = X^{\mathbf I(\mf o)}.
\end{equation}
This equivalence induces a group isomorphism
\begin{equation}\label{eq:mIhom}
m_{\mathbf{I}}: R_{\mathbf{I}}(\mathbf{G}(k)) \xrightarrow{\sim}  R(\mathcal{H}_{\mathbf{I}}),
\end{equation}
where $R_{\mathbf{I}}(\mathbf{G}(k))$ (resp. $R(\mathcal{H}_{\mathbf{I}})$) is the Grothendieck group of $\mathrm{Rep}_{\mathbf{I}}(\mathbf{G}(k))$ (resp. $\mathrm{Mod}(\mathcal{H}_{\mathbf{I}})$). 

\paragraph{Deformations of Modules of the Iwahori-Hecke Algebra}
\nomenclature{$\mathcal H_c$}{}
\nomenclature{$\mf J_c$}{}
\nomenclature{$X|_{W_c}$}{}
Suppose $\mathbf{P}_c$ is a parahoric subgroup containing $\mathbf{I}$ with pro-unipotent radical $\mathbf{U}_c$ and reductive quotient $\mathbf{L}_c$.
The finite Hecke algebra $\mathcal H_{c}$ of $\bfL_c(\mathbb F_q)$ embeds as a subalgebra of $\mathcal H_{\bfI}$.
For $X\in \mathrm{Rep}_{\bfI}(\bfG(\sfk))$ admissible, the Moy--Prasad theory of unrefined minimal $K$-types \cite{moyprasad} implies that the finite dimensional $\bfL_c(\mathbb F_q)$-representation $X^{\bfU_{c}(\mf o)}$ is a sum of principal series unipotent representations and so corresponds to an $\mathcal H_{c}$-module with underlying vector space 
$$(X^{\bfU_c(\mf o)})^{\bfI(\mf o)/\bfU_c(\mf o)} = X^{\bfI(\mf o)}.$$
The $\mathcal H_c$-module structure obtained in this manner coincides naturally with that of
$$\Res_{\mathcal H_c}^{\mathcal H_\bfI}m_\bfI(X).$$
Let 
$$\mathfrak J_c:\mathcal H_c\to \C[W_c]$$
be the isomorphism introduced by Lusztig in \cite[Theorem 3.1]{lusztigdeformation}.
Given any $\mathcal H_c$-module $M$ we can use the isomorphism $\mathfrak J_c$ to obtain a $W_c$-representation which we denote by $M_{q\to1}$.
Define
\begin{equation}
    \label{eq:Wcqto1}
    X|_{W_c}:=(\Res_{\mathcal H_c}^{\mathcal H_\bfI}m_\bfI(X))_{q\to 1}.
\end{equation}

We will need to recall some structural facts about the Iwahori--Hecke algebra. Let $\CX:=X_*(\mathbf T,\bar{\mathsf k})=X^*(\mathbf{T}^\vee,\bar{\sfk})$ and consider the (extended) affine Weyl group $\widetilde{W} := W \ltimes \CX$. Let 
$$S := \{s_{\alpha} \mid \alpha \in \Delta\} \subset W$$
denote the set of simple reflections in $W$. For each $x \in \CX$, write $t_x \in \widetilde{W}$ for the corresponding translation. If $W$ is irreducible, let $\alpha_h$ be the highest root and set $s_0=s_{\alpha_h} t_{-\alpha_h^\vee}$, $S^a=S\cup\{s_0\}$. If $W$ is a product, define $S^a$ by adjoining to $S$ the reflections $s_0$, one for each irreducible factor of $W$. Consider the length function $\ell: \widetilde{W} \to \Z_{\geq 0}$ extending the usual length function on the affine Weyl group $W^a=W\ltimes \mathbb Z \Phi^\vee$
\nomenclature{$l$}{}
\[\ell(w t_x)=\sum_{\substack{\alpha\in \Phi^+\\w(\alpha)\in \Phi^-}} |\langle x,\alpha\rangle+1|+\sum_{\substack{\alpha\in \Phi^+\\w(\alpha)\in \Phi^+}} |\langle x,\alpha\rangle|.
\]
For each $w \in \widetilde{W}$, choose a representative $\bar w$ in the normaliser $N_{\bfG(\mathsf k)}(\mathbf I(\mf o))$. Recall the affine Bruhat decomposition
\[\bfG(\mathsf k)=\bigsqcup_{w\in \widetilde W} \mathbf I(\mf o) \bar w \mathbf I(\mf o), 
\]
For each $w \in \widetilde{W}$, write $T_w \in \mathcal{H}_{\mathbf{I}}$ for the characteristic function of $\mathbf I(\mf o) \bar w \mathbf I(\mf o) \subset \mathbf{G}(\sfk)$. Then $\{T_w\mid w\in \widetilde W\}$ forms a $\mathbb C$-basis for $\mathcal H_{\mathbf{I}}.$

The relations on the basis elements $\{T_w \mid w \in \widetilde{W}\}$ were computed in \cite[Section 3]{iwahorimatsumoto}:
\begin{equation}\label{eq:relations}
    \begin{aligned}
    &T_w\cdot T_{w'}=T_{ww'}, \qquad \text{if }\ell(ww')=\ell(w)+\ell(w'),\\
    &T_s^2=(q-1) T_s+q,\qquad s\in S^a.
    \end{aligned}
\end{equation}

\nomenclature{$R$}{}
\nomenclature{$\mathcal H_{\bfI,v}$}{}
Let $R$ be the ring $\C[v,v^{-1}]$ and for $a\in \C^*$ let $\C_a$ be the $R$-module $R/(v-a)$.
Let $\mathcal H_{\bfI,v}$ denote the Hecke algebra with base ring $R$ instead of $\C$ and where $q$ is replaced with $v^2$ in the relations (\ref{eq:relations}).
By specialising $v$ to $\sqrt{q}$, $1$, we obtain
\begin{equation}
    \mathcal H_{\bfI,v}\otimes_R\C_{\sqrt q} \cong \mathcal H_{\bfI}, \qquad \mathcal H_{\bfI,v}\otimes_R\C_1 \cong \C[ \widetilde W].
\end{equation}
Suppose $Y=Y(s,n,\rho)$ is a standard Iwahori spherical representation, see Section \ref{s:unip-cusp}, and let $M=m_\bfI(Y)$. 
By \cite[Section 5.12]{kl} there is a $\mathcal H_{\bfI,v}$ module $M_v$, free over $R$, such that 
$$M_v\otimes_R\C_{\sqrt q} \cong M$$
as $\mathcal H_\bfI$-modules.
We can thus construct the $\widetilde W$-representation
$$Y_{q\to 1}:= M_v\otimes_R \C_{1}.$$
\nomenclature{$R(\widetilde{W})$}{}
\nomenclature{$(\bullet)_{q\to 1}$}{}
Let $R(\widetilde{W})$ be the Grothendieck group of $\mathrm{Rep}(\widetilde W)$. Since the standard modules form a $\Z$-basis for $R_\bfI(\bfG(\sfk))$, the Grothendieck group of $\mathrm{Rep}_\bfI(\bfG(\sfk))$, there is a unique homomorphism
\begin{equation}\label{eq:qto1hom}(\bullet)_{q\to 1}: R_\bfI(\bfG(\sfk)) \to R(\widetilde W)\end{equation}
extending $Y \mapsto Y_{q \to 1}$. Moreover, since 
$$\Res_{W_c}^{\widetilde W} Y_{q\to 1} = Y|_{W_c}$$
for the Iwahori spherical standard modules we have that
\begin{equation}
    \label{eq:heckerestriction}
    \Res_{W_c}^{\widetilde W}X_{q\to 1} = X|_{W_c}
\end{equation}
for all $X\in R_\bfI(\bfG(\sfk))$.

\paragraph{Reduction to Weyl Groups}
Recall the definition of $\OO^s$ in Section \ref{sec:kawanaka}.
\begin{theorem}
    \label{thm:locwf}
    Suppose $X\in \mathrm{Rep}_\bfI(\bfG(k))$.
    Let $c\subset c_0$.
    Then 
    \begin{equation}
        ^{\barF_q}\WF(X^{\bfU_c(\mf o)})= \max\set{\OO^s(E):[\Res_{W_c}^{\tilde W}(X^{\bfI(\mf o)}):E]>0}
    \end{equation}
    and $^K\WF_c(X)$ is the lift of these orbits.
\end{theorem}
\begin{proof}
    This is just putting together Equation \ref{eq:heckerestriction} and the results in Section \ref{sec:kawanaka}.
\end{proof}

\subsection{The Wavefront Set of Spherical Arthur Representations}
Suppose $\bfG$ is adjoint and split over $k$.
Fix a maximal $K$-split torus $\bfT_K$, a root basis $\Delta$, and a hyperspecial point $x_0\in \mathcal A(\bfT_K,K)$.
Let $c_0$ be the chamber defined as in the end of Section \ref{par:basicnotation2} and $\bfI$ be the corresponding Iwahori-subgroup.
Recall from the end of Section \ref{par:acharduality} the definitions of $d_{A,\bfT_K}^K$ and $d_{\bfT_K}^{\bar k}$.
Let $(\pi,X)$ be the spherical Arthur representation with parameter $\OO^\vee$.
With respect to Lusztig's parameterisation of representations with cuspidal unipotent support, this is the representation with parameter $[(q^{\frac12 h^\vee},0,triv)]$, where $h^\vee$ is a neutral element for an $\lsl_2$-triple for $\OO^\vee$. 
In this section we will prove the following theorem.
\begin{theorem}
    \label{thm:arthurwf}
    Let $(\pi,X)$ be the spherical Arthur representation with parameter $\OO^\vee\in \mathcal N_o^\vee$.
    Then
    \begin{equation}
        ^{K}\WF(X) = d_{A,\bfT_K}^K(\OO^\vee,1),\quad \hphantom{ }^{\bar k}\WF(X) = d_{\bfT_K}^{\bar k}(\OO^\vee).
    \end{equation}
\end{theorem}
\begin{remark}
    By the discussion preceeding Proposition 2.3.5 in \cite{cmo}, the above theorem in fact holds without the restriction that $\bfG$ is adjoint.
\end{remark}
Let $\ms n\in \OO^\vee$.
Our strategy is to apply Theorem \ref{thm:locwf}.
The conditions apply since $X$ is Iwahori spherical (since it is spherical).
The first step is thus to get a grasp on the $\tilde W$ structure of $(X^{\bfI(\mf o)})_{q\to 1}$.
With respect to Lusztig's parameterisation of representations with cuspidal unipotent support, $\ms{AZ}(X)$ corresponds to $[(q^{\frac12 h^\vee},\ms n,triv)]$, where $h^\vee$ is a neutral element for $\OO^\vee$. 
In particular $s_0 = 1$ (as defined in Equation \ref{eq:s0}) and so by Theorem \ref{thm:Langlands} (1), $\ms{AZ}(X)$ is tempered.
Let $Y'$ be the standard module of $\ms{AZ}(X)$ (so that $\ms{AZ}(X)$ is the unique simple quotient of $Y'$).
Then since $\ms{AZ}(X)$ is tempered, by \cite[Theorem 8.2]{kl} we have that $Y' = \ms{AZ}(X)$.
In particular, by \cite[Corollary 8.1]{reeder} we have that 
$$(\ms {AZ}(X)^{\bfI(\mf o)})_{q\to 1}  = (Y')_{q\to 1} = \sgn \otimes H^*(\mathcal B^\vee_{\ms n})^{triv}$$
where $X_*$ acts trivially on the total cohomology space.
It follows that 
$$(X^{\bfI(\mf o)})_{q\to 1} = H^*(\mathcal B^\vee_{\ms n})^{triv}$$
and that the action of $\tilde W$ factors through $\cdot: \tilde W\to W$.
It thus suffices to understand $H^*(\mathcal B^\vee_{\ms n})^{triv}$ as a $W$-module and to do this we will use the theory of perverse sheaves.

We adopt the conventions of \cite{shoji}, except that we will work with perverse sheaves over $\C$ rather than $l$-adic sheaves, and we will ignore shifts since we are exclusively interested in the total cohomology.
For $J\in \mathbf P(\tilde \Delta)$, we will write $c(J)$ for the face of $c_0$ with type equal to $J$, and we write $W_J$ in place of $W_{c(J)}$ for notational convenience.

Let $\pi:\tilde {\mf g^\vee}\to \mf g^\vee$ be the Grothendieck resolution and $\underline \C$ be the constant sheaf on $\tilde {\mf g^\vee}$.
Let $N^\vee$ be the nilpotent cone on of $\mf g^\vee$.
Then $\pi_*\underline \C = \IC(\mf g^\vee,\mathcal L)$ where $\mathcal L$ is the local system on the regular semisimple elements $\mf g^\vee_{reg}$ obtained by pushing forward the constant sheaf along $\pi^{-1}(\mf g^\vee_{reg})\to \mf g^\vee_{reg}$.
This map is an etale covering with Galois group $W^\vee$ and $\mathcal L$ is a $\mb G^\vee$-equivariant sheaf of $W^\vee$ representations.
We can thus decompose it as $\mathcal L = \bigoplus_{E\in \mathrm{Irr}(W^\vee)}E\otimes \mathcal L_E$.
Thus $\pi_*\C = \bigoplus_{E\in \mathrm{Irr}(W^\vee)}E\otimes \IC(\mf g^\vee,\mathcal L_E)$ and we have that $\IC(\mf g^\vee,\mathcal L_E)\mid_{N^\vee} = \IC(\overline {\OO^\vee_E},\mathcal E_E)$ where the map $E\mapsto (\OO^\vee_E,\mathcal E_E)$ is the Springer correspondence \cite{shoji}. 
Thus $\pi_*\underline \C\mid_{N^\vee} = \bigoplus_{E\in \mathrm{Irr}(W^\vee)}\IC(\overline{\OO^\vee_E},\mathcal E_E)$.
\begin{lemma}
    Let $J\in \mathbf P(\tilde\Delta)$ and $E\in \mathrm{Irr}(W_J)$.
    View $W_J$ as a subgroup of $W^\vee$ via 
    \begin{equation}
        \begin{tikzcd}[column sep=small]W_J\arrow[r,"\cdot"] & W \arrow[r,"\sim"] & W^\vee\end{tikzcd}.
    \end{equation}
    The first map is the projection map $\cdot:\tilde W\to W$.
    Define $K_{J,E}$ to be the perverse sheaf
    \nomenclature{$K_{J,E}$}{}
    \begin{equation}
        K_{J,E} := \bigoplus_{E'\in \mathrm{Irr}(W^\vee)}[E':\Ind_{W_J}^{W^\vee}E]\IC(\overline{\OO^\vee_{E'}},\mathcal E_{E'}).
    \end{equation}
    Then the support of $K_{J,E}$ lies in the closure of $d_{S,\bfT_K}(c(J),\OO^s(E))$.
\end{lemma}
\begin{proof}
    By \cite[Proposition 4.3]{achar_aubert}, $[\Ind_{W_J}^{W^\vee}E:E']>0$ implies that $\OO^\vee_{E'} \le \Ind_{\bfl_{c(J)}}^{\mf g^\vee} \OO''$ where $\OO''$ is the Springer support of the unique special representation in the same family as $E$.
    But $\OO'' = d_{LS}(\OO'(E))$ and so in fact 
    $$\OO^\vee_{E'} \le d_{S,x_0,\bfT_K}(c(J),\OO^s(E)) = d_{S,\bfT_K}(c(J),\OO^s(E))$$
    where the last equality is Proposition \ref{prop:ds}.
    Therefore $\supp (K_{J,E})$ is contained in the closure of $d_{S,\bfT_K}(c(J),\OO^s(E))$ as required.
\end{proof}
\begin{theorem}
    \label{thm:wfspringer}
    Let $\ms n\in N^\vee$ and $\OO_{\ms n}^\vee = G^\vee.\ms n$.
    \begin{enumerate}
        \item For all $J\in \mathbf P(\tilde\Delta)$ and $E\in \mathrm{Irr}(W_J)$ such that $E$ is an irreducible constituent of $\Res_{W_J}^{\tilde W}H^*(\mathcal B^\vee_{\ms n})^{triv}$, we have that $\OO_{\ms n}^\vee \le d_{S,\bfT_K}(c(J),\OO^s(E))$.
        \item Let $\tilde \OO\in \tilde\theta_{\bfT_K}^{-1}(d_A(\OO_{\ms n}^\vee,1))$ and $(c(J),\OO')\in I_{o,d,c_0}^K(\tilde \OO)$.
        Let $E$ be the special representation of $W_J$ corresponding to $d_{\bfl_{c(J)}}(\OO')$.
        Then 
        \begin{equation}
            [E:\Res_{W_J}^{\tilde W}H^*(\mathcal B^\vee_{\ms n})^{triv}] = 1.
        \end{equation}
    \end{enumerate}
\end{theorem}
\begin{proof}
    \begin{enumerate}
        \item Let $J\in \mathbf P(\tilde\Delta)$ and $E\in \mathrm{Irr}(W_J)$.
        We have that
        \begin{align}
            [E:\Res_{W_J}^{\tilde W}H^*(\mathcal B^\vee_{\ms n})] &= [E:\Res_{W_J}^{W^\vee}\bigoplus_{E'\in \mathrm{Irr}(W^\vee)}E'\otimes H^*_{\ms n}\IC(\overline{\OO^\vee_{E'}},\mathcal E_{E'})] \\
            &= \sum_{E'\in \mathrm{Irr}(W)}[E':\Ind_{W_J}^{W^\vee}E]\dim H^*_{\ms n}\IC(\overline{\OO^\vee_{E'}},\mathcal E_{E'}) \\
            &= \dim H^*_{\ms n}K_{J,E}.
        \end{align}
        Thus if $[E:\Res_{W_J}^{\tilde W}H^*(\mathcal B^\vee_{\ms n})]>0$ then 
        $$\ms n\in \supp K_{J,E} \subseteq \overline{d_{S,\bfT_K}(c(J),\OO^s(E))}.$$
        But $[E:\Res_{W_J}^{\tilde W}H^*(\mathcal B^\vee_{\ms n})^{triv}]>0$ implies that $[E:\Res_{W_J}^{\tilde W}H^*(\mathcal B^\vee_{\ms n})]>0$ and so indeed $\OO_{\ms n}^\vee\le d_{S,\bfT_K}(c(J),\OO^s(E))$.

        \item Let $\tilde \OO,c,\OO',E$ be as in the statement of the theorem.
        Since $\OO'$ is distinguished, it is special and so $\OO'(E) = \OO'$ and $\tilde \OO = \tilde \OO(c(J),\OO^s(E))$. 
        Let $E' = j_{W_J}^{W^\vee}E$.
        Then $[\Ind_{W_J}^{W^\vee}E:E'] = 1$, $\OO^\vee_{E'} = d_{S,x_0,\bfT_K}(c(J),\OO^s(E)) = \OO_{\ms n}^\vee$, and $\mathcal E_{E'}$ is the trivial local system.
        Since $[E'':\Ind_{W_J}^{W^\vee}E]>0$ implies that $\OO^\vee_{E''} \le d_{S,x_0,\bfT_K}(c(J),\OO^s(E)) = \OO_{\ms n}^\vee$ \cite[Proposition 4.3]{achar_aubert}, we have that 
        \begin{equation}
            H^*(K_{J,E})\mid_{\OO_{\ms n}^\vee} = \bigoplus_{E''\in \mathrm{Irr}(W^\vee):\OO^\vee_{E''} = \OO_{\ms n}^\vee}[E'':\Ind_{W_J}^{W^\vee}E]\mathcal E_{E''}.
        \end{equation}
        By a similar calculation as above we have that $[E:\Res_{W_J}^{\tilde W}H^*(\mathcal B^\vee_{\ms n})^{triv}] = [triv:H^*(K_{J,E})\mid_{\OO^\vee}]$.
        Thus $[E:\Res_{W_J}^{\tilde W}H^*(\mathcal B^\vee_{\ms n})^{triv}] = [E':\Ind_{W_J}^{W^\vee}E] = 1$ as required.
    \end{enumerate}
\end{proof}

We can now give a proof of Theorem \ref{thm:arthurwf}.
\begin{proof}
    Let $J\in \mathbf P(\tilde\Delta)$.
    By Theorem \ref{thm:locwf} we know that
    \begin{equation}
        ^K\WF_{c(J)}(X)=\set{\mathcal L(c,\OO^s(E)):E\in \mathrm{Irr}(W_J),[E:\Res_{W_J}^{\tilde W}V^{\ms I}]>0}.
    \end{equation}
    By part 1 of Theorem \ref{thm:wfspringer}, for any $\tilde \OO\in \hphantom{ }^K\WF_{c(J)}(X)$ we have $d_{S,\bfT_K}(\tilde \OO)\ge \OO^\vee$.
    By Theorem \ref{thm:canoninv}, it follows that $\tilde\theta_{\bfT_K}(\tilde \OO)\le_A d_A(\OO^\vee,1)$.
    But by part 2 of Theorem \ref{thm:wfspringer} we have that for any $\tilde \OO\in \tilde\theta_{\bfT_K}^{-1}(d_A(\OO^\vee,1))$ and $(c(J),\OO')\in I_{c_0}^{o,d}(\tilde \OO)$ we have that $\tilde \OO\in \hphantom{ }^K\WF_{c(J)}(X)$.
    It follows that 
    $$^K\WF(X) = \bar\theta_{\bfT_K}^{-1}(d_A(\OO^\vee,1)) = d_A^K(\OO^\vee,1).$$
    Since $\mathcal N_o(\bar k/K)(\tilde \OO) = d^{\bar k}(\OO^\vee)$ for any $\tilde \OO\in \tilde\theta_{\bfT_K}^{-1}(d_A(\OO^\vee,1))$ it follows that 
    $$^{\bar k}\WF(X) = d^{\bar k}(\OO^\vee).$$
\end{proof}

\bibliographystyle{alpha}
\bibliography{ms}

\printnomenclature
\end{document}